\documentclass[a4paper,11pt]{amsart}

\usepackage{amsmath,amsthm,amssymb}
\usepackage{times}
\usepackage{enumerate}
\usepackage[square,sort,comma,numbers]{natbib}
\usepackage[colorlinks,linkcolor=blue,citecolor=blue,urlcolor=blue]{hyperref}
\usepackage{pstricks}
\usepackage{tikz-cd}
\usepackage{latexsym}
\usepackage{color,graphicx}
\usepackage{float}

\allowdisplaybreaks

\newtheorem{theorem}{Theorem}[section]
\newtheorem{corollary}[theorem]{Corollary}
\newtheorem{lemma}[theorem]{Lemma}
\newtheorem{proposition}[theorem]{Proposition}
\newtheorem{assumption}[theorem]{Assumption}

\theoremstyle{definition}
\newtheorem{definition}[theorem]{Definition}
\newtheorem{remark}[theorem]{Remark}

\DeclareMathOperator{\supp}{supp}
\DeclareMathOperator{\spann}{span}
\DeclareMathOperator{\law}{Law}

\DeclareMathOperator{\Var}{Var}

\numberwithin{equation}{section}

\newcommand{\eps}{\varepsilon}
\newcommand{\R}{\mathbb{R}}
\newcommand{\I}{\mathbb{I}}

\newcommand{\N}{\mathbb{N}}
\newcommand{\p}{\mathbb{P}}
\newcommand{\Cf}{\mathrm{C}}
\newcommand{\Df}{\mathrm{D}}
\newcommand{\B}{\mathcal{B}}
\newcommand{\cdl}{c\`{a}dl\`{a}g }

\newcommand{\F}{\mathcal{F}}

\newcommand{\E}[1]{\ensuremath{\mathbb E \left[ #1 \right]}}

\newcommand{\M}{\mathcal{M}}
\newcommand{\Ito}{It\^o }
\newcommand{\Itos}{It\^o's }
\newcommand{\e}{\ensuremath{\mathbb E}}
\newcommand{\W}{\mathcal{W}}
\newcommand{\cG}{\mathcal{G}}
\newcommand{\m}{\vartheta}

\newcommand{\cM}{\mathcal{M}}
\newcommand{\cH}{\mathcal{H}}
\newcommand{\cN}{\mathcal{N}}
\newcommand{\cP}{\mathcal{P}}
\newcommand{\cS}{\mathcal{S}}
\newcommand{\cV}{\mathcal{V}}
\newcommand{\cL}{\mathcal{L}}
\newcommand{\mui}{\mu_0}

\newcommand{\mcD}{\mathcal D}

\newcommand{\mcM}{\mathcal{M}}

\usepackage[colorinlistoftodos,prependcaption,color=yellow,textsize=tiny]{todonotes}

\usepackage{autonum}
\usepackage[foot]{amsaddr}

\usepackage{autonum}

\title[SPDEs and SGD]{Conservative SPDEs as fluctuating mean field limits of stochastic gradient descent}

\author{Benjamin Gess$^{\dagger\ddagger}$} 
\author{Rishabh S. Gvalani$^{\ddagger}$}
\author{Vitalii Konarovskyi$^{\dagger\S}$}
\address[$\dagger$]{Fakult\"{a}t f\"{u}r Mathematik, Bielefeld Universit\"{a}t, 33615 Bielefeld, Germany}
\address[$\ddagger$]{Max Planck Institute for Mathematics in the Sciences, 04103 Leipzig, Germany}
\address[$\S$]{Institute of Mathematics of NAS of Ukraine, 01024 Kiev, Ukraine}
\email{benjamin.gess@math.uni-bielefeld.de}
\email{rishabh.gvalani@mis.mpg.de}
\email{vitalii.konarovskyi@math.uni-bielefeld.de}
\date{\today}



\begin{document}

\maketitle

\begin{abstract}
  The convergence of stochastic interacting particle systems in the mean-field limit to solutions of conservative stochastic partial differential equations is established, with optimal rate of convergence. As a second main result, a quantitative central limit theorem for such SPDEs is derived, again, with optimal rate of convergence.
  
  The results apply, in particular, to the convergence in the mean-field scaling of stochastic gradient descent dynamics in overparametrized, shallow neural networks to solutions of SPDEs. It is shown that the inclusion of fluctuations in the limiting SPDE improves the rate of convergence, and retains information about the fluctuations of stochastic gradient descent in the continuum limit. 
\end{abstract}

\tableofcontents
\section{Introduction}
\label{sec:introduction}

The analysis of machine learning algorithms is confronted with algorithms in high dimension, with a large number of degrees of freedom (parameters), huge data-sets, and high computing capacities. This motivates the analysis of scaling limits, corresponding to the asymptotic regimes in which these parameters become large, leading to a variety of relative scaling regimes. A particularly relevant one is the so-called overparametrised regime, which corresponds to the case in which the number of parameters $M$ is much larger than the (large) number of training samples (data) $N$. Indeed, a large class of real-world algorithms fall into this class, see, for example \cite{Bi.Ca.Ce.Na2018}. 

The success of such algorithms comes as a statistical surprise. Classical belief and estimates in statistics would suggest that vast overparametrization leads to overfitting \cite{Zh.Be.Ha.Re.Vi2021}, contradicting empirical evidence in machine learning. A systematic explanation of this observation constitutes a key challenge in the scientific understanding of machine learning. For recent progress concerning the related ``double-descent'' phenomenon of the error observed in machine learning we refer to \cite{Belkin:2020,Hastie:2022,Mei:2022,Ro.Va2018a}. A central standing conjecture is that the (stochastic) learning algorithm employed in empirical risk minimization introduces an ``implicit bias'' towards minimizers that generalize well, thereby avoiding those that would lead to overfitting. 

In order to prove or disprove this conjecture, universal models for machine learning are needed as the basis for the analysis of the stochastic dynamics of (stochastic) learning algorithms and their implicit bias. Motivated by this, several scaling limits of stochastic gradient descent dynamics have been analysed in the literature. In particular, the overparametrised regime, with its ``mean-field'' \cite{Ch.Ba2018,Si.Sp2020,Ja.Mo.Mo2020} and ``lazy training'' \cite{Ch.Oy.Ba2020,Ja.Ga.Ho2018,Zh.Ma.Gr2019} scalings has achieved significant attention in recent years. In these works, the joint scaling limit of small learning rate $\alpha\to0$ and overparametrization $M\to\infty$ is considered. Roughly speaking, it is shown, for example in \cite{Me.Mo.Ng2018}, that the empirical distribution $\nu^{M,\alpha}$ of the network parameters following stochastic gradient descent converges to the solution $\mu^{0}$ of a (deterministic) gradient flow in the sense that
\begin{equation}
d(\nu^{M,\alpha},\mu^{0})=O(M^{-\frac{1}{2}})+O(\alpha^{\frac{1}{2}})\quad\text{for }\alpha\to0,\,M\to\infty.\label{eq:montanari}
\end{equation}
This corresponds to a law of large numbers result, since it proves the concentration of the random measures $\nu^{M,\alpha}$ onto a deterministic path. The mean behavior of stochastic gradient descent can then be analysed by considering the limiting dynamics $\mu^{0}$. However, after passing to $\mu^{0}$ all of the information about the inherent fluctuations of stochastic gradient descent is lost. Since there is substantial empirical evidence that stochasticity is decisive for the implicit bias of stochastic gradient descent \cite{Wu.Ma.E2018,Al.Do.Ti2020,Ke.Mu.No.Sm.Ta2017}, universal limiting models incorporating these fluctuations are needed. In this work, we rigorously identify a class of nonlinear conservative SPDEs which serve as such a fluctuating continuum model. 

All previous known results rely on considering the \textit{joint} scaling limit $M\to\infty$, $\alpha\to 0$. In contrast, in practice, the sizes of networks $M$ are typically large, while the learning rate $\alpha$ is moderately small (e.g. \cite{Wu.Zo.Br.Gu2021}). This corresponds to the scaling limit $M\to\infty$ with $\alpha>0$ small but fixed. The identification of a scaling limit in this regime is demanding, since it informally corresponds to the solutions of a nonlinear SPDE with challenging well-posedness properties. The development of a corresponding well-posedness framework and the rigorous treatment of the scaling regime $M\to\infty$, $\alpha>0$ are two of the main contributions of the present work.

We now give a more precise account of the setup and results of the present work. Supervised learning starts from a given training set of data\footnote{For simplicity we assume that the ground-truth is given by a function $f$.} $\mcD\subseteq\R^{n_{0}}\times\R^{k_0}$ with inputs $\Theta=\{\theta:(\theta,f(\theta))\in\mcD\}$ and labels $\{f(\theta):(\theta,f(\theta))\in\mcD\}$. One then chooses a space of hypotheses. Here, we consider a fully-connected feed-forward network with one hidden layer 
\begin{equation}
  f^{M}(x,\theta)
  =\frac{1}{M}\sum_{i=1}^{M}c_{i}\phi(U_i \theta+b_{i})
  =\frac{1}{M}\sum_{i=1}^{M}\Phi(x_i,\theta)\label{eq:ML-shallow}
\end{equation}
with weights/parameters $x_i:=(c_i,U_i,b_i)\in\R^{k_0}\times\R^{d_0\times n_0}\times\R^{d_0}=:\R^d $, $\Phi(x_i,\theta)=c_i\phi (U_i\theta +b_i )$,  and $\phi$ a nonlinear activation function. In practice, $\phi$ is often chosen as the rectified linear unit (ReLU). In fact, the results of this paper apply to more general choices of $\Phi$, and we only restrict for simplicity to the specific choice in this introduction. Thereby, we obtain a parametrization of the space of hypotheses  $\mcM=\{f^M(x,\theta):x\in \R ^{Md}\}.$ The aim of risk minimization then is to select a suitable model $f^M(x,\cdot )\in \mcM$ minimizing the risk
\begin{align}
  L(x) & =\int_\mcD|f(\theta)-f^{M}(x,\theta)|^{2}d\m(\theta), \label{eq:intro_Lemp}
\end{align}
where we concentrate on square loss $l(x,y)= \frac{1}{ 2 }|x-y|^{2}$, and $\m$ is some measure on the data set $\mcD$. For example, $\mcD$ finite with $\m$ the uniform distribution corresponds to empirical risk minimization. This leads to the optimization problem 
$\mathrm{Err}:=^{!}\min_{x}L(x) 
$. 
In machine learning, this optimization is approximated by variants of the  stochastic gradient descent algorithm \cite{[RM51]}, corresponding to a random choice of the direction of descent. More precisely, the optimization dynamics are specified via
\begin{equation}\label{equ_SGD}
  x(n+1)=x(n)- \frac{\alpha}{ P }\sum_{p=1}^{P}\nabla_{x}l(f(x(n),\theta_p),f(\theta_{p})),
\end{equation}
where $\theta_{p}$ are i.i.d.~samples  drawn from $\m$, $\alpha$ is the learning rate, $P$ is the mini-batch size, and $x(0)$ is initialized i.i.d from a measure $\mu_0$. Stochastic gradient descent corresponds to $P=1$ and gradient descent to $P=N$. Note that, since $\e_\m\nabla_{x}l(f(\cdot),f^M(x,\cdot))=\nabla_{x}L(x)$, this is an unbiased estimator of gradient descent. The convergence of these (stochastic) optimization algorithms depends crucially on the properties of the empirical risk $L$. In machine learning, this risk landscape is typically non-convex, non-smooth, and degenerate, making the rigorous analysis of the convergence of stochastic gradient descent challenging. 

In a series of works, see, e.g. \cite{E.Ma.Wu2020,Li.Ta.E2019}, it has been shown that for small learning rate $\alpha>0$, the time discrete dynamics \eqref{equ_SGD} can be approximated up to first order by the following  SDE 
\begin{equation} 
dX_{t}^{M,\alpha}=U(X_{t}^{M,\alpha},\mu_{t}^{M,\alpha})dt+ \sigma^{ \frac{1}{ 2 }}\Sigma^{\frac{1}{2}}(X_t^{M,\alpha}) dB(t), \label{eq:intro_SME}
\end{equation}
where $\mu^{M,\alpha}_t:=\frac{1}{M}\sum_{i}\delta_{(X_{t}^{M,\alpha})_{i}}$ is the empirical measure of the above system, $B$ is a Brownian motion in $\R^{Md}$, $U(x,\mu)=(V(x_i,\mu))_{i\in[M]}$, $\Sigma(x,\mu)=(\tilde A(x_i,x_j,\mu))_{i,j\in[M]}$ for $x=(x_i)_{i\in[M]}\in \R^{Md}$, and
\begin{equation}\label{equ_definition_of_coefficients_intro}
\begin{split}
V(x_i,\mu)&=\nabla F(x_i)-\langle \nabla_{x_i}K(x_i,\cdot),\mu\rangle,\\
G(x_i,\mu,\theta)&=\left( f(\theta)-\int \Phi(y,\theta)\mu(d y) \right)\nabla_{x_i}\Phi(x_i,\theta)\\
&-\e_{\m}\left[ \left( f(\theta )-\int \Phi(y,\theta )\mu(d y) \right)\nabla_{x_i}\Phi(x_i,\theta ) \right],\\
\tilde{A}(x_i,x_j,\mu)&=\e_{\m} \left[G(x_i,\mu,\theta)\otimes G(x_j,\mu,\theta)\right]
\end{split}
\end{equation}
and $\sigma=\frac{ \alpha }{ P }$ is the fluctuation intensity. Notably, compared to plain gradient descent, \eqref{equ_definition_of_coefficients_intro} retains information on the fluctuations in  \eqref{equ_SGD}, and offers a higher order of approximation. 

In the case of shallow networks \eqref{eq:ML-shallow} and square loss (\ref{eq:intro_Lemp}) we observe, following \cite{Ch.Ro.Br.Va2020,Ro.Va2018a,Rotskoff:2022}, that we can represent the square loss as follows
\begin{align}
L(x) & =C_{f}-\frac{1}{M}\sum_{i=1}^{M}F(x_{i})+\frac{1}{2M^{2}}\sum_{i,j=1}^{M}K(x_{i},x_{j})\label{eq:ML-square_loss}
\end{align}
where $C_{f}=\e_{\m}|f(\theta)|^{2}$, $F(x_i) = \mathbb{E}_\m\left[f(\theta)\Phi(x_i,\theta)\right]$, $K(x_i,x_j) =\mathbb{E}_\m\left[ \Phi(x_i,\theta)\Phi(x_j, \theta)\right]$, $i,j\in[M]$. This representation of the loss reveals intricate relations to statistical physics by interpreting the parameters as particles interacting via the interaction potential given by the risk (see ~\cite{Rotskoff:2022}):  The empirical distribution $\mu^{M,\alpha}$ of the parameter dynamics \eqref{eq:intro_SME} can be identified as a solution to the martingale problem
\begin{equation}\label{eq:discrete_SPDE}
    d\mu_{t}^{M,\alpha}=-\nabla\cdot(V(\cdot ,\mu_{t}^{M,\alpha})\mu_{t}^{M,\alpha})dt+\frac{\sigma}{2}D^{2}:(A(\cdot,\mu_{t}^{M,\alpha})\mu_{t}^{M,\alpha})dt+\sigma^{\frac{1}{2}}\nabla\cdot(d\mcM_{t}^{M,\alpha})
\end{equation}
with initial datum $\mu_0^M= \frac{1}{M}\sum_{i=1}^M\delta_{(X_0^{M,\alpha})_i}$ and where $\mcM^{M,\alpha}$ is a continuous martingale satisfying 
$$
[\langle\psi,\mcM_{\cdot }^{M,\alpha}\rangle]_{t}=\int_{0}^{t}\int\int\psi(x)\otimes\psi(y):\tilde{A}(x,y,\mu_{s}^{M,\alpha})\mu_{s}^{M,\alpha}(dx)\mu_{s}^{M,\alpha}(dy)ds
$$
and $A(x,\mu)=\tilde{A}(x,x,\mu)$. Based on this, in several works \cite{Ch.Ro.Br.Va2020,Ro.Va2018a,Rotskoff:2022} it has been informally suggested that in the overparametrised limit ($M\to\infty$) the dynamics will converge to solutions of the martingale problem
\begin{equation}
  d\mu_{t}^{\alpha}=-\nabla\cdot(V(\cdot ,\mu_{t}^{\alpha})\mu_{t}^{\alpha})dt+\frac{\sigma}{2}D^{2}:(A(\cdot ,\mu_{t}^{\alpha})\mu_{t}^{\alpha})dt+\sigma^{\frac{1}{2}}\nabla\cdot(d\mcM_{t}^{\alpha}),\label{eq:intro-martingale}
\end{equation}
where $\mcM$ is a continuous martingale satisfying
\begin{equation}
  [\langle\psi,\mcM_{\cdot }^{\alpha}\rangle]_{t}=\int_{0}^{t}\int\int\psi(x)\otimes\psi(y):\tilde{A}(x,y,\mu_{s}^{\alpha})\mu_{s}^{\alpha}(dx)\mu_{s}^{\alpha}(dy)ds.\label{eq:intro_quadr_variation}
\end{equation}

The proof of this conjecture is one of the main results of this work. Precisely, we prove that in the $M\to\infty$ scaling limit, the empirical measure $\mu^{M,\alpha}$ converges to a solution $\mu^{\alpha}$ of the conservative SPDE~\eqref{eq:intro-SPDE}, with the optimal rate of convergence
\[
d(\mu^{M,\alpha},\mu^{\alpha})=O(M^{-1/2})\quad\text{for }M\to\infty \,.
\]
The rigorous proof relies on the development of a well-posedness framework for (\ref{eq:intro-martingale}), that is, to an infinite dimensional martingale problem with nonlocal coefficients and degenerate ellipticity. We approach this by analyzing instead the probabilistically strong well-posedness of SPDEs that have (\ref{eq:intro-martingale}) as their martingale problem. However, the naive guess for such an SPDE, corresponding to informally taking the square-root of the coefficients in the quadratic variation (\ref{eq:intro_quadr_variation}), leads to an SPDE with irregular diffusion coefficients, for which the validity of strong uniqueness appears unclear. Instead, as one of the first main ideas of this work, we introduce an alternative ``coupling'' SPDE associated to (\ref{eq:intro-martingale}), by 
\begin{equation}
d\mu_{t}^{\alpha}=-\nabla\cdot(V(\cdot,\mu_{t}^{\alpha})\mu_{t}^{\alpha})dt+\frac{\sigma}{2}D^{2}:(A(\cdot,\mu_{t}^{\alpha})\mu_{t}^{\alpha})dt+\sigma^{\frac{1}{2}}\nabla\cdot\int_{\Theta}G(\cdot,\mu_t^\alpha,\theta)\mu_{t}^{\alpha}\,W(d\theta,dt),\label{eq:intro-SPDE}
\end{equation}
which we will call the {\it Stochastic Mean-Field Equation}, where $W$ is a cylindrical Wiener process in $L_{2}(\Theta,\m)$. We then prove the (probabilistically strong) well-posedness of this SPDE. Since (\ref{eq:intro-SPDE}) is an SPDE with nonlocal coefficients, and degenerate coercivity, this is a challenging task. We next state a paraphrased version of our main results on the well-posedness of~\eqref{eq:intro-SPDE}. We refer the reader to Theorems~\ref{the_uniqueness_for_atomic_initial_conditions},~\ref{the_uniqueness_for_initial_condition_with_l2_density}, and~\ref{the_well_posedness_of_dk_equation_in_general_case} for the precise results under varying assumptions on the coefficients and the initial data.

\begin{theorem}[See Theorems~\ref{the_uniqueness_for_atomic_initial_conditions},~\ref{the_uniqueness_for_initial_condition_with_l2_density}, and~\ref{the_well_posedness_of_dk_equation_in_general_case}] Given sufficiently nice initial data $\mu_0 \in \mathcal{P}_2(\mathbb{R}^d)$ and sufficiently nice coefficients $V,A,G$, the stochastic mean-field equation~\eqref{eq:intro-SPDE} has a unique superposition solution in the sense of Definition~\ref{def_superposition_principle}.
\end{theorem}

The proof of this result relies on establishing a superposition principle for the stochastic mean-field equation~\eqref{eq:intro-SPDE}, i.e. the proof that each solution to \eqref{eq:intro-SPDE} is given as a superposition of solutions to the SDE with interaction
\begin{equation} 
  \label{equ_eq:intro-sde_with_interaction}
  dX(u,t)=V(X(u,t),\mu_{t})\,dt+\int_{\Theta}G(X(u,t),\mu_{t},\theta)W(d\theta,dt),
\end{equation}
where $\mu_{t}=\mu_{0}\circ X(\cdot,t)^{-1}$. While the superposition principle for deterministic PDE has been well-established in a series of ground-breaking works~\cite{Ambrosio:2004,Ambrosio_Grippa:2008,DiPerna:1989,F08,Trevisan:2016}, even in infinite dimensions \cite{Re2020}, the few existing results for the case of SPDEs \cite{Flandoli:2009,CG19} rely on restrictive assumptions on the regularity of the coefficients or initial data, which are not satisfied in the present case. Therefore, a new argument for the case of measure-valued, discrete initial data corresponding to \eqref{eq:discrete_SPDE} and Lipschitz continuous coefficients is developed in the present work, as well as a new proof for measure-valued initial data, relaxing the regularity assumptions on the coefficients from \cite{CG19}.

We then prove the uniqueness of solutions to this Lagrangian system, which by the superposition principle implies the uniqueness of solutions to \eqref{eq:intro-SPDE}. Based on the superposition principle, we next establish the convergence of the empirical measures $\mu^{M,\alpha}$ to the solution $\mu^{\alpha}$ of the stochastic mean-field equation~\eqref{eq:intro-SPDE} with optimal rate $M^{- \frac{1}{ 2 }}$.

\begin{theorem}[See Theorem~\ref{the_continuous_dependents_of_solutions_to_sde_with_interaction}] 
  \label{the_rate_of_convergence_of_empirical_measure}
  Let $\mu^{M,\alpha}$ be the superposition empirical measure associated to the SDE~\eqref{eq:intro_SME} started from independent samples of $\mu_0$. Then, for fixed $\alpha>0$ and as $M\to\infty$, we have the estimate
\begin{equation}
\mathbb{E}\sup_{t \in [0,T]}\mathcal{W}_2^2(\mu_t^{M,\alpha},\mu_t^{\alpha}) \lesssim  \mathbb{E} \mathcal{W}_2^2 (\mu_0^{M},\mu_0)\lesssim  M^{-1} \, ,
\label{eq:lln1}
\end{equation}
where $\mu^\alpha$ is a superposition solution of the stochastic mean-field equation~\eqref{eq:intro-SPDE} with $\sigma =\frac{\alpha}{P}$ and initial datum $\mu_0$, $\W_p$ denotes the $p$-Wasserstein distance, and the implicit constants are independent of $\alpha$.
\end{theorem}

We next analyze the law of large numbers behavior of solutions $\mu^{\alpha}$  to ~\eqref{eq:intro-SPDE} in the limit of small learning rate $\alpha \to 0$, proving an optimal rate of convergence to the deterministic transport equation
\begin{equation}
  d\mu^0_t=-\nabla \cdot \left( V(\cdot,\mu^0_t)\mu^0_t \right)dt \, .
  \label{eq:introtransport}
 \end{equation}
 
 \begin{theorem}[See Theorem~\ref{the_lln}]\label{the_lln_introduction}
If $\mu^\alpha$  is a superposition solution of~\eqref{eq:intro-SPDE} with initial datum $\mu_0$, then in the limit $\alpha \to 0$, we have the estimate
\begin{equation}
\mathbb{E}\sup_{t \in [0,T]}\mathcal{W}_2^2(\mu_t^{\alpha},\mu_t^{0}) \lesssim \alpha,
\label{eq:lln2}
\end{equation}
where $\mu^0$ is the solution of the transport equation~\eqref{eq:introtransport} with initial datum $\mu_0$ and the implicit constant is independent of $\alpha>0$.
\end{theorem}

Combining Theorems~\ref{the_rate_of_convergence_of_empirical_measure} and~\ref{the_lln_introduction}, we conclude
\begin{equation}
\mathbb{E}\sup_{t \in [0,T]}\mathcal{W}_2^2(\mu_t^{M,\alpha},\mu_t^{0}) \lesssim \alpha + M^{-1}, \, 
\label{eq:lln3}
\end{equation}
which implies that the limits $M \to \infty$ and $\alpha\to 0$ of $\mu^{M,\alpha}$ commute. In addition, using the results of the present work the intermediate limits can be characterized. Indeed, since $\mu^{M,\alpha}$ is itself shown to be a superposition solution to the stochastic mean-field equation~\eqref{eq:intro-SPDE},~\eqref{eq:lln2} implies that taking the limit $\alpha \to 0$ for $M>0$ fixed yields a solution $\mu^{M}$ to the transport equation~\eqref{eq:introtransport} with initial datum $\mu_0^M$. Subsequently, considering the limit $M \to \infty$ and applying~\eqref{eq:lln3} implies that $\mu^{M}$ converges to a solution to the transport equation~\eqref{eq:introtransport} with initial datum $\mu_0$. Taking the limits in the opposite order follows in an analogous manner, see Figure~\ref{fig:commute}. 

\begin{figure}
    \centering
   \begin{tikzpicture}
   \draw[->] (0,0) node[left]{$\mu^{M,\alpha}$} -- node[above] {$M \to \infty$} (2.5,0) node[right] {SMFE $\eqref{eq:intro-SPDE}_{\mu_0}$};
   \draw[->] (3.3,-.2) -- node[right] {$\alpha \to 0$} (3.3,-2) node[below] {\hspace{5mm} TE $\eqref{eq:introtransport}_{\mu_0}$};   
    \draw[->] (-.5,-0.2) -- node[left] {$\alpha \to 0$}(-0.5,-2) node[below] {\hspace{5mm}TE $\eqref{eq:introtransport}_{\mu_0^M}$}; 
     \draw[->] (0.7,-2.2) -- node[below]{$M \to \infty$} (2.5,-2.2); 
   \end{tikzpicture}
    \caption{The limits of small learning rate and large parameter size for the empirical measure $\mu^{M,\alpha}$. Here, $\eqref{eq:introtransport}_{\nu}$ (resp. $\eqref{eq:intro-SPDE}_{\nu}$) denotes a solution of~\eqref{eq:introtransport} (resp.~\eqref{eq:intro-SPDE}) with initial datum $\nu$.}
    \label{fig:commute}
\end{figure}

Having established the law of large numbers behavior of $\mu^{M,\alpha}$, we next turn to its asymptotic fluctuations and prove a quantified central limit theorem for the stochastic mean-field equation~\eqref{eq:intro-SPDE}, again providing optimal bounds on the rate of convergence. As discussed in~\cite{Rotskoff:2022}, there are two sources of fluctuations, one due to the sampling from the initial measure $\mu_0$ and one due to the dynamical fluctuations in SGD. We define the corresponding fluctuation field  $$\eta^{M,\alpha}:=\min\{\alpha^{-1/2},M^{1/2}\}(\mu^{M,\alpha}-\mu^0).$$
As is done in~\cite[Section~4.4]{Rotskoff:2022}, \ we will\ focus on \ the\ case \ where \ the dynamical \ fluctuations\ dominate, \ i.e. $\alpha$ decreases slowly than $M^{-1}$. \  Precisely, assuming that $M=M(\alpha)$ with $\lim_{\alpha \to 0} \alpha^{-1} M^{-1}(\alpha)<\infty$, we prove that the fluctuation field $\eta^{\alpha}=\alpha^{-1/2}(\mu^{M(\alpha),\alpha}-\mu^0)$ for $\alpha \to 0$ converges to a solution of the linear SPDE
  \begin{align}
    \label{eq:intro-linSPDE2}
    d\eta_t&= -\nabla \cdot \left( V(\cdot ,\mu_t^{0})\eta_t+ \langle \tilde{V}(x ,\cdot ) , \eta_t \rangle_0 \mu_t^{0}(dx) \right)dt\\
    &-P^{-\frac12} \int_{ \Theta }   \nabla \cdot \left( G(\cdot ,\mu_t^{0},\theta)\mu_t^{0} \right)W(d \theta,dt),
\end{align}
and prove an optimal rate of convergence. This generalizes the result obtained in~\cite[Section~4.4]{Rotskoff:2022}, and proves that~\eqref{eq:intro-SPDE} correctly reproduces the central limit fluctuations of the stochastic gradient descent. The proof of this optimal rate of convergence relies on a careful estimation of the error terms, including a new stopping time argument. 

\begin{theorem}[See Theorem~\ref{the_clt}]\label{the_clt_intro}
Given sufficiently nice initial data $\mu_0 \in \mathcal{P}_2(\mathbb{R}^d)$ and sufficiently nice coefficients $V,A,G$, consider the fluctuation field $\eta^\alpha$ as defined earlier. Then, $\eta^{\alpha}$ converges to a weak solution $\eta$ of~\eqref{eq:intro-linSPDE2} with initial datum $\eta_0=a\tilde\eta_0$,
where $\tilde\eta_0$ is a centred Gaussian random variable with covariance
\[
\mathbb{E} \langle \varphi,\tilde\eta_0 \rangle \langle \psi,\tilde\eta_0\rangle = \int \left(\varphi(x)- \langle\varphi,\mu_0\rangle\right) \left(\psi(x)- \langle\psi, \mu_0\rangle\right) \, d \mu_0(x) \, ,
\]
for any smooth $\varphi,\psi$ and $a=\lim_{\alpha \to 0} \alpha^{-\frac 12} M^{-\frac 12}(\alpha)$.
Furthermore,  $\eta^{\alpha}$ satisfies
\begin{equation}
    \W_2^2(\law(\eta^{\alpha}),\law(\eta) )\lesssim \alpha \, ,
\end{equation}
where the above $\W_2$ distance is defined with respect to an appropriate negative Sobolev norm.
\end{theorem}

\begin{remark} For the case that the fluctuations arising from the sampling from the initial measure dominate, that is, if $M=M(\alpha)$ with $\lim_{\alpha\to 0}\alpha^{-1}M^{-1}(\alpha)=+\infty$, then the same arguments as in the proof of Theorem \ref{the_clt_intro} imply that $\eta^M=M^{1/2}(\mu^{M(\alpha),\alpha}-\mu^0)$ converges to a solution $\eta$ to the linear PDE~\eqref{eq:intro-linSPDE2} with $G=0$ started from the centered Gaussian random variable $\tilde\eta_0$ defined in Theorem~\ref{the_clt_intro} and yield the optimal rate of convergence 
\begin{equation}
    \W_2^2(\law(\eta^{M}),\law(\eta)) \lesssim M^{-1}.
\end{equation}
\end{remark}

We next show that the central limit Theorem \ref{the_clt_intro} in particular implies that the stochastic mean-field equation offers a higher order approximation of the SGD dynamics than the deterministic mean field equation. We define the empirical distribution of SGD by
\begin{equation} 
  \label{equ_empirical_process_for_sgd_intro}
  \nu^{M,\frac{1}{M}}_t=\frac{1}{M}\sum_{i=1}^M\delta_{x_i(\lfloor Mt\rfloor)},\quad t\geq 0,
\end{equation}
where $x(k)=(x_i(k))_{i\in[M]}$, $k\in\N_0$, is defined by~\eqref{equ_SGD} with $\alpha= \frac{1}{M}$ and $P=1$. The central limit theorem obtained in~\cite{Si.Sp2020} gives 
\[
\nu^{M,\frac{1}{M}}=\mu^0+M^{-\frac{1}{2}}\eta+o(M^{- \frac12}),
\]
with $\eta$ as in Theorem~\ref{the_clt_intro}. Moreover, for $\mu^{\frac{1}{M}}$  we have  by Theorem~\ref{the_clt_intro}
\[
  \mu^{\frac{1}{M}}=\mu^0+M^{- \frac{1}{ 2 }}\eta+O(M^{- 1}).
\]
This indicates that the solutions $\mu^{\frac{1}{M}}$ to~\eqref{eq:intro-SPDE} provide a higher order approximation to the SGD dynamics $\nu^{M,\frac{1}{M}}$, in the sense that 
\[
  \nu^{M,\frac{1}{M}}-\mu^{\frac{1}{M}}=o(M^{- \frac12}),
\]
which supersedes the order of approximation by the non-fluctuating limit $d(\nu^{M,\frac{1}{M}},\mu^{0})\approx M^{-1/2}$.

\begin{theorem}(See Theorem~\ref{the_approximation_of_sgd}) 
  \label{the_approximation_intro}
  Let $\mu^{\frac{1}{M}}$ be a superposition solution to the stochastic mean-field equation~\eqref{eq:intro-SPDE} with $\alpha= \frac{1}{ M }$. Let also $\nu^{M,\frac 1M}$, be the empirical process associated to the SGD~\eqref{equ_SGD}, which is defined by~\eqref{equ_empirical_process_for_sgd_intro}. Then,   for every $p \in [1,2)$,
  \[
    \W_p(\law(\mu^{\frac 1M}),\law(\nu^{M,\frac 1M}))=o(M^{- \frac{1}{ 2 }}).
  \]
\end{theorem}

\subsection{Overview of the literature}
To first order, in the small learning rate limit $\alpha \to0$, stochastic gradient descent converges to deterministic gradient descent. As argued above, this law of large numbers scaling limit does not incorporate information on the fluctuations of stochastic gradient descent. However,  considering higher order approximations, stochastic gradient descent can be shown to converge to solutions to so-called stochastic modified equations
\begin{equation}\label{eq:ML-SDE}
dX_{t}=-\nabla\big(L(X_{t})+\frac 14 \alpha|\nabla L(X_{t})|^{2}\big) dt+(\alpha\Sigma(X_{t}))^{\frac{1}{2}}dW_{t}
\end{equation}
with $\Sigma$ given in terms of the variance of the stochastic sampling of the empirical loss, see \cite{E.Ma.Wu2020,Li.Ta.E2017}. For the validity of this limit for moderately large learning rates see \cite{Li.Ma.Ar2021}. A discussion of \eqref{eq:ML-SDE} with jump noise can be found in \cite{Ng.Si.Me.Ga2019}.

The effect of the randomness inherent to stochastic gradient descent on the implicit bias and on implicit regularization has been analyzed in  \cite{Wu.Ma.E2018,Al.Do.Ti2020,Ke.Mu.No.Sm.Ta2017}. 

Overparameterised limits of shallow networks in the mean-field training regime have received considerable attention in recent years. In \cite{Ro.Va2018,Ch.Ba2018,Si.Sp2020,Ja.Mo.Mo2020,Me.Mo.Ng2018} the convergence of gradient descent to a Wasserstein gradient flow has been shown and analysed. Notably, this limit is different from the ``lazy training'' regime which can be treated in terms of a linearisation around initialization, see \cite{Ja.Ga.Ho2018,du2018gradient}. An instructive comparison of the scaling regimes is given in \cite{Ch.Oy.Ba2020} and of their performance in \cite{Gh.Me.Mi.Mo2020,chizat2020implicit,refinetti2021classifying}.

Linear SPDE have been rigorously identified in the context of central limit fluctuations in stochastic gradient descent in \cite{Si.Sp2020b,Ro.Va2018}. A fluctuating, nonlinear mean-field limit, incorporating the fluctuations of stochastic gradient descent was, informally, suggested in \cite{Ch.Ro.Br.Va2020,Ro.Va2018a,Rotskoff:2022}, taking the form of the conservative SPDE (\ref{equ_mean_field_equation_with_corr_noise}) below. The rigorous derivation of this conservative SPDE as well as a proof of a quantified central limit theorem remained open problems in the literature. These are solved in the present work. 

For an overview of the literature on conservative SPDEs we refer to \cite[Section 1.1]{FG21}. We here concentrate on nonlocal conservative SPDEs. In \cite{DV95} nonlocal, nonlinear stochastic Fokker-Planck equations have been considered, proving the uniqueness of solutions by several methods, e.g. by duality arguments, coupling arguments, and the Krylov--Rozovskii variational framework. Under less restrictive assumptions on the coefficients and solutions, this has been extended in  \cite[25, p. 115]{Kurtz:1999}. The case of measure-valued solutions has been treated in \cite{CG19}. Additionally, motivated by applications to fluid dynamics, signed measure-valued solutions to nonlocal, nonlinear stochastic Fokker-Planck equations have been considered in the literature, see, for example, e.g. \cite{RV14,K10,KS12,AX06} and the references therein. 

The convergence of interacting particle systems to solutions of nonlocal, nonlinear Fokker--Planck equations, and the closely related phenomenon of propagation of chaos has been considered in \cite{CG19}, \cite{Kotelenez:1995}  and \cite[Theorem 2.3]{KX01} and the references therein. SDEs with interaction have been analyzed in ~\cite{Dorogovtsev:2007:en,Dorogovtsev:2007:en,Dorogovtsev:2010,Dorogovtsev:2020,Pilipenko:2006,Belozerova:2020,Wang_Feng:2021}, and their relation to  SPDEs and to McKean--Vlasov SDEs with common noise have been considered in \cite{Kotelenez:1995,Dorogovtsev:1997} and \cite{Dorogovtsev:2002,Wang_Feng:2021,Kurtz:1999,Carmona_Delarue:2016}, respectively.

The SPDE considered in this work bears some similarity with the so-called Dean--Kawasaki equation introduced in \cite{Dea1996,Kaw1994}, and which corresponds to \eqref{eq:discrete_SPDE} for independent particles. In contrast to the Dean--Kawasaki equation, the noise caused by SGD is spatially correlated, which allows the development of a full mathematical treatment introduced in the present work. The more singular case of the Dean--Kawasaki equation has attracted considerable interest in the literature, yielding the construction of (renormalized) solutions \cite{vo.St2009,AndvRe2010,KonvRe2019,Konarovskyi:AP:2017}, negative results on the existence of non-trivial solutions \cite{KonLehVon2019,KonLehvRe2020}, and regularized models \cite{CorShaZim2019,CorShaZim2020}.

Central limit theorems for conservative, local SPDEs have been analyzed in \cite{DFG20}. The case of linear transport noise has been analyzed in \cite{Ga.Lu2022}. For central limit theorems for parabolic SPDEs with multiplicative, semilinear noise we refer to \cite{Hu.Nu.Vi.Zh2020,Ch.Kh.Nu.Pu2022,Hu.Li.Wa2020} and the references therein. Higher order approximations of interacting particle systems by conservative, local SPDEs have been shown in \cite{DFG20}, and for non-interacting particle systems up to arbitrary order in \cite{Co.Fi2021}. The authors are not aware of any previous results on central limit theorems for nonlocal conservative SPDEs.

\subsection{Outline of the paper}
In Section~\ref{sec:mean_field_equation_with_correlated_noise}, well-posedness results for the SDE with interaction~\eqref{equ_eq:intro-sde_with_interaction} is shown assuming Lipschitz continuity of its coefficients, and the existence of a superposition solution to the stochastic mean-field equation~\eqref{eq:intro-SPDE} is established. This allows us to connect the uniqueness with the superposition principle in Corollary~\ref{cor_superposition_principle}. The well-posedness of~\eqref{equ_eq:intro-sde_with_interaction} and the continuous dependence of its solutions with respect to the initial particle distribution is obtained in Section~\ref{sub:sde_with_interaction}. Section~\ref{sub:uniqueness_and_superposition_principle} is devoted to the proof of the uniqueness for the stochastic mean-field equation~\eqref{eq:intro-SPDE}. The limit Theorems~\ref{the_lln} and~\ref{the_clt} are proved in Section~\ref{sec:lln_and_clt_for_the_mean_field_equation}. In Section~\ref{sec:mean_field_limit_and_stochastic_gradient_descent}, the higher order approximation of the SGD dynamics by solutions to the stochastic mean-field equation is obtained.

\subsection{Basic notation}
\label{sub:basic_notation}

Let $d \in \N$ be fixed. For $m \in \N_0:=\N \cup \{0\}$ the space of $m$-times continuously differentiable  functions from an open set $\Gamma \subset \R^d $ to $\R $ is denoted by $\Cf^m(\Gamma )$. The subspace of $\Cf^m(\Gamma )$ of all bounded together with their derivatives (resp. compactly supported) functions is denoted by $\Cf_b^m(\Gamma )$ (resp. $\Cf_c^m(\Gamma)$). We write $\Cf(\Gamma)$, $\Cf_b(\Gamma)$ and $\Cf_c(\Gamma)$ for $\Cf^0(\Gamma)$, $\Cf^0_b(\Gamma)$ and $\Cf^0_c(\Gamma)$, respectively.  Let $\varphi, f_i \in \Cf^m(\Gamma) $, $i \in [n]:=\{1,\dots,n\}$. We set $\partial_i \varphi= \frac{\partial }{\partial x_i}\varphi$ and $\partial_{i,j}^2 \varphi=\frac{\partial^2}{\partial x_ix_j}\varphi$. For $f=(f_i)_{i \in [n]}$, we write $\nabla f$ for the matrix with rows $\left( \partial_j f_i \right)_{j \in [d]}$, $i \in [n]$, and $\nabla \cdot f= \sum_{ i=1 }^{ d } \partial_i f_i$ if $n=d$. We also set $D^2 \varphi=(\partial_{i,j} \varphi)_{i,j \in [d]}$. The supremum norm in $\Cf^{m}_b(\Gamma)$ will be denoted by $\|\cdot \|_{\Cf^m_b}$, that is, 
\[
  \|f\|_{\Cf^m_b}= \sum_{ |\alpha|\leq m }  \sup\limits_{ x \in \Gamma  }\left|D^{\alpha}f(x)\right|,
\]
where $D^{\alpha}=\frac{\partial ^{|\alpha|}}{\partial x_1^{\alpha_1}\dots\partial x_d^{\alpha_d}}$ and $|\alpha|=\alpha_1+\dots+\alpha_d$ for $\alpha=(\alpha_1,\dots,\alpha_d) \in \N_0^d$.

For vectors $a,b \in \R^d $ and matrices $A,B \in \R^{d \times d}$ we will use the notation $a \cdot b=\sum_{ i=1 }^{ d } a_ib_i$, $|a|=\sqrt{ a \cdot a }$, $a \otimes b=(a_ib_j)_{i,j \in [d]}$, $A:B=\sum_{ i,j=1 }^{ d } a_{i,j}b_{i,j}$, $A \cdot b=\left( \sum_{ j=1 }^{ d } a_{i,j}b_j \right)_{i \in [d]}$ and $|A|=\sqrt{ A:A }$. In particular, $(\nabla f)\cdot g=\left( \sum_{ j=1 }^{ d } (\partial_j f_i)g_j \right)_{i \in [n]}$, where $f=(f_i)_{i \in [n]}$ and $g=(g_i)_{i \in [d]}$.

For $p\geq 0$ we also introduce the subset $\Cf^2_p(\R^d)$ of all functions $f$ from $\Cf^2(\R^d )$ such that $|f(x)|+(1+|x|) |\nabla f(x)|+ (1+|x|^2) |D^2 f(x)|\leq C(1+|x|^p)$, $x \in \R^d $, for some $C>0$.

For every $m \in \N_0$ and $\delta \in (0,1)$ denote by $\Cf^{m,\delta}(\R^d )$ the subset of all functions from $\Cf^m(\R^d )$ whose $m$-th derivatives are locally $\delta$-H\"older continuous, that is, a function $f \in \Cf^m(\R^d)$ belongs to $\Cf^{m+\delta}(\R^d )$ if for every $R>0$ there exists a constant $C$ such that 
\[
  |D^{\alpha}f(x)-D^{\alpha}f(y)|\leq C|x-y|^{\delta} 
\]
for all $x,y \in \R^d $ with $|x-y|\leq R$ and $\alpha \in \N_0^d$ with $|\alpha|=m$. We equip the space $\Cf^{m,\delta}(\R^d )$ with the Fr\'echet topology generated by the following seminorms
\[
  \|f\|_{m+\delta,K}= \sup\limits_{ x \in K  }\frac{ |f(x)| }{ 1+|x| }+ \sum_{ 1\leq |\alpha|\leq m }   \sup\limits_{ x \in K  }|D^{\alpha}f(x)|+ \sum_{ |\alpha|=m } \sup\limits_{ \substack{x,y \in K \\ x\not= y} }\frac{ |D^{\alpha}f(x)-D^{\alpha}f(y)| }{ |x-y|^{\delta} }
\]
for all compact sets $K \subset \R^d $. Set $\Cf^{m,\delta}_{lb}=\big\{ f \in \Cf^{m,\delta}:\  \|f\|_{m+\delta}:=\|f\|_{m+\delta,\R^d }<\infty\big\}$.

We also denote by $\tilde{\Cf}^{m,\delta}$ the set of all functions $f \in \Cf^m(\R^d \times \R^d )$ whose mixed $m$-th derivatives are locally $\delta$-H\"older continuous, that is, for every  $R>0$ there exists a constant $C>0$ such that
\[
|D^{\alpha}_xD^{\alpha}_yf(x,y)-D^{\alpha}_{x'}D^{\alpha}_yf(x',y)-D^{\alpha}_xD^{\alpha}_{y'}f(x,y')+D^{\alpha}_{x'}D^{\alpha}_{y'}f(x',y')|\leq C|x-x'|^{\delta}|y-y'|^{\delta}
\]
for all $x,y \in \R^d $ with $|x-y|<R$ and $\alpha \in \N_0^d$ with $|\alpha|=m$.  Similarly to $\Cf^{m,\delta}(\R^d )$, the space $\tilde\Cf^{m,\delta}(\R^d )$ will be equipped with the Fr\'echet topology generated by the seminorms 
\begin{align}
  \|f\|_{m+\delta,K}^{\sim}&= \sup\limits_{ x,y \in K  }\frac{ |f(x,y)| }{ (1+|x|)(1+|y|) }+ \sum_{ 1\leq |\alpha|\leq m }   \sup\limits_{ x,y \in K  }|D^{\alpha}_xD^{\alpha}_yf(x,y)|\\
  &+ \sum_{ |\alpha|=m } \sup\limits_{ \substack{x,x',y,y' \in K \\ x\not= x',y\not= y'} }\frac{ |D^{\alpha}_xD^{\alpha}_yf(x,y)-D^{\alpha}_{x'}D^{\alpha}_yf(x',y)-D^{\alpha}_xD^{\alpha}_{y'}f(x,y')+D^{\alpha}_{x'}D^{\alpha}_{y'}f(x',y')| }{ |x-x'|^{\delta}|y-y'|^{\delta} }.
\end{align}
for all compact sets $K \subset \R^d $.  Set $\tilde\Cf^{m,\delta}_{lb}(\R^d )=\left\{ f \in \tilde{\Cf}^{m,\delta}(\R^d ):\ \|f\|_{m+\delta}^{\sim}:=\|f\|_{m+\delta,\R^d }^{\sim}< \infty \right\}$.

Let $L_2(\R^d )$ be the Hilbert space of all 2-integrable functions on $\R^d $ with respect to the Lebesgue measure with the usual $L_2$-norm $\|\cdot \|_{L_2}$ and inner product $\langle \cdot ,\cdot  \rangle_{L_2}$. 

For $J \in \N_0$ and an open domain $\Gamma \subset \R^d $ we denote the complete extension of the space $\Cf_c^\infty(\Gamma )$  with respect to the norm defined by
\[
  \|\varphi\|_{J,\Gamma}^2=\sum_{ |\alpha|\leq J }   \int_{\Gamma}   |D^\alpha \varphi(x)|^2 dx 
\]
by $H^{J}(\Gamma)$. It is well-known that $H^J(\Gamma)$ is a separable Hilbert space with the inner product
\[
  \langle \varphi_1 , \varphi_2 \rangle_{J,\Gamma}=\sum_{ |\alpha|\leq J }   \int_{ \Gamma }  D^{\alpha}\varphi_1(x)D^{\alpha}\varphi_2(x)dx.
\]
The dual space to $H^J(\Gamma)$ equipped with the norm
\[
  \|f\|_{-J,\Gamma}= \sup\limits_{ \varphi \in \Cf_c^{\infty}(\R^d ) } \frac{ \langle \varphi , f \rangle_{0,\Gamma} }{ \|\varphi\|_J }
\]
will be denoted by $H^{-J}(\Gamma)$. It is also a separable Hilbert space with the inner product denoted by $\langle \cdot  , \cdot  \rangle_{-J,\Gamma}$. According to the Riesz representation theorem, there exists the isometry between $H^J(\Gamma)$ and $H^{-J}(\Gamma)$ denoted by $L_{J,\Gamma}$.  We will often drop $\Gamma$ from the notation of the inner product and the norm on a Sobolev space, if it does not lead to the confusion. For more details about the Sobolev spaces $H^{J}(\Gamma)$ and $H^{-J}(\Gamma)$ we refer the reader to, e.g.,~\cite{Adams:1975}. 

For $I=[0,T]$ or $I=[0,\infty)$ the space of all \cdl functions from $I$ to a metric space $E$ equipped with the Skorohod topology will be denoted by $\Df(I,E)$. The set $\Cf(I,E)$ of all continuous functions from $I$ to $E$ is a closed subset of $\Df(I,E)$ and the induced topology on $\Cf(I,E)$ is equivalent to the topology of uniform convergence on compacts.

The space of all probability measures (resp. signed measures with finite total variations) on $(\R^d ,\B(\R^d ))$ equipped with the topology of weak convergence will be denoted by $\cP(\R^d )$ (resp. by $\cM(\R^d )$). Let $\cP_p(\R^d )$ denote the subset of $\cP(\R^d)$ of all probability measures with finite $p$-moment for $p\geq 1$, that is, 
\[
  \cP_p(\R^d )=\left\{ \mu \in \cP(\R^d ):\ \langle \phi_p , \mu \rangle<\infty \right\},
\]
where $\phi_p(x)=|x|^p$, $x \in \R^d $, and $\langle \varphi  , \mu  \rangle$ (and also $\langle \varphi(x) , \mu(dx) \rangle$) is the integration of $\varphi$ with respect to $\mu$. It is well-know that $\cP_p(\R^d )$ is a Polish space with the Wasserstein distance given for each $\mu,\nu \in \cP_p(\R^d )$ by 
\[
  \W_p(\mu,\nu)=\inf\left( \int_{ \R^d  }   \int_{ \R^d  }   |x-y|^p\chi(dx,dy)\right)^{\frac{1}{p}},  
\]
where the infimum is taken over all probability measures on $\R^d \times \R^d $ with marginals $\mu$ and $\nu$.

We will fix a measure space $(\Theta,\cG,\m)$ such that $\m$ is a finite measure and the space $L_2(\Theta,\m):=L_2(\Theta,\cG,\m)$, which consists of all 2-integrable with respect to $\m$ functions (more precisely, equivalence classes) from $\Theta$ to $\R $, is separable. The usual inner product and norm on $L_2(\Theta,\m)$ are denoted by $\langle \cdot  , \cdot \rangle_{\m}$ and $\|\cdot \|_{\m}$, respectively. Let $W_t$, $t\geq 0$, be a cylindrical Wiener process on $L_2(\Theta,\m)$ defined on a complete probability space $(\Omega,\F,\p)$ and $(\F_t)_{t\geq 0}$ be the right-continuous complete extension of the filtration generated by $W$, which exists according to \cite[Lemma~7.8]{Kallenberg:2002}. We recall that for an $(\F_t)$-progressively measurable $L_2(\Theta,\m)$-valued process $g_t$, $t\geq 0$, satisfying 
\[
  \int_{ 0 }^{ t } \int_{ \Theta }   g_t^2(\theta)\m(d \theta)<\infty \quad \mbox{a.s.}  
\]
for every $t>0$, the integral\footnote{For the definition of the stochastic integral with respect to a cylindrical Wiener process see, e.g., \cite[Section~2.2.4]{Gawarecki:2011}. The equality~\eqref{equ_quadratic_variation_of_the_integral} holds due to Theorem~2.4~\cite{Gawarecki:2011} and the fact that $G_tG^{*}_t=\|g\|^2_{\m}$ for $G_t=\langle g_t , \cdot  \rangle_{\m}$.} defined by 
\[
   \int_{ 0 }^{ t } g_s(\theta)W(d \theta,ds):= \int_{ 0 }^{ t } G_s dW_s, \quad t\geq 0,  
\]
is a continuous local $(\F_t)$-martingale with quadratic variation
\begin{equation} 
  \label{equ_quadratic_variation_of_the_integral}
  \left[ \int_{ 0 }^{ \cdot  } g_s(\theta)W(d \theta,ds)  \right]_t=\int_{ 0 }^{ t } \int_{ \Theta }   g_s^2(\theta)\m(d \theta)ds, \quad t\geq 0,  
\end{equation}
where $G_t=\langle g_t , \cdot  \rangle_{\m}$, $t\geq 0$, is an $(\F_t)$-progressively measurable process on the space of Hilbert--Schmidt operators on $L_2(\Theta,\m)$. Denote the space of all Hilbert--Schmidt operators from $L_2(\Theta,\m)$ to a Hilbert space $H$ by $\cL_2(L_2(\Theta,\m);H)$ and $\|\cdot \|_{\mathrm{HS},H}$ be the Hilbert--Schmidt norm on that space. In particular, $\|G_t\|_{\mathrm{HS},\R }=\|g_t\|_{\m}$.

\section{Well-posedness and superposition principle}
\label{sec:mean_field_equation_with_correlated_noise}

In this section, we establish the well-posedness and the superposition principle for the stochastic mean-field equation 
\begin{equation} 
  \label{equ_mean_field_equation_with_corr_noise}
  \begin{split} 
    d\mu_t=  \frac{1}{ 2 }D^2: \left(A(t,\cdot ,\mu_t)\mu_t\right)dt&- \nabla \cdot \left( V(t,\cdot ,\mu_t)\mu_t \right)dt\\
    &- \int_{ \Theta }   \nabla \cdot \left( G(t,\cdot ,\mu_t,\theta)\mu_t \right)W(d \theta,dt), 
  \end{split}
\end{equation}
where $W$ is a cylindrical Wiener process in $L_2(\Theta,\m)$ defined on a complete probability space $(\Omega,\F,\p)$ and the functions $V:[0,\infty)\times \R^d \times \cP_2(\R^d )\times \Omega \to \R^d $ and $G:[0,\infty) \times \R^d \times \cP_2(\R^d )\times \Omega \to (L_2(\Theta,\m))^d$, $A:[0,\infty) \times \R^d \times \cP_2(\R^d ) \times \Omega \to \R ^{d \times d}$ satisfy the following Assumption~\ref{ass_basic_assumption}. 
  \begin{assumption} 
  \label{ass_basic_assumption}
  The functions $V$ and $G$ are $\B([0,\infty))\otimes \B(\R^d )\otimes \B(\cP_2(\R^d ))\otimes\F$-measurable and bounded (in $(t,x,\mu)$) on every compact subset of $[0,\infty)\times \R^d \times  \cP_2(\R^d )$ a.s., the restrictions of $V$ and $G$ to the time interval $[0,t]$ are $\B([0,t])\otimes \B(\R^d )\otimes\B(\cP_2(\R^d ))\otimes \F_t$-measurable, and 
\[
  A(t,x,\mu)= \big( \langle G_i(t,x,\mu,\cdot ) , G_j(t,x,\mu,\cdot ) \rangle_{\m} \big)_{i,j \in [d]}
\]
for all $t\geq 0$, $x \in \R^d$, $\mu \in \cP_2(\R^d )$.
  \end{assumption}

\begin{definition} 
  \label{def_definition_of_solution_to_dk_equation}
  Let $\mui \in \cP_2(\R^d )$. A continuous $(\F_t)$-adapted process $\mu_t$, $t\geq 0$, in $\cP_2(\R^d )$ is a {\it (strong) solution} to the stochastic mean-field equation~\eqref{equ_mean_field_equation_with_corr_noise} started from $\mui$ if for every $\varphi \in \Cf_c^2(\R^d )$ a.s. the equality 
  \begin{equation} 
  \label{equ_integral_equality_in_the_definition_of_solution_to_dke}
    \begin{split}
      \langle \varphi , \mu_t \rangle&= \langle \varphi , \mui \rangle+ \frac{1}{ 2 }\int_{ 0 }^{ t }  \left\langle D^2 \varphi:A(s,\cdot ,\mu_s) , \mu_s \right\rangle ds\\
      &+ \int_{ 0 }^{ t } \left\langle \nabla \varphi \cdot V(s,\cdot ,\mu_s) ,\mu_s  \right\rangle ds + \int_{ 0 }^{ t } \int_{ \Theta}   \left\langle \nabla \varphi \cdot G(s,\cdot ,\mu_s, \theta) , \mu_s \right\rangle W(d \theta,ds)  
    \end{split}
  \end{equation}
  holds for every $t\geq 0$.
\end{definition}
All integrals in the definition above are well-defined due to the a.s. boundedness of the functions inside the integral $\langle \cdot  , \mu_s \rangle$.

\begin{remark} 
  \label{rem_extension_of_definition_to_bounded_functions_phi}
  If for every $T>0$ and compact set $K$ in $\cP_2(\R^d )$ the coefficients $V$ and $G$ are a.s. bounded on $[0,T]\times \R^d \times K$, that is, there exists a (random) constant $C>0$ such that 
  \[
    |V(t,x,\mu)|+ \left\| |G(t,x,\mu,\cdot )| \right\|_{m}\leq C, \quad t \in [0,T],\ x \in \R^d ,\ \mu \in K,
  \]
  and $\mu_t$, $t\geq 0$, is a solution to~\eqref{equ_mean_field_equation_with_corr_noise}, then the integral equality~\eqref{equ_integral_equality_in_the_definition_of_solution_to_dke} in Definition~\ref{def_definition_of_solution_to_dk_equation} holds for every $\varphi \in \Cf_b^2(\R^d )$. This follows from the dominated convergence theorem and the compactness of $\{ \mu_t,\ t \in [0,T] \}$ in $\cP_2(\R^d )$ due to the continuity of $\mu_t$, $t\geq 0$.
\end{remark}
\begin{remark} 
  \label{rem_about_quadratic_variation_of_solution_to_dk_equation}
  Let $\mu_t$, $t\geq 0$, be a solution to the equation~\eqref{equ_mean_field_equation_with_corr_noise}. Then for every $\varphi \in \Cf_c^2(\R^d )$ the process $\langle \varphi , \mu_t \rangle$, $t\geq 0$, is a continuous local $(\F_t)$-semimartingale with quadratic variation 
  \begin{align}
    \left[ \langle \varphi , \mu_{\cdot } \rangle \right]_t&= \int_{ 0 }^{ t } \int_{ \Theta }   \left\langle \nabla \varphi \cdot G(s,\cdot ,\mu_s,\theta) , \mu_s \right\rangle \left\langle \nabla \varphi \cdot G(s,\cdot ,\mu_s,\theta) , \mu_s \right\rangle \m(d \theta)ds\\
    &= \int_{ 0 }^{ t } \int_{ \R^d  }   \int_{ \R^d  }   \left(\nabla \varphi(x) \otimes \nabla \varphi(y)\right): \tilde{A}(s,x,y,\mu_s)\mu_s(dx)\mu_s(dy)ds,  
  \end{align}
  where $\tilde{A}(s,x,y,\mu)=\left( \langle G_i(s,x,\mu,\cdot ) , G_j(s,y,\mu, \cdot ) \rangle_{\m} \right)_{i,j \in [d]}$. The expression for the quadratic variation of $\langle \varphi , \mu_t \rangle$, $t\geq 0$, directly follows from~\eqref{equ_quadratic_variation_of_the_integral}.
\end{remark}

Together with the stochastic mean-field equation~\eqref{equ_mean_field_equation_with_corr_noise}, we will consider the following associated SDE with interaction 
\begin{equation} 
  \label{equ_equation_with_interaction}
  \begin{split} 
    dX(u,t)&= V(t,X(u,t),\bar{\mu}_t)dt+\int_{ \Theta }   G(t,X(u,t),\bar{\mu}_t,\theta) W(d \theta,dt),\\ 
    X(u,0)&= u, \quad \bar{\mu}_t= \mui \circ X^{-1}(\cdot ,t), \quad u \in \R^d, \ \ t\geq 0. 
  \end{split}
\end{equation}
This type of equation was introduced and studied by Dorogovtsev in \cite[Section~2]{Dorogovtsev:2007:en}. We next give the definition of a solution to~\eqref{equ_equation_with_interaction}, following \cite[Definition~2.1.1]{Dorogovtsev:2007:en}.

\begin{definition} 
  \label{def_definition_of_solutions_to_sde_with_interaction}
  A family of continuous processes $\{X(u,t),\ t\geq 0\}$, $u \in \R^d $, is called a {\it (strong) solution} to the SDE with interaction~\eqref{equ_equation_with_interaction} if the restriction of $X$ to the time interval $[0,t]$ is $\B([0,t])\otimes\B(\R^d ) \otimes \F_t$-measurable, $\bar{\mu}_t= \mui \circ X^{-1}(\cdot ,t) \in \cP_2(\R^d )$ a.s. for all $t\geq 0$  and for every $u \in \R^d $ a.s.
  \[
    X(u,t)=u+\int_{ 0 }^{ t } V(s,X(u,s),\bar{\mu}_s)ds+\int_{ 0 }^{ t } \int_{ \Theta }   G(s,X(u,s),\bar{\mu}_s,\theta) W(d \theta,ds)
  \]
  for all $t\geq 0$. 
\end{definition}

We remark that $\bar{\mu}_t$, $t\geq 0$, is an $(\F_t)$-progressively measurable process in $\cP_2(\R^d )$. Moreover, due to Fubini's theorem it does not depend on the version of $X(u ,\cdot )$, $u \in \R^d $, that is, if $Y(u,t)$, satisfies the same measurablity conditions from Definition~\ref{def_definition_of_solutions_to_sde_with_interaction} as $X$ and for every $u \in \R^d $ a.s. $Y(u,\cdot )=X(u, \cdot )$, then $\bar{\mu}_t=\mui\circ Y^{-1}(\cdot ,t)$, $t\geq 0$, a.s.

The key tool in the investigation of solutions to the stochastic mean-field equation~\eqref{equ_mean_field_equation_with_corr_noise} is the fact that it satisfies the (strong) superposition principle.
\begin{definition} 
  \label{def_superposition_principle}
  A continuous process $\mu_t$, $t\geq 0$, in $\cP_2(\R^d )$ started from $\mui \in \cP_2(\R^d )$ is a {\it strong superpositon solution}\footnote{The notion of the strong superposiltion solution considered in this work is close to one introduced by Flandoli (see \cite[Definition~5]{Flandoli:2009}). Since we will only work with equation~\eqref{equ_equation_with_interaction} which has a unique solution, we avoid more general definitions, like in~\cite{Ambrosio:2004,Flandoli:2009,Trevisan:2016}, which needs the introduction of distributions on the path space.} to the stochastic mean-field equation~\eqref{equ_mean_field_equation_with_corr_noise} or {\it satisfies the superposition principle} if there exists a solution $X(u,t)$, $t\geq 0$, $u \in \R^d $, to the SDE with interaction~\eqref{equ_equation_with_interaction} such that $\mu_t=\mui \circ X^{-1}(\cdot ,t)$, $t\geq 0$, a.s. 
\end{definition}

In order to build a strong superposition solution to the equation~\eqref{equ_mean_field_equation_with_corr_noise}, we will need the Lipschitz continuity assumption on the coefficients, which will guarantee the well-posedness of the SDE~\eqref{equ_equation_with_interaction}.

\begin{assumption} 
  \label{ass_lipshitz_continuity}
  The coefficients $V$ and $G$ are Lipschitz continuous with respect to $x$ and $\mu$, that is, for every $T>0$ there exists $L>0$ such that a.s. for every $t \in [0,T]$, $x,y \in \R^d $ and $\mu,\nu \in \cP_2(\R^d )$
  \begin{align}
    |V(t,x,\mu)-V(t,y,\nu)|&+ \left\| |G(t,x,\mu,\cdot )-G(t,y,\nu,\cdot )|\right\|_{\m}  \leq L\left( |x-y|+\W_2(\mu,\nu) \right).
  \end{align}
  and 
  \[
    |V(t,0,\delta_0)|+ \left\||G(t,0,\delta_0,\cdot )|\right\|_{\m}\leq L,
  \]
  where $\delta_0$ denotes the $\delta$-measure at $0$ on $\R^d $.
\end{assumption}

\begin{remark} 
  \label{rem_linear_growth_of_coefficients_under_lipschitz_conditions}
  Assumption 2 implies linear growth of $V$ and $G$, that is, a.s. 
  \[
    |V(t,x,\mu)|+\left\| |G(t,x,\mu,\cdot )|\right\|_{\m}\leq L(1+|x|+\W_2(\mu,\delta_0))
  \]
  for all $x \in \R^d $ and $\mu \in \cP_2(\R^d )$.
\end{remark}

In the next section, using a standard approach, we will prove the well-posedness of the SDE with interaction~\eqref{equ_equation_with_interaction} stated in the following theorem.

\begin{theorem}[Well-posedness of the SDE with interaction] 
  \label{the_well_posedness_of_sde_with_interaction}
  Under Assumptions~\ref{ass_basic_assumption} and~\ref{ass_lipshitz_continuity}, the SDE with interaction~\eqref{equ_equation_with_interaction} has a unique strong solution $X(u,t)$, $t\geq 0$, $u \in \R^d $, for every $\mui \in \cP_2(\R^d )$. Moreover, for every $T>0$ and $p\geq 2$ there exists a constant $C>0$ (only depending of $L,p,d$ and $T$) such that 
  \[
    \e \sup\limits_{ t \in [0,T] }|X(u,t)|^{p}\leq C(1+ \langle \phi_p , \mui \rangle+|u|^p)
  \]
  for all $u \in \R^d $, where $\phi_p(x)=|x|^p$, $x \in \R^d $.
\end{theorem}
\begin{corollary}[Moment preservation property] 
  \label{cor_moment_preserving_property}
  Under Assumptions~\ref{ass_basic_assumption} and~\ref{ass_lipshitz_continuity}, the measure-valued process $\bar\mu_t=\mui\circ X^{-1}(\cdot ,t)$, $t\geq 0$, is moment preserving, that is, for every $T>0$ and $p\geq 2$
  \begin{equation} 
  \label{equ_finiteness_of_m_moments}
  \e {\sup\limits_{ t \in [0,T] }\langle \phi_p , \bar\mu_t \rangle}\leq  C\left( 1+\langle \phi_p , \mu_0 \rangle \right).
  \end{equation}
\end{corollary}

The corollary directly follows from the inequality in Theorem~\ref{the_well_posedness_of_sde_with_interaction} by its integration with respect to $\mui$.

Next, using \Itos formula, one can prove the existence of a superposition solution to~\eqref{equ_mean_field_equation_with_corr_noise}.
\begin{theorem}[Existence of solutions] 
  \label{the_existence_of_solutions_to_dke}
  Let $V,G,A$ satisfy Assumptions~\ref{ass_basic_assumption} and~\ref{ass_lipshitz_continuity}. Then for every $\mui \in \cP_2(\R^d )$ there exists a solution $\mu_t$, $t\geq 0$, to the stochastic mean-field equation~\eqref{equ_mean_field_equation_with_corr_noise} started from $\mui$ that satisfies the superposition principle and is moment preserving.
\end{theorem}

The proof of this theorem directly follows from Lemma~\ref{lem_from_sde_with_interaction_to_dk_equation} below.  We note that the uniqueness of the stochastic mean-field equation is closely related to the superposition principle. Indeed, the well-posedness of the SDE with interaction (see Theorem~\ref{the_well_posedness_of_sde_with_interaction}) and Theorem~\ref{the_existence_of_solutions_to_dke} immediately imply the following corollary.

\begin{corollary} 
  \label{cor_superposition_principle}
  Let $V,G,A$ satisfy Assumptions~\ref{ass_basic_assumption},~\ref{ass_lipshitz_continuity}. The stochastic mean-field equation~\eqref{equ_mean_field_equation_with_corr_noise} has a unique solution if and only if every its solution in the sense of Definition~\ref{def_definition_of_solution_to_dk_equation} is a strong superposition solution in the sense of Definition~\ref{def_superposition_principle}.
\end{corollary}

In Section~\ref{sub:uniqueness_and_superposition_principle}, we will prove the uniqueness to the stochastic mean-field equation making further assumptions on the initial condition and its coefficients.

\subsection{SDE with interaction}
\label{sub:sde_with_interaction}

We start this section with the proof that any solution to the SDE with interaction~\eqref{equ_equation_with_interaction} provides a solution to the equation~\eqref{equ_mean_field_equation_with_corr_noise} in the sense of Definition~\ref{def_definition_of_solution_to_dk_equation}. 

\begin{lemma} 
  \label{lem_from_sde_with_interaction_to_dk_equation}
  Let $V,G,A$ satisfy Assumption~\ref{ass_basic_assumption}, $X(u,t)$, $t\geq 0$, $u \in \R^d $, be a solution to the SDE with interaction~\eqref{equ_equation_with_interaction} with $\mui \in \cP_2(\R^d )$ and $\bar{\mu}_t=\mui\circ X^{-1}(\cdot ,t)$, $t\geq 0$, be continuous in $\cP_2(\R^d )$. Then $\bar\mu_t$, $t\geq 0$, is a solution to the stochastic mean-field equation~\eqref{equ_mean_field_equation_with_corr_noise} started from $\mui$.
\end{lemma}

\begin{proof} 
  We note that for every $u \in \R^d $ the process $X(u,t)$, $t\geq 0$, is a continuous semimartingale in $\R^d $ with quadratic variation
  \[
    [X(u,\cdot )]_t=\int_{ 0 }^{ t } A(s,X(u,s),\bar\mu_s) ds, \quad t\geq 0.
  \]
  Indeed, according to~\eqref{equ_quadratic_variation_of_the_integral} and the polarisation equality, one gets 
  \begin{align}
    [X_i(u,\cdot ),X_j(u, \cdot )]_t&=  \int_{ 0 }^{ t } \langle G_i(s,X(u,s),\bar\mu_s, \cdot ) , G_j(s,X(u,s),\bar\mu_s, \cdot ) \rangle_{\m} ds\\
    &=  \int_{ 0 }^{ t } A_{i,j}(s,X(u,s),\bar\mu_s)ds, \quad t\geq 0,  
  \end{align}
  for all $i,j \in [d]$.
  
  Next, taking $\varphi \in \Cf_c^2(\R^d )$ and applying \Itos formula to $\varphi(X(u,t))$ for every $u \in \R^d $, we get a.s.
  \begin{align}
    \varphi(X(u,t))&= \varphi(u)+ \int_{ 0 }^{ t } \nabla \varphi(X(u,s)) \cdot V(s,X(u,s),\bar\mu_s) ds\\
    &+ \frac{1}{ 2 }\int_{ 0 }^{ t } D^2 \varphi(X(u,s)): A(s,X(u,s),\bar\mu_s) ds\\
    &+ \int_{ 0 }^{ t } \nabla \varphi(X(u,s)) \cdot d N(u,s), \quad t\geq 0.
  \end{align}
  where $N(u,t)=\int_{ 0 }^{ t } \int_{ \Theta }   G(s,X(u,s),\bar\mu_s,\theta)W(d \theta,ds)  $, $t\geq 0$. Using the definition of the stochastic integral with respect to a cylindrical Wiener process, it is easily to see that
  \begin{align}
    \varphi(X(u,t))&= \varphi(u)+ \int_{ 0 }^{ t } \nabla \varphi(X(u,s)) \cdot V(s,X(u,s),\bar\mu_s) ds\\
    &+ \frac{1}{ 2 }\int_{ 0 }^{ t } D^2 \varphi(X(u,s)): A(s,X(u,s),\bar\mu_s) ds\\
    &+ \int_{ 0 }^{ t } \int_{ \Theta } \nabla \varphi(X(u,s)) \cdot G(s,X(u,s),\bar\mu_s,\theta) W(d \theta,ds), \quad t\geq 0.
  \end{align}
  Since $\varphi$ has a compact support and the set $\{ \bar\mu_s,\ s \in [0,t] \}$ is compact a.s. as the image of the compact set $[0,t]$ under the continuous map $s\mapsto\bar\mu_s$, the functions $\nabla \varphi(x) \cdot V(s,x,\bar\mu_s)$, $D^2 \varphi(x):A(s,x,\bar\mu_s)$ and $\|\nabla \varphi(x) \cdot G(s,x,\bar\mu_s,\cdot )\|_{\m}^2$, $s \in [0,t]$, $x \in \R^d $ are bounded a.s., by Assumption~\ref{ass_basic_assumption}. Hence, we may integrate the above expression with respect to $\mui$ and use Fubini's theorem and the equality $\int_{ \R^d  }   \psi(X(u,t))\mui(du)= \langle \psi , \bar{\mu}_t \rangle $ for $\psi \in \Cf_b(\R^d )$ to get that $\bar{\mu}_t$, $t\geq 0$, satisfies~\eqref{equ_integral_equality_in_the_definition_of_solution_to_dke}. This ends the proof of the lemma.
\end{proof}

We next prove Theorem~\ref{the_well_posedness_of_sde_with_interaction}. Since its proof is similar to the proofs of \cite[Theorem~2.2.1]{Dorogovtsev:2007:en}, we will only provide a sketch. 

\begin{proof}[Proof of Theorem~\ref{the_well_posedness_of_sde_with_interaction}] 
  Let $\bar\mu^0_t=\mui$, $t\geq 0$. We define inductively for every $n\geq 1$ the family of continuous processes $\{X_n(u,t),\ t\geq 0\}$, $u \in \R^d $, as solutions to the usual SDEs 
  \begin{equation}\label{equ_equation_for_X_n}
    \begin{split} 
      dX_n(u,t)&= V\left(s,X_n(u,t),\bar\mu_s^{n-1}\right)dt+\int_{ \Theta }   G\left(t,X_n(u,t),\bar\mu_s^{n-1},\theta\right)W(d \theta,dt),\\
      X_n(u,0)&= u
    \end{split}
  \end{equation}
  and
  \begin{equation} 
    \label{equ_equality_for_mu_n}
    \bar\mu_n= \mui\circ X_n^{-1}(\cdot ,t), \quad   t\geq 0.
  \end{equation}
  Using Assumptions~\ref{ass_basic_assumption} and~\ref{ass_lipshitz_continuity}, it is easy to see that SDE~\eqref{equ_equation_for_X_n} has a unique solution for every $u \in \R^d $ and for every $p\geq 2$, $n\geq 1$ and $T>0$ there exists a constant $C>0$, independent of $u$, such that 
  \begin{equation} 
  \label{equ_bound_of_moments_for_x_n}
  \e \sup\limits_{ t \in [0,T] }|X_n(u,t)|^p\leq C(1+|u|^p)
  \end{equation} for all $n\geq 1$ and $u \in \R^d $.
  
  We will first show that $X_n$ satisfies the measurability assumptions from Definition~\ref{def_definition_of_solutions_to_sde_with_interaction}. For this we will prove that $X_n$ has a continuous version in $u$. Taking $u,v \in \R^d $, $T>0$, $p \geq 2$, and using H\"older's inequality and the Burkholder--Davis--Gundy inequality, we estimate  for each $t \in [0,T]$
  \begin{align}
    &\e{\sup\limits_{ s \in [0,t] }\left|X_n(u,s)-X_n(v,s)\right|^{p}}\leq C_1|u-v|^p\\
  &\qquad+ C_1 \e{\int_{ 0 }^{ t } \left| V\left( s,X_n(u,s),\bar\mu_s^{n-1} \right)-V\left( s,X_n(v,s),\bar\mu_s^{n-1} \right) \right|^{p}ds}\\
  &\qquad+C_1 \e{ \int_{ 0 }^{ t }  \left\| |G(s,X_n(u,s),\bar\mu_s^{n-1},\cdot)-G(s,X_n(v,s),\bar\mu_s^{n-1},\cdot)|\right\|^{p}_{\m}ds    },
  \end{align}
  where $C_1$ is a constant that depends only on $p,T$ and $d$. By Assumption~\ref{ass_lipshitz_continuity} and Gronwall's lemma, we get 
  \begin{equation} 
  \label{equ_estimate_for_xut_and_xvt}
    \e{\sup\limits_{ t \in [0,T] }\left|X_n(u,t)-X_n(v,t)\right|^{p}}\leq C|u-v|^{p},
  \end{equation}
  where $C$ also depends only on $p,T,L$ and $d$.  Therefore, we can conclude from \cite[Theorem~3.23]{Kallenberg:2002} that $X_n(u,\cdot )$, $u \in \R^d $, has a continuous version as a $\Cf([0,\infty))$-valued process, which will be also denoted by $X_n$. Note that the choice of different version of $X_n(u,\cdot )$, $u \in \R^d $, does not change the fact that $X_n(u,\cdot )$ solves equation~\eqref{equ_equation_for_X_n} for every $u \in \R^d $. Thus, the desired measurability of $X_n$ follows from the continuity of $(u,t)\mapsto X_n(u,t)$ a.s. We also note that  $\bar\mu^n$, defined by~\eqref{equ_equality_for_mu_n}, is an $(\F_t)$-progressively measurable continuous process in $\cP_2(\R^d )$, where the continuity follows from bound~\eqref{equ_bound_of_moments_for_x_n} and de la Vall\'ee-Poussin \cite[Theorem~1.8]{Liptser:2001}. Note that $\bar\mu_{\cdot }$ does not depend on the choice of a version of $X_n(u,\cdot )$, $u \in \R^d $.

  Next, using Assumption~\ref{ass_lipshitz_continuity} again, we can estimate for every $T>0$, $t \in [0,T]$, $u \in \R^d $ and $n\geq 1$ 
  \begin{align}
    \e{ \sup\limits_{ s \in [0,t] }\left| X_{n+1}(u,s)-X_n(u,s) \right|^{2} }&\leq C \int_{ 0 }^{ t } \e{ \sup\limits_{ r \in [0,s] }\left| X_{n+1}(u,r)-X_n(u,r) \right|^2 }d s\\
    &+ C \int_{ 0 }^{ t } \e{\sup\limits_{ r \in [0,s] }\W_2^2\left( \bar\mu^{n}_r,\bar\mu^{n-1}_r \right)}ds,
  \end{align}
  where $C$ is independent of $u$, $t$ and $n$. By Gronwall's lemma, we have 
  \begin{equation} 
  \label{equ_estimate_of_x_n_via_mu_n}
  \e{ \sup\limits_{ s \in [0,t] }\left| X_{n+1}(u,s)-X_n(u,s) \right|^2 }\leq Ce^{CT} \int_{ 0 }^{ t } \e{\sup\limits_{ r \in [0,s] }\W_2^2\left( \bar\mu^{n}_r,\bar\mu^{n-1}_r \right)}ds.
  \end{equation}
  Using the definition of the Wasserstein distance, we further estimate for $t \in [0,T]$ and $n\geq 2$
  \begin{align}
    \e{\sup\limits_{ s \in [0,t] }\W_2^2\left( \bar\mu^{n}_s,\bar\mu^{n-1}_s \right)}&\leq \e{ \sup\limits_{ s \in [0,t] } \int_{ \R^d  }   |X_{n}(u,s)-X_{n-1}(u,s)|^2\mui(du)  }\\
    &\leq \int_{ \R^d  }   \e{ \sup\limits_{ s \in [0,t] }\left| X_{n}(u,s)-X_{n-1}(u,s) \right|^2 }\mui(du)\\
    &\leq Ce^{CT}\int_{ 0 }^{ t } \e{ \sup\limits_{ r \in [0,s]}\W_2^2\left( \bar\mu_r^{n-1},\bar\mu_r^{n-2} \right) }ds,  
  \end{align}
  where we have used the equality $\mui(\R^d )=1$ and~\eqref{equ_estimate_of_x_n_via_mu_n} in the last step. Iterating the above inequality $n-2$ times, we get 
  \begin{align}
    \e \sup\limits_{ t \in [0,T] }&\W_2^2\left( \bar\mu^n_t,\bar\mu^{n-1}_t \right)\leq C^2e^{2CT}\int_{ 0 }^{ t }\int_{ 0 }^{ s_1 } \e{ \sup\limits_{ r \in [0,s_2]}\W_2^2\left( \bar\mu_r^{n-2},\bar\mu_r^{n-3} \right) }ds_2ds_1\leq \dots\\
    &\leq C^{n-1}e^{(n-1)CT}\int_{ 0 }^{ t }\int_{ 0 }^{ s_1 }\dots \int_{ 0 }^{ s_{n-2} }   \e{ \sup\limits_{ r \in [0,s_{n-1}]}\W_2^2\left( \bar\mu_r^{1},\bar\mu_r^{0} \right) }ds_{n-1}\dots ds_2ds_1\\
    &\leq \frac{C^{n-1}e^{(n-1)CT}T^{n-1} }{(n-1)!}\,\e \sup\limits_{ t \in [0,T] }\W_2^2\left( \bar\mu_t^1,\bar\mu_t^0 \right).
  \end{align}
  The finiteness of the expectation on the right hand side of the above inequality follows from~\eqref{equ_bound_of_moments_for_x_n}.  By~\eqref{equ_estimate_of_x_n_via_mu_n},
  \[
    \e{ \sup\limits_{ t \in [0,T] }\left| X_{n+1}(u,t)-X_n(u,t) \right|^2 }\leq\frac{C^{n}e^{nCT}T^{n} }{(n-1)!}\,\e \sup\limits_{ t \in [0,T] }\W_2^2\left( \bar\mu_t^1,\bar\mu_t^0 \right).
  \]
  Next, using the Borel--Cantelli lemma, it is easily seen that there exist continuous processes $\bar\mu_t$, $t\geq 0$, in $\cP_2(\R^d )$ and $X(u,t)$, $t\geq 0$, in $\R^d $, $u \in \R^d $, such that for every $T>0$ and $u \in \R^d $
  \begin{align}
    & \sup\limits_{ t \in [0,T] }\left| X_n(u,t)-X(u,t) \right|\to 0\quad \mbox{and} \quad  \quad \sup\limits_{ t \in [0,T] }\W_2\left( \bar\mu^{n}_t,\bar\mu_t \right)  \to 0\quad \mbox{a.s.}
  \end{align}
  as $n \to \infty$. Moreover, for every $u \in \R^d $ a.s.
  \[
    X(u,t)=u+\int_{ 0 }^{ t } V(s,X(u,s),\bar{\mu}_s)ds+\int_{ 0 }^{ t } \int_{ \Theta }   G(s,X(u,s),\bar{\mu}_s,\theta) W(d \theta,ds)
  \]
  for all $t\geq 0$, by Assumption~\ref{ass_lipshitz_continuity}.  Since the constant in inequality~\eqref{equ_estimate_for_xut_and_xvt} does not depend on $u,v$ and $n$, the inequality remains true for $X_n$ replaced by $X$, by Fatou's lemma. Therefore, the $\Cf([0,\infty))$-valued random field $X(u,\cdot )$, $u \in \R^d $, has a continuous version, which is also denoted by $X$. This implies that $X$ satisfies the measurability assumptions of Definition~\ref{def_definition_of_solutions_to_sde_with_interaction}. We also remark that for every $\varphi \in \Cf_b(\R^d )$ and $t\geq 0$ one has a.s.
  \[
    \langle \varphi , \bar\mu_t \rangle=\lim_{ n\to\infty }\langle \varphi ,\bar\mu^n_t  \rangle= \lim_{ n\to\infty }\int_{ \R^d  }   \varphi(X_n(u,t))\mui(du)=\int_{ \R^d  }   \varphi(X(u,t))\mui(du),
  \]
  by the dominated convergence theorem. This completes the proof of the well-posedness of SDE~\eqref{equ_equation_with_interaction}.

  We next show the finiteness of moments of the solution $X$. For each $u \in \R^d $ and $n,m \in \N$ we define the stopping times
  \[
    \tau_u^n= \inf\left\{ t\geq 0:\ |X(u,t)|\geq n \right\} \quad \mbox{and}\quad \sigma^n=\inf\left\{ t\geq 0:\ \W_2(\bar{\mu}_t,\delta_0)\geq n \right\}.
  \]
  Let also $\tau^{n,m}_u=\tau_u^n \wedge \sigma^m$.  Then, using H\"older's inequality, the Burkholder--Davis--Gundy inequality (see, e.g., Theorem~3.28~\cite{Karatzas:1991}) and Remark~\ref{rem_linear_growth_of_coefficients_under_lipschitz_conditions}, we estimate for $p\geq 2$ and every $n,m \in \N$, $t \in [0,T]$
  \begin{align}
    &\e \sup\limits_{ s \in [0,t] }|X(u,s\wedge \tau^{n,m}_u)|^p \\\leq& C|u|^p+C\e \sup\limits_{ s \in [0,t] }\left|\int_{ 0 }^{ s\wedge\tau^{n,m}_u} V(r,X(u,r),\bar{\mu}_r)dr \right|^p\\
    &+C\e \sup\limits_{ s \in [0,t] }\left|\int_{ 0 }^{ s \wedge \tau^{n,m}_u} G(r,X(u,s),\bar\mu_r,\theta)W(d \theta,dr) \right|^p\\
    \leq & C|u|^p+C\e \int_{ 0 }^{ t \wedge \tau_u^{n,m}} \left|V(r,X(u,r),\bar\mu_r)\right|^pds\\
    &+ C\e \left( \int_{ 0 }^{ t \wedge \tau_u^{n,m}} \left\|\left|G(s,X(u,s),\bar\mu_s,\cdot )\right|\right\|_{\m}^2ds  \right)^{ \frac{p}{ 2 }}\\
    \leq & C|u|^p+C \e \int_{ 0 }^{ t \wedge \tau_u^{n,m}} \left( 1+|X(u,s)|^p+\W_2^p(\bar\mu_s,\delta_0) \right)ds \\
    \leq & C(1+|u|^p)+C\e \int_{ 0 }^{ t } |X(u,s \wedge \tau_u^{n,m})|^pds\\
    &+ C\e \int_{ 0 }^{ t } \W_2^p(\bar\mu_{s \wedge \tau_u^{n,m}},\delta_0)ds \\
    \leq & C(1+|u|^p)+C \e \int_{ 0 }^{ t } \sup\limits_{ r \in [0,s] }|X(u,r \wedge \tau_u^{n,m})|^pds\\
    &+ C \int_{ 0 }^{ t } \e\W_2^p(\bar\mu_{s \wedge \tau_u^{n,m}},\delta_0)ds,
  \end{align}
  where the constant $C$ depends only on $L,p,d$ and $T$.  By Gronwall's lemma, we get 
  \begin{align}
    \e \sup\limits_{ s \in [0,t] }\left|X(u,s \wedge \tau^{n,m}_u)\right|^p\leq Ce^{CT}\left( 1+|u|^p+ \int_{ 0 }^{ t }\e \W_2^p(\bar\mu_{s \wedge \tau^{n,m}_u},\delta_0)ds  \right).
  \end{align}
  Making $n\to\infty$ and using Fatou's lemma, we obtain 
  \begin{equation} 
  \label{equ_inequality_for_expectation_of_x_p}
    \e \sup\limits_{ s \in [0,t] }\left|X(u,s \wedge \sigma^m)\right|^p\leq Ce^{CT}\left( 1+|u|^p+ \int_{ 0 }^{ t }\e \W_2^p(\bar\mu_{s \wedge \sigma^m},\delta_0) ds \right)
  \end{equation}
  for all $t \in [0,T]$, $u \in \R^d $ and $m\geq 1$. In order to bound the integral in the inequality above, we first estimate
  \begin{align}
    &\e \sup\limits_{ s \in [0,t] } \W_2^p(\bar\mu_{s \wedge \sigma^m},\delta_0)\\&=\e\left(\sup\limits_{ s \in [0,t] }\int_{ \R^d }   |X(u,s \wedge \sigma^m)|^2 \mui(du)\right)^{ \frac{p}{2 }}  \\
    &\leq \e \int_{ \R^d  }   \sup\limits_{ s \in[0,t] }|X(u,s\wedge \sigma^m)|^p \mui(du) \\
    & \leq C e^{CT}\int_{ \R^d  }   \left( 1+|u|^p+ \int_{ 0 }^{ t } \e\W_2^p(\bar\mu_{s \wedge \sigma^m},\delta_0)ds \right)\mui(du)\\
    &= C_1\left( 1+\langle \phi_p , \mui \rangle \right)+ C_1 \int_{ 0 }^{ t } \e\W_2^p(\bar\mu_{s \wedge \sigma^m},\delta_0)ds \\
    &= C_1\left( 1+\langle \phi_p , \mui \rangle \right)+ C_1 \int_{ 0 }^{ t } \e \sup\limits_{ r \in[0,s] }\W_2^p(\bar\mu_{r \wedge \sigma^m},\delta_0)ds,
  \end{align}
  where $C_1=C e^{CT}$.  Hence, Gronwall's lemma yields  
  \[
    \e \sup\limits_{ s \in[0,t] }\W_2^p(\bar\mu_{s \wedge \sigma^m},\delta_0)\leq C_1e^{C_1T}(1+\langle \phi_p , \mui \rangle) 
  \]
  for all $t \in [0,T]$ and $m \geq 1$. Combining the inequality above with~\eqref{equ_inequality_for_expectation_of_x_p}, we obtain 
  \[
    \e \sup\limits_{ s \in [0,t] }|X(u,s \wedge \sigma^m)|^p\leq C\left( 1+\langle \phi_p , \mui \rangle+|u|^p \right),
  \]
  where the constant $C$ depends only on $L,p,d$ and $T$. Now, making $m\to\infty$ and using Fatou's lemma, we complete the proof of the theorem. 
\end{proof}

We will further prove the continuous dependence of $\bar\mu_{\cdot }$ on the initial condition. This is known for SDEs with interaction with non-random coefficients driven by a Brownian sheet (see \cite[Exercise~5.3.1]{Dorogovtsev:2007:en}).

\begin{theorem} 
  \label{the_continuous_dependents_of_solutions_to_sde_with_interaction}
  Let the coefficients of SDE~\eqref{equ_equation_with_interaction} satisfy Assumptions~\ref{ass_basic_assumption},~\ref{ass_lipshitz_continuity} and let $X_i$, $i \in [2]$, be solutions to~\eqref{equ_equation_with_interaction} with $\mui=\mui^i \in \cP_2(\R^d )$. Then, for every $T>0$, there exists a constant $C>0$ such that 
  \begin{equation} 
  \label{equ_estimate_for_x1u_minus_x2v}
    \e \sup\limits_{ t \in [0,T] }|X_1(u,t)-X_2(v,t)|^2\leq C \left(|u-v|^2+\W_2^2(\mui^1,\mui^2)\right), \quad u,v \in \R^d ,
  \end{equation}
  and
  \begin{equation} 
  \label{equ_estimate_for_bar_mu1_bar_mu2}
  \e \sup\limits_{ t \in [0,T] }\W^2_2\left(\bar\mu^1_t,\bar\mu^2_t\right)\leq C\W_2^2\left(\mui^1,\mui^2\right),
  \end{equation}
  where $\bar\mu^i_t=\mui^i\circ X_i(\cdot ,t)$.
\end{theorem}
\begin{proof} 
  Using H\"older's inequality, the Burkholder--Davis--Gundy inequality and Assumption~\ref{ass_lipshitz_continuity}, we estimate for each $t \in [0,T]$ and $u,v \in \R^d $
  \begin{align}
    \e \sup\limits_{ s \in [0,t] }& |X_1(u,s)-X_2(v,s)|^2\leq 3|u-v|^2\\
    &+ 3 t \e \int_{ 0 }^{ t } \left| V(s,X_1(u,s),\bar\mu_s^1)-V(s,X_2(v,s),\bar\mu_s^2) \right|^2 ds\\
  &+3\e \sup\limits_{ s \in [0,t] }\left| \int_{ 0 }^{ s } \int_{ \Theta }   \left(G(r,X_1(u,r),\bar\mu_r^1,\theta)-G(r,X_2(v,r),\bar\mu_r^2,\theta)\right)W(d \theta,dr)   \right|^2\\
  &\leq 3|u-v|^2+C \int_{ 0 }^{ t } \left(\e|X_1(u,s)-X_2(v,s)|^2+\e\W_2^2\left( \bar\mu_s^1,\bar\mu_s^2 \right)\right)ds\\
  &\leq 3|u-v|^2+C \int_{ 0 }^{ t } \e \sup\limits_{ r \in [0,s] }|X_1(u,r)-X_2(v,r)|^2ds +C \int_{ 0 }^{ t } \e\W_2^2\left( \bar\mu_s^1,\bar\mu_s^2 \right)ds,
  \end{align}
  where $C$ is independent of $u,v$ and $t$.  By Gronwall's lemma, we get for each $t \in [0,T]$
  \begin{equation} 
  \label{equ_estimate_for_diff_of_x1_and_x2_}
  \e \sup\limits_{ s \in [0,t] }|X_1(u,s)-X_2(v,s)|^2\leq C|u-v|^2+C \int_{ 0 }^{ t } \e\W_2^2\left( \bar\mu_s^1,\bar\mu_s^2 \right)ds. 
  \end{equation}

  We next take an arbitrary probability measure $\chi$ on $\R^d \times \R^d $ with marginals $\mui^i$, $i \in [2]$. Then the probability measure $\chi_t$ defined by $\chi_s(B)=\chi\left( \{(u,v):\ (X_1(u,s),X_2(v,s)) \in B\}\right) $, $B \in \B(\R^d \times \R^d )$, that is the pushforward of $\chi$ under the map $(u,v)\mapsto (X_1(u,s),X_2(v,s))$, has the marginals $\bar\mu_s^i$, $i \in [2]$, for each $t \in [0,T]$. Thus, by the estimate above, we obtain 
  \begin{align}
    \e \sup\limits_{ s \in [0,t] }\W^2_2(\bar\mu_s^1,\bar\mu_s^2)&\leq \e \sup\limits_{ s \in [0,t] }\int_{ \R^d  }   \int_{ \R^d  }   |x-y|^2\chi_s(dx,dy)\\
    &\leq  \int_{ \R^d  }   \int_{ \R^d  }   \e \sup\limits_{ s \in [0,t] }|X_1(u,s)-X_2(v,s)|^2\chi(du,dv)\\
    &\leq C\int_{ \R^d  }   \int_{ \R^d  }   |u-v|^2\chi(du,dv)+ C\int_{ 0 }^{ t } \e \W_2^2(\bar\mu_s^1,\bar\mu_s^2) ds.
  \end{align}
  Taking infimum over all $\chi$ with marginals $\mui^i$, $i \in [2]$, we get the inequality 
  \begin{align}
    \e \sup\limits_{ s \in [0,t] }\W^2_2(\bar\mu_s^1,\bar\mu_s^2)&\leq  C\W_2^2(\mui^1,\mui^2)+ C\int_{ 0 }^{ t } \e \W_2^2(\bar\mu_s^1,\bar\mu_s^2) ds\\
    &\leq C\W_2^2(\mui^1,\mui^2)+ C\int_{ 0 }^{ t } \e \sup\limits_{ r \in [0,s] } \W_2^2(\bar\mu_r^1,\bar\mu_r^2) ds,
  \end{align}
  where $C$ depends only on $T$, $L$ and $d$.  Next, Gronwall's lemma implies
  \[
    \e \sup\limits_{ t \in [0,T] }\W_2^2(\bar\mu_s^1,\bar\mu_s^2)\leq Ce^{CT}\W_2^2(\mui^1,\mui^2).
  \]
  This completes the proof of~\eqref{equ_estimate_for_bar_mu1_bar_mu2}.

  The estimate~\eqref{equ_estimate_for_x1u_minus_x2v} now directly follows from~\eqref{equ_estimate_for_bar_mu1_bar_mu2} and~\eqref{equ_estimate_for_diff_of_x1_and_x2_}. This concludes the proof of the proposition.
\end{proof}

We will also need to use the fact that two paths of particles described by an SDE with interaction never meet. 

\begin{lemma} 
  \label{lem_two_particles_newer_meet}
  Let $V$ and $G$ satisfy Assumptions~\ref{ass_basic_assumption},~\ref{ass_lipshitz_continuity} and let $X$ be a solution to~\eqref{equ_equation_with_interaction}. Then for every $u,v \in \R^d $, $u \not= v$, a.s. $X(u,t)\not= X(v,t)$ for all $t\geq 0$.
\end{lemma}

\begin{proof} 
  We set $X(u,v,t)=X(u,t)-X(v,t)$, $t\geq 0$, and introduce the following $(\F_t)$-stopping times 
  \[
    \sigma_n=\inf\left\{ t\geq 0:\ |X(u,v,t)|\leq  \frac{1}{ n } \right\}, \quad n\geq 1.
  \]
  Then by \Itos formula and Assumption~\ref{ass_lipshitz_continuity}, we can estimate for every $T>0$ and $t \in [0,T]$
  \[
    \e \frac{1}{ |X(u,v,t\wedge \sigma_n)| }\leq \frac{1}{ |u-v| }+C \int_{ 0 }^{ t } \e \frac{1}{ |X(u,v,s\wedge \sigma_n)| }ds ,
  \]
  for some constant $C$ independent of $t$ and $n$. By Gronwall's lemma, $\e \frac{1}{ |X(u,v,t\wedge \sigma_n)| }< \frac{e^{CT}}{ |u-v| }$. Passing to the limit as $n\to\infty$ and using Fatou's lemma, one gets $\e \frac{1}{ |X(u,v,t\wedge \sigma)| }< \frac{e^{CT}}{ |u-v| }$, $t \in [0,T]$, where $\sigma$ is defined similarly to $\sigma_n$ with $ \frac{1}{ n }$ replaced by $0$. Therefore, $X(u,v,t\wedge \sigma)>0$ a.s. for all $t\geq 0$. This implies $t< \sigma$ a.s. for all $t\geq 0$. Hence, $\sigma= +\infty$ a.s. which completes the proof of the lemma.
\end{proof}

\subsection{Uniqueness and superposition principle}
\label{sub:uniqueness_and_superposition_principle}

The main goal of this section is to prove the uniqueness of solutions to the stochastic mean-field equation. We will consider separately a few types of initial particle distributions: atomic, with $L_2$-density, and with a finite second moment. Depending on the type of initial conditions, we will need additional assumptions on the coefficients of the equation, which will appear in corresponding sections.

\subsubsection{Atomic initial conditions}
\label{ssub:atomic_initial_conditions}

We remind the reader that $\phi_m(x)=|x|^m$, $x \in \R^d $, and set 
\begin{equation} 
  \label{equ_definition_of_hn}
  \cN_n=\left\{ \mu \in \cP_2(\R^d ):\ \mu=\sum_{ l=1 }^{ n } \alpha_l \delta_{x^l}\ \mbox{for some}\ \alpha_l\geq 0\ \mbox{and}\ x^l \in \R^d  \right\}, \quad n\geq 1.
\end{equation}
It is easily seen that $\cN_n$ is a closed subspace of $\cP_2(\R^d )$ for each $n\geq 1$.
\begin{theorem} 
  \label{the_uniqueness_for_atomic_initial_conditions}
  Let $V,G,A$ satisfy Assumptions~\ref{ass_basic_assumption},~\ref{ass_lipshitz_continuity}, and let $\mui \in \cN_n$ for some $n\geq 1$. Then there exists a unique solution $\mu_t$, $t\geq 0$, to~\eqref{equ_mean_field_equation_with_corr_noise} started from $\mui$  and satisfying
  \begin{equation} 
  \label{equ_boundedness_of_pth_moment}
    \sup\limits_{ t \in [0,T] }\left\langle \phi_p , \mu_t \right\rangle< \infty \quad \mbox{a.s.}
  \end{equation}
  for $p=n^2+n$ and all $T>0$. Moreover, $\mu_t$, $t\geq 0$, is a superposition solution to~\eqref{equ_mean_field_equation_with_corr_noise}.
\end{theorem}

We define the following function $F_n:\cP_{n(n-1)}(\R^d ) \to [0,\infty)$ by
\begin{equation} 
  \label{equ_function_f}
  F_n(\mu)=\int_{ \R^{dn} }   \prod_{ i,j:i \not= j }   \left|z^i-z^j\right| \prod_{ i=1 }^{ n } \mu(dz^i) =\int_{ \R^{dn} }   \prod_{ i,j:i < j }   \left|z^i-z^j\right|^2 \prod_{ i=1 }^{ n } \mu(dz^i).
\end{equation}
Then
\begin{equation} 
  \label{equ_connection_f_and_ch}
  F^{-1}_n(\{ 0 \})=\cN_{n-1}
\end{equation}
for every $n\geq 2$.

\begin{remark} 
  \label{rem_f_is_a_polynomial}
  The function $F_n$ can be written as a polynomial of the maps 
  \[
    \mu\mapsto\int_{ \R^d  }   z_1^{k_1}\dots z_d^{k_d}\mu(dz)
  \]
  with $k_1+\dots+k_d\leq n(n-1)$, $k_i \in \N \cup \{ 0 \}$, $i \in [d]$.
\end{remark}

To prove Theorem~\ref{the_uniqueness_for_atomic_initial_conditions}, we will apply \Itos formula to the semimartingale $F(\mu_t)$, $t\geq 0$. Hence, we will need an analog of \Itos formula for functions of the form 
\begin{equation} 
  \label{equ_function_g}
  G(\mu)=f(\langle \varphi_1 , \mu \rangle,\dots,\langle \varphi_m , \mu \rangle),\quad \mu \in \cP_2(\R^d ),
\end{equation}
for some smooth functions $f$ and $\varphi_l$, $l \in m$, involving functional derivatives~\cite[Section~2]{Dawson:1993}
\[
  \frac{ \delta G(\mu) }{ \delta\mu }(x)=\frac{\partial }{\partial \eps}G(\mu+\eps \delta_x)|_{\eps=0}
\]
and 
\[
  \frac{ \delta^2 G(\mu) }{ \delta \mu^2 }(x,y)=\frac{\partial^2}{\partial \eps_1 \partial \eps_2}G(\mu+\eps_1\delta_x+\eps_2\delta_y)|_{\eps_1=\eps_2=0}
\]
in order to simplify computations.

We will often work with the linear stochastic mean-field equation obtained by freezing a solution $\mu_t$, $t\geq 0$, in coefficients. In particular, we will write $v(t,x)$, $a(t,x)$ and $g(t,x,\theta)$ for $V(t,x,\mu_t)$, $A(t,x,\mu_t)$ and $G(t,x,\mu_t,\theta)$.

\begin{proposition} 
  \label{pro_ito_formula}
  Let $V,A,G$ satisfy Assumptions~\ref{ass_basic_assumption},~\ref{ass_lipshitz_continuity}, $p\geq 2$ and $\nu_t$, $t\geq 0$, be a continuous $(\F_t)$-adapted process in $\cP_2(\R^d )$ satisfying for every $\varphi \in \Cf_c^2(\R^d )$
  \begin{equation} 
  \label{equ_definition_of_process_nu}
    \begin{split}
      \langle \varphi , \nu_t \rangle&= \langle \varphi , \nu_0 \rangle+ \int_{ 0 }^{ t } \langle \nabla \varphi \cdot  v(s,\cdot ) , \nu_s \rangle ds+ \frac{1}{ 2 }\int_{ 0 }^{ t } \langle D^2 \varphi: a(s,\cdot ) , \nu_s \rangle ds\\
      &+ \int_{ 0 }^{ t } \int_{ \Theta }\langle \nabla \varphi \cdot g(s,\cdot,\theta) , \nu_s \rangle W(d \theta,ds).
    \end{split}
  \end{equation}
  If there exist an $(\F_t)$-stopping time $\sigma$ such that $\sup\limits_{ t \in [0,\sigma] }\langle \phi_p , \nu_t \rangle< \infty$ a.s., then 
  \begin{equation} 
  \label{equ_ito_formula_for_measure_valued_processes}
    \begin{split}
      G(\nu_{t\wedge \sigma})&=G(\nu_0)+ \int_{ 0 }^{ t\wedge \sigma }  \left\langle \nabla \frac{ \delta G(\nu_s) }{\delta\nu_s  }\cdot v(s,\cdot ),\nu_s  \right\rangle ds\\
      &+ \frac{1}{ 2 }\int_{ 0 }^{ t\wedge \sigma }  \left\langle D^2 \frac{ \delta G(\nu_s) }{ \delta\nu_s }:a(s,\cdot ) , \nu_s \right\rangle ds\\
      &+ \frac{1}{ 2 }\int_{ 0 }^{ t\wedge\sigma }  \left\langle \nabla\otimes\nabla \frac{ \delta^2G(\nu_s) }{ \delta\nu_s^2 } : \tilde{a}(s,\cdot ) , \nu_s\otimes\nu_s \right\rangle ds\\
      &+ \int_{ 0 }^{ t\wedge \sigma } \int_{ \Theta }   \left\langle \nabla \frac{ \delta G(\nu_s) }{ \delta\nu_s } \cdot g(s,\cdot ,\theta) , \nu_s \right\rangle W(d \theta,ds),
    \end{split}
  \end{equation}
  for each function $G$ defined by~\eqref{equ_function_g} with $\varphi_l \in \Cf_p^2(\R^d )$, $l \in [m]$, and $f \in \Cf^2(\R^m )$, where $\tilde{a}_{i,j}(t,x,y)=\langle g_i(t,x,\cdot ) , g_j(t,y,\cdot )  \rangle_{\m}$, $i,j \in [d]$.
\end{proposition}

\begin{proof} 
  We will first show that equality~\eqref{equ_definition_of_process_nu} holds for every $\varphi \in \Cf_p^2(\R^d )$ and $t$ replaced by $t \wedge\sigma$. We fix $g \in C_b^2(\R )$ such that $g(x)=0$, $x\leq 0$, and $g(x)=1$, $x\geq 1$. Set for $n\geq 1$
  \[
    \kappa_n(x):=
    \begin{cases}
      g\left( n-\ln |x| \right), & \mbox{ if } x \not= 0, \\
      1, & \mbox{ if } x=0,
    \end{cases} \quad 
    x \in \R^d .
  \]
  It is easy to see that for each $x \in \R^d $ $\kappa_n(x)\to 1$ as $n\to\infty$ and $\kappa_n$, $n\geq 1$, belongs uniformly to $\Cf_0^2(\R^d )$, that is, there exists a constant $C>0$ such that 
  \[
    |\kappa_n(x)|+(1+|x|) |\nabla \kappa_n(x)|+ (1+|x|^2) |D^2 \kappa_n(x)|\leq C
  \]
  for all $x \in \R^d $ and $n\geq 1$. Let $\varphi \in \Cf^2_p(\R^d )$ and $\varphi_n=\varphi \kappa_n$, $n\geq 1$. Note that $\varphi_n \in \Cf^2_c(\R^d )$ for all $n\geq 1$. Therefore, for every $t\geq 0$ 
  \begin{align}
    \langle \varphi_n , \nu_t \rangle&= \langle \varphi_n , \nu_0 \rangle+ \int_{ 0 }^{ t } \langle \nabla \varphi_n \cdot  v(s,\cdot ) , \nu_s \rangle ds+ \frac{1}{ 2 } \int_{ 0 }^{ t } \langle D^2 \varphi_n: a(s,\cdot ) , \nu_s \rangle ds\\
    &+ \int_{ 0 }^{ t } \int_{ \Theta }\langle \nabla \varphi_n \cdot g(s,\cdot,\theta) , \nu_s \rangle W(d \theta,ds).
  \end{align}

  We will next pass to the limit as $n\to\infty$. A simple computation gives that $\varphi_n(x) \to \varphi(x)$,  $\nabla \varphi_n(x) \to \nabla \varphi(x)$ and $D^2 \varphi_n(x) \to D^2 \varphi(x)$ as $n \to \infty$ for all $x \in \R^d $.  Define for $k\geq 1$
  \[
    \sigma_k=\inf\left\{ t\geq 0:\ \left\langle \phi_p , \nu_t \right\rangle> k,\ \W_2(\nu_t,\delta_0)>k \right\}\wedge \sigma.
  \]
  Since the set $\{ \mu \in \cP_2(\R^d ):\ \left\langle \phi_p , \mu \right\rangle\leq k \}$ is closed in $\cP_2(\R^d )$ and $\nu_t$, $t\geq 0$, is a continuous process, $\sigma_k$ is an $(\F_t)$-stopping time, by \cite[Proposition~2.1.5~(a)]{Ethier:1986}. Let $k\geq 1$ and $t\geq 0$ be fixed.  Since $\sup\limits_{ s \in [0,\sigma_k] }\langle \phi_p , \nu_s \rangle\leq k$ and $\varphi_n(x)\leq C(1+|x|^p)$, $x \in \R^d $, $n\geq 1$, for some $C>0$, the dominated convergence theorem implies that 
  \[
    \langle \varphi_n , \nu_{t\wedge \sigma_k} \rangle \to \langle \varphi , \nu_{t\wedge \sigma_k} \rangle \quad \mbox{a.s.}
  \]
  as $n\to\infty$.  By Remark~\ref{rem_linear_growth_of_coefficients_under_lipschitz_conditions}, for every $x \in \R^d $ and $s \in [0,t\wedge \sigma_k]$ one has
  \begin{equation} 
  \label{equ_estimate_of_nabla_phi_v}
    \begin{split}
      \left| \nabla \varphi_n \cdot v(s,x)  \right| &\leq \left(|\nabla \varphi(x)| |\kappa_n(x)| +|\varphi(x)| |\nabla\kappa_n(x)|\right)|V(s,x,\mu_s)|\\
      &\leq L\left(|\nabla \varphi(x)| |\kappa_n(x)|+|\varphi(x)| |\nabla \kappa_n(x)|\right)\left( 1+|x|+\W_2(\mu_s,\delta_0) \right)\\
      &\leq L\left(|\nabla \varphi(x)| |\kappa_n(x)|+|\varphi(x)| |\nabla \kappa_n(x)|\right)\left( 1+|x|+k \right)\\
      &\leq C_1(1+|x|^p),
    \end{split}
  \end{equation}
  where $C_1$ is a constant which is independent of $s$ and $x$. Since $\int_{ 0 }^{ t\wedge \sigma_k } C_1(1+\langle \phi_p , \nu_s \rangle)ds\leq  C_1(1+k)t <\infty$, the dominated convergence theorem implies 
  \[
    \int_{ 0 }^{ t\wedge \sigma_k } \left\langle \nabla \varphi_n \cdot v(s,\cdot ) , \nu_s \right\rangle ds \to \int_{ 0 }^{ t\wedge \sigma_k } \left\langle \nabla \varphi \cdot v(s,\cdot ) , \nu_s \right\rangle ds  \quad \mbox{a.s.}
  \]
  as $n\to\infty$.  Similarly, by Remark~\ref{rem_linear_growth_of_coefficients_under_lipschitz_conditions}, for every $s \in [0,t\wedge \sigma_k]$
  \begin{align}
    & |D^2 \varphi_n:a(s,x)|\\ \leq & |\kappa_n(x)D^2 \varphi(x):a(s,x)|+\left|\left(\nabla \varphi(x)\otimes \nabla \kappa_n(x)\right):a(s,x)\right|\\
    &+\left| \left( \nabla \kappa_n(x)\otimes \nabla \varphi(x) \right):a(s,x) \right|+\left| \varphi(x)D^2 \kappa_n(x) :a(s,x) \right|\\
     \leq & \left( |\kappa_n(x)| |D^2 \varphi(x)|+2 |\nabla \varphi(x)| |\nabla \kappa_n(x)|+|\varphi(x)| |D^2 \kappa_n(x)| \right) |a(s,x)| \\
     \leq &  \big( |\kappa_n(x)| |D^2 \varphi(x)|+2 |\nabla \varphi(x)| |\nabla \kappa_n(x)|\\
    &+C|\varphi(x)| |D^2 \kappa_n(x)| \big) \left( 1+|x|+\W_2(\mu_s,\delta_0) \right)^2\leq C_1 (1+|x|^p),
  \end{align}
  where $C_1$ is also a constant independent of $x$ and $s$. Therefore, using the dominated convergence theorem again, we get 
  \[
    \int_{ 0 }^{ t\wedge \sigma_k } \left\langle D^2 \varphi_n:a(x,\cdot ) ,\nu_s  \right\rangle ds \to \int_{ 0 }^{ t\wedge \sigma_k} \left\langle D^2 \varphi:a(x,\cdot ) , \nu_s  \right\rangle  ds \quad \mbox{a.s.} 
  \]
  as $n\to\infty$.  It remains only to show the convergence of the stochastic integrals.  We consider
  \begin{align}
    &\e \left( \int_{ 0 }^{ t\wedge \sigma_k  } \left\langle \left( \nabla \varphi -\nabla \varphi_n \right)\cdot g(s,\cdot ,\theta) ,\nu_s  \right\rangle W(d \theta,ds) \right)^2\\
    &\qquad\qquad=\e \int_{ 0 }^{ t \wedge \sigma_k } \int_{ \Theta }   \left\langle \left( \nabla \varphi-\nabla \varphi_n \right) \cdot g(s, \cdot ,\theta) , \nu_s \right\rangle^2 \m(d \theta)ds\\
    &\qquad\qquad\leq \e \int_{ 0 }^{ t \wedge \sigma_k } \int_{ \Theta }   \left\langle \left |\nabla \varphi-\nabla \varphi_n \right|^2 |g(s, \cdot ,\theta)|^2 , \nu_s \right\rangle \m(d \theta)ds\\
    &\qquad\qquad\leq \e \int_{ 0 }^{ t \wedge \sigma_k }  \left\langle \left |\nabla \varphi-\nabla \varphi_n \right|^2 \left\| |g(s, \cdot ,\theta)|\right\|^2_{\m} , \nu_s \right\rangle ds\to 0
  \end{align}
  as $n\to\infty$. In the last step, we have used the dominated convergence theorem, since $|\nabla \varphi-\nabla \varphi_n|^2 \left\| |g(s,x)|\right\|^2_{\m} \leq C(1+|x|^p)$ on $[0,t\wedge\sigma_k]$ with a (non-random) constant $C$ that does not depend on $s$ and $x$, where the estimate can be obtained in a similar manner to~\eqref{equ_estimate_of_nabla_phi_v}. Consequently, 
  \[
    \int_{ 0 }^{ t \wedge \sigma_k } \int_{ \Theta }   \left\langle \nabla \varphi_n \cdot g(s,\cdot ,\theta) ,\nu_s  \right\rangle W(d \theta,ds) \to   
    \int_{ 0 }^{ t \wedge \sigma_k } \int_{ \Theta }   \left\langle \nabla \varphi \cdot g(s,\cdot ,\theta) ,\nu_s  \right\rangle W(d \theta,ds)
  \]
  in $L_2(\Omega)$ as $n\to\infty$.
  
  Summarizing obtained convergence results, we can conclude that for each $t\geq 0$ and $k\geq 0$ we have a.s.
  \begin{align}
    \langle \varphi , \nu_{t\wedge \sigma_k} \rangle&= \langle \varphi , \nu_0 \rangle+ \int_{ 0 }^{ t\wedge\sigma_k } \langle \nabla \varphi \cdot  v(s,\cdot ) , \nu_s \rangle ds+ \frac{1}{ 2 } \int_{ 0 }^{ t\wedge \sigma_k } \langle D^2 \varphi: a(s,\cdot ) , \nu_s \rangle ds\\
    &+ \int_{ 0 }^{ t\wedge \sigma_k } \int_{ \Theta }\langle \nabla \varphi \cdot g(s,\cdot,\theta) , \nu_s \rangle W(d \theta,ds).
  \end{align}
  Now we pass to the limit as $k\to\infty$. Remark that the map $t \mapsto \langle \varphi , \nu_t \rangle$ is not continuous a.s. in the Euclidean topology on $[0,\infty)$ because $\varphi$ is not bounded in general. However, $\sigma_k \to  \sigma$ in the discrete topology as $k\to\infty$, i.e. a.s. there exists $\tilde{k}$ such that $\sigma_k=\sigma$ for all $k\geq \tilde{k}$. This allows to pass to the limit as $k\to\infty$. This gives~\eqref{equ_definition_of_process_nu} with $t$ replaced by $t\wedge \sigma$. We can conclude from this equality that the process $\langle \varphi , \nu_{t\wedge \sigma} \rangle$ has a continuous version and is an $(\F_t)$-semimartingale.

  We next set $\langle \varphi , \nu_t \rangle=\left(\langle \varphi_1 , \nu_t \rangle,\dots,\langle \varphi_m , \nu_t \rangle\right)$. Applying \Itos formula, one obtains
  \begin{align}
    G(\nu_{t\wedge \sigma})&= f\left( \langle \varphi_1 , \nu_{t\wedge \sigma} \rangle,\dots, \langle \varphi_m , \nu_{t\wedge \sigma} \rangle \right)=f\left( \langle \varphi ,\nu_0  \rangle \right)\\
    &= \sum_{ l=1 }^{ m }\int_{ 0 }^{ t\wedge \sigma } \partial_l f(\langle \varphi , \nu_s \rangle)d \langle \varphi_l , \nu_s \rangle\\
    &+ \frac{1}{ 2 }\sum_{ l,k=1 }^{ m } \int_{ 0 }^{ t\wedge \sigma } \partial_l\partial_kf(\langle \varphi , \nu_s \rangle)d \left[ \langle \varphi_l , \nu_{\cdot } \rangle, \langle \varphi_k , \nu_{\cdot } \rangle \right]_s\\
    &= f\left( \langle \varphi ,\nu_0  \rangle \right)+ \sum_{ l=1 }^{ m }\int_{ 0 }^{ t\wedge \sigma } \partial_lf(\langle \varphi , \nu_s \rangle)\langle \nabla \varphi_l \cdot v(s,\cdot ) , \nu_s \rangle ds\\
    &+ \frac{1}{ 2 }\sum_{ l=1 }^{ m } \int_{ 0 }^{ t\wedge \sigma } \partial_lf(\langle \varphi , \nu_s \rangle)\langle D^2 \varphi_l:a(s,\cdot ) , \nu_s \rangle ds \\
    &+ \sum_{ l=1 }^{ m }\int_{ 0 }^{ t\wedge \sigma}\int_{ \Theta}  \partial_lf(\langle \varphi , \nu_s \rangle)  \left\langle \nabla \varphi_l \cdot g(s,\cdot ,\theta) , \nu_s \right\rangle W(d \theta,ds)\\ 
    &+ \frac{1}{ 2 } \sum_{ l,k=1 }^{ m }\int_{ 0 }^{ t\wedge \sigma}\int_{ \Theta }   \partial_l\partial_k f(\langle \varphi , \nu_s \rangle)  \langle \nabla \varphi_l \cdot g_l(s,\cdot ,\theta) ,\nu_s  \rangle \langle \nabla \varphi_k \cdot g_k(s,\cdot ,\theta) ,\nu_s  \rangle \m(d \theta)ds. 
  \end{align}
  Using the equalities 
  \[
    \frac{ \delta G(\mu) }{ \delta\mu }(x)=\sum_{ l=1 }^{ m } \partial_l f(\langle \varphi , \mu \rangle)\varphi_l(x)
  \]
  and 
  \[
    \frac{ \delta^2 G(\mu) }{ \delta\mu^2 }(x,y)=\sum_{ l,k=1 }^{ m } \partial_l\partial_k f(\langle \varphi , \mu \rangle)\varphi_l(x)\varphi_k(y),
  \]
  we get~\eqref{equ_ito_formula_for_measure_valued_processes} that completes the proof of the proposition.
\end{proof}

\begin{remark} 
One can extend the obtained \Itos formula to any, e.g., bounded twice continuously differentiable function $G$ on $\cP_2(\R^d )$, using an approximation analog to Bernstein polynomials similarly as it was done in the proof of \cite[Theorem~2]{KonLehvRe2020}. We do not consider this extension here since the obtained \Ito formula is needed only for the proof of Theorem~\ref{the_uniqueness_for_atomic_initial_conditions}, where it will be applied to the function $F_n$, that is defined by~\eqref{equ_function_f} and satisfies the assumptions of Proposition~\ref{pro_ito_formula} .
\end{remark}

\begin{corollary} 
  \label{cor_testing_of_solutions_to_dk_equation_by_bounded_functions}
  Let $\mu_t$, $t\geq 0$, be a solution to the stochastic mean-field equation~\eqref{equ_mean_field_equation_with_corr_noise} whose coefficients satisfy Assumptions~\ref{ass_basic_assumption} and~\ref{ass_lipshitz_continuity}. Then equality~\eqref{equ_integral_equality_in_the_definition_of_solution_to_dke} holds for every $\varphi \in \Cf^2_b(\R^d )$.
\end{corollary}

\begin{proof} 
  We note that $\mu_t$, $t\geq 0$, is a continuous process in $\cP_2(\R^d )$. Hence, $\sup\limits_{ s \in [0,t] }\langle \phi_2 , \mu_s \rangle< \infty$ a.s. for all $t\geq 0$. Consequently, the corollary directly follows from Proposition~\ref{pro_ito_formula} with $f(x)=x$, $x \in \R $, and the inclusion $\Cf_b^2(\R^d ) \subset \Cf_2^2(\R^d )$.
\end{proof}

The following lemma gives the key property of solutions to the stochastic mean-field equation started from atomic initial condition that allows to prove their uniqueness.

\begin{lemma} 
  \label{lem_mu_stays_in_hn}
  Let $V,A,G$ satisfy Assumptions~\ref{ass_basic_assumption},~\ref{ass_lipshitz_continuity}, $n\geq 2$, and $\nu_t$, $t\geq 0$, be defined in Proposition~\ref{pro_ito_formula} with $\mu_0 \in \cN_{n-1}$. Assume that there exists an $(\F_t)$-stopping time $\sigma$ such that $\sup\limits_{ t \in [0,\sigma] }\langle \phi_p,\nu_t \rangle<\infty$ a.s. for $p=n^2-n$. Then a.s. $\nu_{t\wedge \sigma} \in \cN_{n-1}$, $t\geq 0$. 
\end{lemma}

\begin{proof} 
  To prove the lemma, it is enough to show that a.s. $F_n(\nu_{t\wedge\sigma})=0$, $t\geq 0$, by observation~\eqref{equ_connection_f_and_ch}. For this we will apply \Itos formula to the function $F:=F_n$, which satisfies assumptions of Proposition~\ref{pro_ito_formula}, according to Remark~\ref{rem_f_is_a_polynomial}. We first compute the derivatives of $F$ which appear in~\eqref{equ_ito_formula_for_measure_valued_processes}.  For $x \in \R^d $ we have
  \begin{align}
    \frac{ \delta F(\mu) }{ \delta\mu }(x)&= \frac{ \partial }{ \partial\eps }\int_{ \R ^{dn} }    \prod_{ i,j:i \not= j }   |z^i-z^j| \prod_{ i=1 }^{ n } \left(\mu(dz^i)+\eps\delta_{x}\right)\Big|_{\eps=0}\\
    &= \sum_{ i_1=1 }^{ n } \int_{ \R ^{dn} }   \prod_{ i,j:i \not= j }   |z^i-z^j|\delta_{x}(dz^{i_1})\prod_{ i \not= i_1 }   \mu(dz^i). 
  \end{align}
  Writing $\{i,j\}\not= \{i_1,i_2\}$ for the set $\left\{ (i,j) \in [n]^2:\ i \not= j,\ (i,j) \not\in \{(i_1,i_2),(i_2,i_1)\}\right\}$, we get for $k,l \in [d]$
  \begin{align}
    \frac{\partial }{\partial x_k} \frac{ \delta F(\mu) }{ \delta\mu }(x)&=2 \sum_{ i_1,i_2:i_1 \not= i_2 }   \int_{ \R ^{dn} }  (z_k^{i_1}-z_k^{i_2}) \prod_{ \{ i,j \}\not= \{ i_1,i_2 \} }   |z^i-z^j| \delta_x(z^{i_1})\prod_{ i:i \not= i_1 }   \mu(dz^i)
  \end{align}
  and 
  \begin{align}
    \frac{\partial ^2}{\partial x_k\partial x_l}\frac{ \delta F(\mu) }{\delta\mu  }(x)&=4 \sum_{\substack{i_1,i_2,i_3:\\ i_1\not= i_2\not= i_3} }   \int_{ \R ^{dn} }   (z_k^{i_1}-z_k^{i_2})(z_l^{i_1}-z_l^{i_3})\\
    &\hspace{25mm}\prod_{ \substack{\{ i,j \}\not= \{ i_1,i_2 \}\\ \{ i,j \}\not= \{ i_1,i_3 \}} }|z^i-z^j| \delta_x(z^{i_1})\prod_{ i:i \not= i_1 }   \mu(dz^i)\\
    &+ 2 \delta_{kl}\sum_{ i_1,i_2:i_1 \not= i_2 }   \int_{ \R ^{dn} } \prod_{ \{ i,j \}\not= \{ i_1,i_2 \} }   |z^i-z^j|\delta_x(z^{i_1})\prod_{ i:i \not= i_1 }   \mu(dz^i)\\
    &= :I^{k,l}(x,\mu)+\delta_{kl}\tilde{I}^{k}(x,\mu),    
  \end{align}
  where $\delta_{kl}=\I_{\left\{ k \right\}}(l)$ denotes the Kronecker delta.  Similarly, for $x,y \in \R^d $ and $k,l \in [d]$
  \begin{align}
    \frac{ \delta^2F }{ \delta\mu^2 }(x,y)&=\sum_{ i_1,i_2:i_1 \not= i_2 }   \int_{ \R ^{dn} }   \prod_{ i \not= j }   |z^i-z^j|\delta_x(dz^{i_1})\delta_y(dz^{i_2})\prod_{ i \not= i_1,i_2 }   \mu(dz^i),  
  \end{align}
  \begin{align}
    \frac{\partial }{\partial x_k}\frac{ \delta^2F }{ \delta\mu^2 }(x,y)&= 2 \sum_{ i_1,i_2:i_1 \not= i_2 }   \sum_{ i_3:i_3 \not= i_1 } \int_{ \R ^{dn} }   (z^{i_1}_k-z^{i_3}_k) \\
    &\hspace{25mm}\prod_{ \{ i,j \}\not= \{ i_1,i_3 \} }   |z^i-z^j| \delta_x(dz^{i_1})\delta_y(dz^{i_2})\prod_{ i:i \not= i_1,i_2 }   \mu(dz^i)
  \end{align}
  and 
  \begin{align}
    \frac{\partial^2 }{\partial x_k\partial y_l}\frac{ \delta^2F }{ \delta\mu^2 }(x,y)&= 4 \sum_{ i_1,i_2:i_1 \not= i_2 }   \sum_{ i_3:i_3 \not= i_1 } \sum_{ i_4:i_4 \not= i_1,i_2 }  \int_{ \R ^{dn} }   (z^{i_1}_k-z^{i_3}_k)(z^{i_2}_l-z^{i_4}_l) \\
    &\hspace{25mm}\prod_{ \substack{\{ i,j \}\not= \{ i_1,i_3 \}\\ \{ i,j \}\not= \{ i_2,i_4 \}} }   |z^i-z^j| \delta_x(dz^{i_1})\delta_y(dz^{i_2})\prod_{ i:i \not= i_1,i_2 }   \mu(dz^i)\\
    &+ 4 \sum_{ i_1,i_2:i_1 \not= i_2 }   \sum_{ i_3:i_3 \not= i_1,i_2 } \int_{ \R ^{dn} }   (z^{i_1}_k-z^{i_3}_k)(z^{i_2}_l-z^{i_1}_l) \\
    &\hspace{25mm}\prod_{ \substack{\{ i,j \}\not= \{ i_1,i_3 \}\\ \{ i,j \}\not= \{ i_2,i_1 \}} }   |z^i-z^j| \delta_x(dz^{i_1})\delta_y(dz^{i_2})\prod_{ i:i \not= i_1,i_2 }   \mu(dz^i)\\
    &- 2 \delta_{kl} \sum_{ i_1,i_2:i_1 \not= i_2 }   \int_{ \R ^{dn} }   \prod_{ \{ i,j \}\not= \{ i_1,i_2 \} }   |z^i-z^j| \delta_x(dz^{i_1})\delta_y(dz^{i_2})\prod_{ i:i \not= i_1,i_2 }   \mu(dz^i)\\
    &=:J^{k,l}_1(x,y,\mu)+J^{k,l}_2(x,y,\mu)+\delta_{kl}\tilde{J}^{k}(x,y,\mu).
  \end{align}
  We next estimate the terms which appear after applying \Itos formula from Proposition~\ref{pro_ito_formula} to the semimartingale $F(\nu_t)$, $t\geq 0$. For the first term, we get 
  \begin{align}
    \left\langle \nabla \frac{ \delta F(\nu_s) }{\delta\nu_s  }\cdot v(s,\cdot ) , \nu_s \right\rangle&= 2 \sum_{ i_1,i_2: i_1\not= i_2 }\int_{ \R ^{dn} }   (z^{i_1}-z^{i_2})\cdot v(s,z^{i_1})\prod_{ \{ i,j \}\not= \{ i_1,i_2 \} }   |z^i-z^j|\prod_{ i=1 }^{ n } \nu_s(dz^i).
  \end{align}
  Interchanging $i_1$ and $i_2$, the expression above can be rewritten as
  \begin{align}
    &\sum_{ i_1,i_2: i_1\not= i_2 }\int_{ \R ^{dn} }   (z^{i_1}-z^{i_2})\cdot v(s,z^{i_1})\prod_{ \{ i,j \}\not= \{ i_1,i_2 \} }   |z^i-z^j|\prod_{ i=1 }^{ n } \nu_s(dz^i)\\
    &\quad+\sum_{ i_1,i_2: i_1\not= i_2 }\int_{ \R ^{dn} }   (z^{i_2}-z^{i_1})\cdot v(s,z^{i_2})\prod_{ \{ i,j \}\not= \{ i_2,i_1 \} }   |z^i-z^j|\prod_{ i=1 }^{ n } \nu_s(dz^i)\\
    &\quad=\sum_{ i_1,i_2: i_1\not= i_2 }\int_{ \R ^{dn} }   (z^{i_1}-z^{i_2})\cdot \left(v(s,z^{i_1})-v(s,z^{i_2})\right)\prod_{ \{ i,j \}\not= \{ i_1,i_2 \} }   |z^i-z^j|\prod_{ i=1 }^{ n } \nu_s(dz^i).
  \end{align}
  Therefore, using the Lipschitz continuity of $v(s,\cdot )$ from Assumption~\ref{ass_lipshitz_continuity}, we can estimate
  \begin{align}
    \left| \left\langle \nabla \frac{ \delta F(\nu_s) }{\delta\nu_s  }\cdot v_s , \nu_s \right\rangle \right|&\leq L \sum_{ i_1,i_2: i_1\not= i_2 }\int_{ \R ^{dn} }   |z^{i_1}-z^{i_2}|^2\prod_{ \{ i,j \}\not= \{ i_1,i_2 \} }   |z^i-z^j|\prod_{ i=1 }^{ n } \nu_s(dz^i)\\
    &=  L \sum_{ i_1,i_2: i_1\not= i_2 }\int_{ \R ^{dn} }   \prod_{ i,j:i\not= j }   |z^i-z^j|\prod_{ i=1 }^{ n } \nu_s(dz^i)=Ln(n-1)F(\nu_s).
  \end{align}
  In order to estimate 
  \[
    K:=\left\langle D^2 \frac{ \delta F(\nu_s) }{ \delta\nu_s }:a(s,\cdot ) , \nu_s \right\rangle+ \left\langle \nabla\otimes\nabla \frac{ \delta^2F(\nu_s) }{\delta\nu_s^2  }: \tilde{a}(s,\cdot ) , \nu_s\otimes\nu_s \right\rangle,
  \]
  we first split $J_1^{k,l}$ into three terms $J_{1,1}^{k,l},J_{1,2}^{k,l},J_{1,3}^{k,l}$ defined by 
  \begin{align}
    J_{1,1}^{k,l}(x,y,\mu)&= 4\sum_{ \substack{i_1,i_2,i_3,i_4:\\ i_1 \not= i_2 \not= i_3 \not= i_4} }\int_{ \R ^{dn} }   (z^{i_1}_k-z^{i_3}_k)(z^{i_2}_l-z^{i_4}_l) \\
    &\hspace{25mm}\prod_{ \substack{\{ i,j \}\not= \{ i_1,i_3 \}\\ \{ i,j \}\not= \{ i_2,i_4 \}} }   |z^i-z^j| \delta_x(dz^{i_1})\delta_y(dz^{i_2})\prod_{ i:i \not= i_1,i_2 }   \mu(dz^i),\\
    J_{1,2}^{k,l}(x,y,\mu)&= 4 \sum_{ \substack{ i_1,i_2,i_4:\\ i_1\not= i_2\not= i_4} } \int_{ \R ^{dn} }   (z^{i_1}_k-z^{i_2}_k)(z^{i_2}_l-z^{i_4}_l) \\
    &\hspace{25mm}\prod_{ \substack{\{ i,j \}\not= \{ i_1,i_2 \}\\ \{ i,j \}\not= \{ i_2,i_4 \}} }   |z^i-z^j| \delta_x(dz^{i_1})\delta_y(dz^{i_2})\prod_{ i:i \not= i_1,i_2 }   \mu(dz^i)
  \end{align}
  and 
  \begin{align}
    J_{1,3}^{k,l}(x,y,\mu)&= 4\sum_{ \substack{i_1,i_2,i_4:\\ i_1\not= i_2\not= i_4} } \int_{ \R ^{dn} }   (z^{i_1}_k-z^{i_4}_k)(z^{i_2}_l-z^{i_4}_l) \\
    &\hspace{25mm}\prod_{ \substack{\{ i,j \}\not= \{ i_1,i_4 \}\\ \{ i,j \}\not= \{ i_2,i_4 \}} }   |z^i-z^j| \delta_x(dz^{i_1})\delta_y(dz^{i_2})\prod_{ i:i \not= i_1,i_2 }   \mu(dz^i),
  \end{align}
  which are obtained by summing over $\{i_3:i_3\not= i_2,i_4\}$, $\{i_3:i_3=i_2,i_3\not= i_4\}$ and $\{ i_3:i_3=i_4,i_3\not= i_2 \}$ in the second sum, respectively. We remark that $J_{1,1}^{k,l}$ only appears for $n\geq 4$ and $J_{1,2}^{k,l}$, $J_{1,3}^{k,l}$ appear for $n\geq 3$. We rewrite $K$ as the sum $K_1+K_2+K_3$, where
  \begin{align}
    K_1&= \left\langle J_{1,1}(\nu_s): \tilde{a}(s,\cdot ) , \nu_s\otimes\nu_s \right\rangle,\\
    K_2&= \left\langle I(\nu_s):a(s,\cdot ) , \nu_s \right\rangle+\left\langle J_{1,2}(\nu_s): \tilde{a}(s,\cdot ) , \nu_s\otimes\nu_s \right\rangle\\
    &+\left\langle J_{1,3}(\nu_s): \tilde{a}(s,\cdot ) , \nu_s\otimes\nu_s \right\rangle+\left\langle J_{2}(\nu_s): \tilde{a}(s,\cdot ) , \nu_s\otimes\nu_s \right\rangle
  \end{align}
  and 
  \[
    K_3= \left\langle \tilde{I}(\nu_s)\cdot b(s,\cdot ) , \nu_s \right\rangle+\left\langle \tilde{J}(\nu_s): \tilde{b}(s,\cdot ) , \nu_s\otimes\nu_s \right\rangle.
  \]
  In the above equality the functions $b$ and $\tilde{b}$ are defined by $b_k(s,\cdot )=a_{k,k}(s,\cdot )$ and $\tilde{b}_k(s, \cdot )=\tilde{a}_{k,k}(s,\cdot )$, $k \in [d]$. We next compute 
  \begin{align}
    K_1&= 4\sum_{ \substack{i_1,i_2,i_3,i_4:\\ i_1 \not= i_2 \not= i_3 \not= i_4} } \int_{ \R ^{dn} }   \left[(z^{i_1}-z^{i_3})\otimes(z^{i_2}-z^{i_4})\right]: \tilde{a}(s,z^{i_1},z^{i_2}) \prod_{ \substack{\{ i,j \}\not= \{ i_1,i_3 \}\\ \{ i,j \}\not= \{ i_2,i_4 \}} }   |z^i-z^j| \prod_{ i=1 }^n   \nu_s(dz^i)\\
    &= \sum_{ \substack{i_1,i_2,i_3,i_4:\\ i_1 \not= i_2 \not= i_3 \not= i_4} } \int_{ \R ^{dn} }   \left[(z^{i_1}-z^{i_3})\otimes(z^{i_2}-z^{i_4})\right]\\
    &\hspace{17mm}:\left\langle g(s,z^{i_1},\cdot )-g(s,z^{i_3},\cdot ), g(s,z^{i_2},\cdot )-g(s,z^{i_4}, \cdot ) \right\rangle_{\m} \prod_{ \substack{\{ i,j \}\not= \{ i_1,i_3 \}\\ \{ i,j \}\not= \{ i_2,i_4 \}} }   |z^i-z^j| \prod_{ i=1 }^n   \nu_s(dz^i)\\
    &= \sum_{ \substack{i_1,i_2,i_3,i_4:\\ i_1 \not= i_2 \not= i_3 \not= i_4} } \int_{ \R ^{dn} }   \big\langle (z^{i_1}-z^{i_3})\cdot (g(s,z^{i_1},\cdot )-g(s,z^{i_3},\cdot )),\\
    & \hspace{17mm}(z^{i_2}-z^{i_4})\cdot(g(s,z^{i_2},\cdot )-g(s,z^{i_4}, \cdot )) \big\rangle_{\m} \prod_{ \substack{\{ i,j \}\not= \{ i_1,i_3 \}\\ \{ i,j \}\not= \{ i_2,i_4 \}} }   |z^i-z^j| \prod_{ i=1 }^n   \nu_s(dz^i).
  \end{align}
  Therefore, using the Cauchy-Schwarz inequality and then the Lipschitz continuity of $g(s,\cdot )$, we estimate 
  \begin{align}
    |K_1|&\leq \sum_{ \substack{i_1,i_2,i_3,i_4:\\ i_1 \not= i_2 \not= i_3 \not= i_4} } \int_{ \R ^{dn} } |z^{i_1}-z^{i_3}|\|g(s,z^{i_1},\cdot )-g(s,z^{i_3},\cdot )\|_{\m} |z^{i_2}-z^{i_4}|\|g(s,z^{i_2},\cdot )-g(s,z^{i_4}, \cdot )\|_{\m} \\
    &\hspace{25mm}\prod_{ \substack{\{ i,j \}\not= \{ i_1,i_3 \}\\ \{ i,j \}\not= \{ i_2,i_4 \}} }   |z^i-z^j| \prod_{ i=1 }^n   \nu_s(dz^i)\\
    &\leq L^2\sum_{ \substack{i_1,i_2,i_3,i_4:\\ i_1 \not= i_2 \not= i_3 \not= i_4} } \int_{ \R ^{dn} } |z^{i_1}-z^{i_3}|^2|z^{i_2}-z^{i_4}|^2 \prod_{ \substack{\{ i,j \}\not= \{ i_1,i_3 \}\\ \{ i,j \}\not= \{ i_2,i_4 \}} }   |z^i-z^j| \prod_{ i=1 }^n   \nu_s(dz^i)\\
    &= L^2\sum_{ \substack{i_1,i_2,i_3,i_4:\\ i_1 \not= i_2 \not= i_3 \not= i_4} } \int_{ \R ^{dn} } \prod_{ i,j:i \not= j }   |z^i-z^j| \prod_{ i=1 }^n   \nu_s(dz^i)=L^2 \frac{ n! }{ (n-4)! }F(\nu_s).
  \end{align}

  We rewrite $K_2$ in similar way as $K_1$:
  \begin{align}
    K_2&=  4 \sum_{\substack{i_1,i_2,i_3:\\ i_1\not= i_2\not= i_3} }   \int_{ \R ^{dn} }   \left[(z^{i_1}-z^{i_2})\otimes(z^{i_1}-z^{i_3})\right]: a(s,z^{i_1})\prod_{ \substack{\{ i,j \}\not= \{ i_1,i_2 \}\\ \{ i,j \}\not= \{ i_1,i_3 \}} }|z^i-z^j| \prod_{ i=1 }^n   \nu_s(dz^i)\\
    &+ 4 \sum_{ \substack{i_1,i_2,i_4:\\i_1\not= i_2\not= i_4} } \int_{ \R ^{dn} }   \left[(z^{i_1}-z^{i_2})\otimes(z^{i_2}-z^{i_4})\right]: \tilde{a}(s,z^{i_1},z^{i_2}) \prod_{ \substack{\{ i,j \}\not= \{ i_1,i_2 \}\\ \{ i,j \}\not= \{ i_2,i_4 \}} }   |z^i-z^j| \prod_{ i=1 }^n   \nu_s(dz^i)\\
    &+ 4\sum_{ \substack{i_1,i_2,i_4:\\ i_1\not= i_2\not= i_4} } \int_{ \R ^{dn} }   \left[(z^{i_1}-z^{i_4})\otimes(z^{i_2}-z^{i_4})\right]: \tilde{a}(s,z^{i_1},z^{i_2}) \prod_{ \substack{\{ i,j \}\not= \{ i_1,i_4 \}\\ \{ i,j \}\not= \{ i_2,i_4 \}} }   |z^i-z^j| \prod_{ i=1 }^n   \nu_s(dz^i)\\
    &+ 4 \sum_{ \substack{i_1,i_2,i_3:\\ i_1 \not= i_2\not= i_3} } \int_{ \R ^{dn} }   \left[(z^{i_1}-z^{i_3})\otimes(z^{i_2}-z^{i_1})\right]: \tilde{a}(s,z^{i_1},z^{i_2}) \prod_{ \substack{\{ i,j \}\not= \{ i_1,i_3 \}\\ \{ i,j \}\not= \{ i_2,i_1 \}} }   |z^i-z^j| \prod_{ i=1 }^n   \nu_s(dz^i).
  \end{align}
  Interchanging the indexes of summations in second, third and fourth terms in the following way $[i_1 \mapsto i_2,i_2 \mapsto i_1, i_4 \mapsto i_3]$, 
  $[i_1 \mapsto i_2,i_2 \mapsto i_3,i_4 \mapsto i_1]$ and $[i_2 \mapsto i_3,i_3 \mapsto i_2]$, respectively, and using the equalities $a(s,x)=\langle g(s,x, \cdot ) ,g(s,x, \cdot )  \rangle_{\m}$, $\tilde{a}(s,x,y)=\langle g(s,x, \cdot ) , g(s,y, \cdot ) \rangle$, we get
  \begin{align}
    K_2&=  4 \sum_{\substack{i_1,i_2,i_3:\\ i_1\not= i_2\not= i_3} }   \int_{ \R ^{dn} }   \left[(z^{i_1}-z^{i_2})\otimes(z^{i_3}-z^{i_1})\right]: \\
    &\hspace{10mm}\big[ -\langle g(s,z^{i_1}, \cdot ) ,g(s,z^{i_1},\cdot )  \rangle_{\m}+ \langle g(s,z^{i_2}, \cdot ) ,g(s,z^{i_1}, \cdot )  \rangle_{\m} \\
    &\hspace{20mm} -\langle g(s,z^{i_2}, \cdot ) ,g(s,z^{i_3},\cdot )\rangle_{\m}+ \langle g(s,z^{i_1}, \cdot ) , g(s,z^{i_3}, \cdot ) \rangle_{\m}\big] \prod_{ \substack{\{ i,j \}\not= \{ i_1,i_2 \}\\ \{ i,j \}\not= \{ i_1,i_3 \}} }|z^i-z^j| \prod_{ i=1 }^n   \nu_s(dz^i)\\
    &=  4 \sum_{\substack{i_1,i_2,i_3:\\ i_1\not= i_2\not= i_3} }   \int_{ \R ^{dn} }   \left[(z^{i_1}-z^{i_2})\otimes(z^{i_3}-z^{i_1})\right]\\
    & \hspace{17mm}:  \langle g(s,z^{i_1}, \cdot )-g(s,z^{i_2},\cdot ) ,g(s,z^{i_3},\cdot )-g(s,z^{i_1},\cdot )  \rangle_{\m}\prod_{ \substack{\{ i,j \}\not= \{ i_1,i_2 \}\\ \{ i,j \}\not= \{ i_1,i_3 \}} }|z^i-z^j| \prod_{ i=1 }^n   \nu_s(dz^i).
  \end{align}
  Similarly as before, we get
  \[
    K_2\leq 4L^2 \frac{ n! }{ (n-3)! }F(\nu_s).
  \]
  We now compute 
  \begin{align}
    K_3&= 2\sum_{ i_1,i_2:i_1 \not= i_2 }   \int_{ \R ^{dn} } \left[\sum_{ k=1 }^{ d } a_{k,k}(s,z^{i_1})\right] \prod_{ \{ i,j \}\not= \{ i_1,i_2 \} }   |z^i-z^j|\prod_{ i=1 }^n   \nu_s(dz^i)\\
    &-2\sum_{ i_1,i_2:i_1 \not= i_2 }   \int_{ \R ^{dn} }  \left[ \sum_{ k=1 }^{ d } \tilde{a}_{k,k}(s,z^{i_1},z^{i_2}) \right] \prod_{ \{ i,j \}\not= \{ i_1,i_2 \} }   |z^i-z^j| \prod_{ i=1 }^n   \nu_s(dz^i)\\
    &= \sum_{ i_1,i_2:i_1 \not= i_2 }   \int_{ \R ^{dn} } \Bigg[\sum_{ k=1 }^{ d }\big( \langle g_k(s,z^{i_1},\cdot ) ,g_k(s,z^{i_1},\cdot )  \rangle_{\m}+\langle g_k(s,z^{i_2},\cdot ) ,g_k(s,z^{i_2},\cdot )  \rangle_{\m}\\
    &\hspace{30mm}-2\langle g_k(s,z^{i_1},\cdot ) ,g_k(s,z^{i_2},\cdot )  \rangle_{\m}\big)\Bigg] \prod_{ \{ i,j \}\not= \{ i_1,i_2 \} }   |z^i-z^j|\prod_{ i=1 }^n   \nu_s(dz^i)\\
    &= \sum_{ i_1,i_2:i_1 \not= i_2 }   \int_{ \R ^{dn} } \sum_{ k=1 }^{ d }  \|g_k(s,z^{i_1},\cdot )-g_k(s,z^{i_2},\cdot )\|_{\m}^2 \prod_{ \{ i,j \}\not= \{ i_1,i_2 \} }   |z^i-z^j|\prod_{ i=1 }^n   \nu_s(dz^i).
  \end{align}
  Hence, using the Lipschitz continuity of $g(s,\cdot )$ again, we get 
  \begin{align}
    |K_3|&\leq L^2\sum_{ i_1,i_2:i_1 \not= i_2 }   \int_{ \R ^{dn} }  |z^{i_1}-z^{i_2}|^2 \prod_{ \{ i,j \}\not= \{ i_1,i_2 \} }   |z^i-z^j|\prod_{ i=1 }^n   \nu_s(dz^i)=L^2 n(n-1) F(\nu_s).
  \end{align}
  Combining obtained estimates for $K_i$, $i \in [3]$, we can see that $|K|\leq C_nL^2 F(\nu_s)$. 
  
  We next define for $p=n^2-n$ and every $k\geq 1$ the $(\F_t)$-stopping time $\sigma_k$ as follows 
  \[
    \sigma_k=\inf\left\{ t\geq 0:\ \langle \phi_p , \nu_t \rangle \geq k \right\}\wedge \sigma.
  \]
  Then there exists a (non-random) constant $C>0$ such that $F(\mu_{t})\leq C$ for all $t \in [0,\sigma_k]$, by Remark~\ref{rem_f_is_a_polynomial}.   Applying \Itos formula to $F(\mu_t\wedge \sigma_k)$, $t\geq 0$, taking the expectations and using estimates obtained above, we have for every $k\geq 1$ 
  \begin{align}
    \e{ F(\mu_{t\wedge \sigma_k}) }&\leq Ln(n-1) \e\int_{ 0 }^{ t\wedge \sigma_k } F(\mu_s) ds+ L^2C_n \e\int_{ 0 }^{ t\wedge \sigma_k } F(\mu_s) ds\\
    &\leq Ln(n-1) \int_{ 0 }^{ t} \e F(\mu_{s\wedge \sigma_k}) ds+ L^2C_n \int_{ 0 }^{ t} \e F(\mu_{s\wedge \sigma_k}) ds , \quad t\geq 0.  
  \end{align}
  By Gronwall's lemma, $\E{ F(\mu_{t\wedge \sigma_k}) }=0$ for all $t\geq 0$. This implies that for every $t\geq 0$ a.s. $\mu_{t\wedge \sigma_k} \in \cN_{n-1}$. Since $\cN_{n-1}$ is closed and $\mu_t$, $t\geq 0$, is continuous, we obtain that a.s. $\mu_{t\wedge \sigma_k} \in \cN_{n-1}$ for all $t\geq 0$ and $k\geq 1$. Making $k\to\infty$, we can conclude that with probability 1 $\mu_{t\wedge \sigma} \in \cN_{n-1}$, $t\geq 0$, that ends the proof of the proposition.
\end{proof}

We next prove the main statement of this section. 

\begin{proof}[Proof of Theorem~\ref{the_uniqueness_for_atomic_initial_conditions}] 
  Let $\mu_t$, $t\geq 0$, be a solution to the stochastic mean-field equation~\eqref{equ_mean_field_equation_with_corr_noise} started from $\mui \in \cN_n$ and satisfying~\eqref{equ_boundedness_of_pth_moment}. We will show that $\mu_t$ is a superposition solution. Let $\mui=\sum_{ i=1 }^{ n } \alpha_i\delta_{x^i}$ for some $\alpha_i\geq 0$ and $x^i \in \R^d $, $i \in [n]$. Without loss of generality, we may assume that $\alpha_i$, $i \in [n]$, are strictly positive and $x^i$, $i \in [n]$, are distinct.  Applying Lemma~\ref{lem_mu_stays_in_hn}, we can conclude that a.s. $\mu_t \in \cN_n$ for all $t\geq 0$. By the continuity of $\mu_{\cdot }$, there exist $(\F_t)$-adapted continuous processes $\alpha_i(t)$, $Y^i(t)$, $t\geq 0$, $i \in [n]$, such that $a_i(0)=\alpha_i$, $Y_i(0)=x_i$, $i \in [n]$, and 
  \[
    \mu_t=\sum_{ i=1 }^{ n } \alpha_i(t)\delta_{Y^i(t)}, \quad t\geq 0.
  \]
  For every $i \in [n]$ we define the stopping time 
  \[
    \tau_0:=\inf\left\{ t\geq 0:\   |Y^i(t)-x^i|\geq   \eps\ \ \mbox{or}\ \ |Y^j(t)-x^i|\leq 2 \eps\ \ \mbox{for some}\ \ j \not= i \right\},
  \]
  where $\eps= \frac{1}{ 3 }\min\limits_{ i \not= j }|x^i-x^j|$.  For each $i \in [n]$ we next consider functions $\varphi, \psi_i \in \Cf^2_c(\R^d )$ such that  $\varphi_i(x)=1$, $\psi_i(x)=x^i$ for $|x-x^i|<\eps$ and $\varphi_i(x)=\psi_i(x)=0$ for $|x-x^i|\geq 2 \eps$. Then, using the definition of a solution to the stochastic mean-field equation, we obtain
  \begin{align}
    \alpha_i(t\wedge\tau_0)=\langle \varphi_i , \mu_{t\wedge\tau_0} \rangle=\langle \varphi_i , \mu_0 \rangle=\alpha_i, \quad t\geq 0,
  \end{align}
  since $\nabla \varphi_i(x)=0$ and $D^2 \varphi_i(x)=0$ for $|x-x^i|<\eps$ and for $|x-x^i|>2 \eps$. Similarly, using additionally the previous observation, for every $t\geq 0$
  \begin{align}
    \alpha_i Y_i(t\wedge\tau_0)&=  \alpha_i(t\wedge\tau_0) Y_i(t\wedge\tau_0)= \langle \psi_i , \mu_{t\wedge\tau_0} \rangle  \\
    &= \alpha_iY_i(0)+ \int_{ 0 }^{ t\wedge \tau_0 } \alpha_i(s)v(s,Y_i(s))ds+ \int_{ 0 }^{ t\wedge\tau_0 } \int_{ \Theta }   \alpha_i(s)g(s,Y_i(s),\theta)W(d \theta,ds)\\
    &= \alpha_iY_i(0)+ \int_{ 0 }^{ t\wedge \tau_0 } \alpha_iv(s,Y_i(s))ds+ \int_{ 0 }^{ t\wedge\tau_0 } \int_{ \Theta }   \alpha_ig(s,Y_i(s),\theta)W(d \theta,ds)
  \end{align}
  holds. Hence, $Y_i$ is a solution to the usual SDE
  \begin{equation} 
  \label{equ_equation_for_y}
    dY_i(t)=v(t,Y_i(t))dt+ \int_{ \Theta }   g(t,Y_i(t),\theta)W(d \theta,dt), \quad Y_i(0)=x^i 
  \end{equation}
  on the interval $[0,\tau_0]$. Note that~\eqref{equ_equation_for_y} has at most one solution due to the Lipschitz continuity of its coefficients. Let $X(u,t)$, $t\geq 0$, $u \in \R^d $, be a solution to SDE with interaction~\eqref{equ_equation_with_interaction}. Since $X(x^i,t)$, $t\geq 0$, also solves equation~\eqref{equ_equation_for_y}, one has $Y_i(t)=X(x^i,t)$, $t \in [0,\tau_0]$, for each $i \in [n]$. Next, considering the process $\mu_t^1:=\mu_{t+\tau_0}$ conditioning to the $\sigma$-algebra $\F_{\tau_0}$, we get that, $\mu_t^1$, $t\geq 0$, satisfies the same stochastic mean-field equation with the initial condition $\mu_{\tau_0}$. Applying our argument above again to $\mu_t^{1}$, $t\geq 0$, we get that $\alpha_i(t+\tau_0)=\alpha_i$ and $Y_{i}(t+\tau_0)=X(x^i,t+\tau_0)$ for $t \in [0,\tau_1]$ and $i \in [n]$, where
  \begin{align}
    \tau_1:=\inf\big\{ t\geq 0:\ &   |Y_i(t+\tau_0)-Y_i(\tau_0)|\geq   \eps\\
    &\qquad\mbox{or}\ \ |Y_j(t+\tau_0)-Y_i(\tau_0)|\leq 2 \eps\ \ \mbox{for some}\ \ j \not= i \big\},
  \end{align}
  and $\eps= \frac{1}{ 3 }\min\limits_{ i \not= j }|Y_i(\tau_0)-Y_j(\tau_0)|$. Hence, $\alpha_i(t)=\alpha_i$ and $Y_i(t)=X(x^i,t)$, $t \in [0,\tau_0+\tau_1]$, $i \in [n]$. Repining our argument infinitely many times and using the uniform continuity of $Y_i$, $i \in [n]$, a.s. on any compact time interval, we get that $\alpha_i(t)=\alpha_i$ and $Y_i(t)=X(x^i,t)$, $t \in [0,\tau]$, for 
  \[
    \tau= \inf\left\{ t\geq 0:\ Y_i(t)=Y_j(t)\ \ \mbox{for some}\ \ j \not= i \right\}.
  \]
  Since $X(x^i,\cdot )$ and $X(x^j,\cdot )$ never meet for distinct $x^i$, $x^j$, by Lemma~\ref{lem_two_particles_newer_meet}, we can conclude that $\tau=+\infty$. Consequently, $\mu_t=\mui \circ X^{-1}(\cdot ,t)$, $t\geq 0$. This ends the proof of the uniqueness of the superposition principle for the stochastic mean-field equation.
\end{proof}

\subsubsection{Initial condition with \texorpdfstring{$L_2$}{L2}-density}
\label{ssub:initial_condition_with_l2_density}

In this section, we adapt the method from~\cite{Kurtz:1999} in order to prove the uniqueness for the stochastic mean-field equation~\eqref{equ_mean_field_equation_with_corr_noise} if the initial condition has an $L_2$-density with respect to the Lebesgue measure and the coefficients are bounded a.s. We will need the following assumption.

\begin{assumption} 
  \label{ass_boundedness_assumption}
  For every $T>0$ there exists a (non-random) constant $L>0$ such that a.s. for every $x \in \R^d $, $t \in [0,T]$, and $\mu \in \cP_2(\R^d )$
  \[
    |V(t,x,\mu)|+ \left\| |G(t,x,\mu,\cdot )| \right\|_{\m}\leq L
  \]
  holds.
\end{assumption}

Using the notation $\frac{ d\mu }{ dx }$ for the density of a measure $\mu \in \cP(\R^d )$ with respect to the Lebesgue measure, we state the main result of this section.

\begin{theorem} 
  \label{the_uniqueness_for_initial_condition_with_l2_density}
  Let the coefficients of stochastic mean-field equation~\eqref{equ_mean_field_equation_with_corr_noise} satisfy Assumptions~\ref{ass_basic_assumption}, \ref{ass_lipshitz_continuity} and~\ref{ass_boundedness_assumption}, and $\mui \in \cP_2(\R^d )$ be absolutely continuous with respect to Lebesgue measure with $\frac{ d\mui }{ dx } \in L_2(\R^d )$. Then~\eqref{equ_mean_field_equation_with_corr_noise} has a unique solution $\mu_t$, $t\geq 0$, started from $\mui$. Moreover, $\mu_t$, $t\geq 0$, is a superposition solution which is absolutely continuous with respect to the Lebesgue measure with 
  \begin{equation} 
  \label{equ_integrability_of_l2_density}
  \e \Big\|\frac{ d\mu_t }{ dx }\Big\|_{L_2}^2 \leq e^{Ct} \left\|\frac{ d\mui }{ dx }\right\|_{L_2}^2
  \end{equation}
   for all $t\geq 0$ and some constant $C>0$.
\end{theorem}

\begin{proof}
  The proof of the theorem is similar to the proof of the uniqueness result for the nonlinear SPDE in~\cite[Section~3]{Kurtz:1999}, which differs from the equation~\eqref{equ_mean_field_equation_with_corr_noise} by non-random and homogeneous (independent of time) coefficients, on one side, and a more general structure on the other side.  Therefore, we just describe the main steps omitting details. Let $\mu_t$, $t\geq 0$, be a solution to~\eqref{equ_mean_field_equation_with_corr_noise} started from $\mui$. As in the previous section, we freeze $\mu_t$ in the coefficients. As before, we set $v(t,x)=V(t,x,\mu_t)$, $a(t,x)=A(t,x,\mu_t)$ and $g(t,x,\theta)=G(t,x,\mu_t,\theta)$. We will consider $\mu_t$, $t\geq 0$, as a solution to the following linear SPDE
\begin{equation} 
  \label{equ_linear_dk_equation}
  \begin{split} 
    d\mu_t&=  \frac{1}{ 2 }D^2: (a(t,\cdot )\mu_t)dt-\nabla \cdot (v(t,\cdot )\mu_t)dt-\int_{ \Theta }   \nabla \cdot (g(t,\cdot ,\theta)\mu_t)W(d \theta,dt). \\
  \end{split}
\end{equation}
This means that equality~\eqref{equ_definition_of_process_nu} holds for the process $\mu_t$, $t\geq 0$, where test functions $\varphi$ can be taken in $\Cf_b^2(\R^d )$, by Corollary~\ref{cor_testing_of_solutions_to_dk_equation_by_bounded_functions} and Assumption~\ref{ass_lipshitz_continuity}.  

Our goal is to show that~\eqref{equ_linear_dk_equation} has a unique solution. But now we will assume that a solution $\nu_t$, $t\geq 0$, can take values in the space $\cM(\R^d )$ of all signed measures on $\R^d $ with finite total variations and $\nu_0 \in \cM(\R^d )$. Thus, let $\nu_t$, $t\geq 0$, be a continuous $(\F_t)$-adapted process which satisfies~\eqref{equ_definition_of_process_nu} for all $\varphi \in \Cf_b(\R^d )$ and $\nu_0$ is absolutely continuous with respect to the Lebesgue measure with $\frac{ d\nu_0 }{ dx }$ from $L_2(\R^d )$.

For any $\rho \in \cM(\R^d )$ and $\eps>0$, we define 
\[
  P_{\eps}\rho(x):=\int_{ \R^d  }   p_{\eps}(x-y)\rho(dy), 
\]
where $p_{\eps}$ is the heat kernel defined by $p_{\eps}(x)= \frac{1}{ (2 \pi \eps)^{-d/2}}e^{- \frac{ |x|^2 }{ 2 \eps }} $. For every $\varphi \in \Cf_b(\R^d)$, we will also write $P_{\eps}\varphi(x)$ for $\int_{ \R^d  }   p_{\eps}(x-y)\varphi(y)dy $ and $P_{\eps}(\varphi\rho)(x)$ for $\int_{ \R^d  }   p_{\eps}(x-y)\varphi(y)\rho(dy) $.  Setting $\nu^{\eps}_t:=P_{\eps}\nu_t$, $t\geq 0$, and following the computations from~\cite[Section~3]{Kurtz:1999}, based on the application of \Itos formula, we get 
\begin{align}
  \e \|\nu^{\eps}_t\|_{L_2}^2 &=\|\nu_0^{\eps}\|_{L_2}^2-\sum_{ i=1 }^{ d }\e \int_{ 0 }^{ t } 2 \left\langle \nu^{\eps}_s,\partial_{x_i}P_{\eps}(v_i(s,\cdot )\nu_s) \right\rangle_{L_2}ds  \\
  &+\sum_{ i,j=1 }^{ d } \e \int_{ 0 }^{ t }  \left\langle \nu^{\eps}_s,\partial_{x_i}\partial_{x_j}P_{\eps}(a_{i,j}(s, \cdot )\nu_s) \right\rangle_{L_2}ds \\
  &+\e \int_{ 0 }^{ t }\int_{ \Theta }  \Big\|\sum_{ i=1 }^{ d } \partial_{x_i}P_{\eps}(g_i(s,\cdot ,\theta)\nu_s)\Big\|_{L_2}^2\m(d \theta)ds, \quad t\geq 0.
\end{align}

Using Assumptions~\ref{ass_lipshitz_continuity},~\ref{ass_boundedness_assumption} and \cite[Lemma~3.2]{Kurtz:1999}, one can immediately conclude that for every $T>0$ there exists a constant $C$ such that
\begin{equation} 
  \label{equ_estimate_of_the_drift_term_in_the_kurtz_approach}
  \e \int_{ 0 }^{ t } 2 \left\langle \nu^{\eps}_s,\partial_{x_i}P_{\eps}(v_i(s,\cdot )\nu_s) \right\rangle_{L_2}ds\leq C \int_{ 0 }^{ t }  \e \left\|P_{\eps}(|\nu_s|)\right\|_{L_2}^2ds, \quad t \in [0,T],
\end{equation}
where $|\nu_s|$ is the total variation of $\nu_s$.  Using the integration by parts in the $L_2$-norm, the sum of the last two integrals in the expansion of $\e \|\nu^{\eps}_t\|_{L_2}^2$ can be rewritten as 
\begin{align}
  \frac{1}{ 2 }\sum_{ i,j=1 }^{ d }& \int_{ \R^d  }   \int_{ \R^d  }   \left( \frac{ (x_i-y_i)(x_j-y_j) }{ 4\eps^2 }- \frac{1}{ 2\eps }\I_{\left\{ i = j \right\}} \right)p_{2 \eps}(x-y)\\
  &\cdot\langle g_i(s,y,\theta)-g_i(s,x,\theta) , g_j(s,y,\theta)-g_j(s,x,\theta) \rangle_{\m}\nu_s(dx)\nu_s(dy).  
\end{align}
This leads to the same estimate~\eqref{equ_estimate_of_the_drift_term_in_the_kurtz_approach}, see the proof of \cite[Lemma~3.3]{Kurtz:1999} for more details. Hence, for every $T>0$, there exists a constant $C>0$ such that
\begin{equation} 
  \label{equ_linear_spde_g_inequality}
  \e \|\nu^\eps_t\|_{L_2}^2 \leq \|\nu^{\eps}_0\|_{L_2}^2+C \int_{ 0 }^{ t } \e \left\|P_{\eps}(|\nu_s|)\right\|_{L_2}^2 ds.
\end{equation}
Following proofs of Corollary~3.1 and Theorem~3.3 from~\cite{Kurtz:1999}, we can conclude that every solution to the linear SPDE~\eqref{equ_linear_dk_equation} with $L_2$-density is unique. Moreover, if its solution $\nu_t$, $t\geq 0$, is a measure-valued process, then $\nu_t$ is absolutely continuous with respect to the Lebesgue measure with density satisfying~\eqref{equ_integrability_of_l2_density} for every $t\geq 0$. This immediately implies that the process $\mu_t$, $t\geq 0$, is a unique solution to~\eqref{equ_linear_dk_equation} started from $\mui$ and is absolutely continuous with respect to the Lebesgue measure that satisfies~\eqref{equ_integrability_of_l2_density}.

In order to show that $\mu_t$, $t\geq 0$, is superposition solution, consider a solution $Y(u,t)$, $t\geq 0$, $u \in \R^d $, to the equation 
\begin{equation} 
  \label{equ_flow_sde_freeze}
  \begin{split} 
    dY(u,t)&= v(t,Y(u,t))dt+ \int_{ \Theta }  g(t,Y(u,t),\theta )W(d \theta ,dt), \\
    Y(u,0)&= u, \quad \bar\nu_t=\mui \circ Y(\cdot ,t)^{-1},
  \end{split}
\end{equation}
which exists and is unique due to Theorem~\ref{the_well_posedness_of_sde_with_interaction}, since its coefficients $v$, $g$ satisfy Assumptions~\ref{ass_basic_assumption},~\ref{ass_lipshitz_continuity}. By Lemma~\ref{lem_from_sde_with_interaction_to_dk_equation}, $\bar\nu_t$, $t\geq 0$, is a solution to the~\eqref{equ_linear_dk_equation} started from $\mui$. The uniqueness result, stated above, yields that $\bar\nu_t=\mu_t$, $t\geq 0$. On the other hand, the unique solution $X$ to~\eqref{equ_equation_with_interaction} solves also~\eqref{equ_flow_sde_freeze}. Hence, $Y$ must coincide with $X$. This completes the proof of the superposition principle.

Then, the uniqueness for the stochastic mean-field equation~\eqref{equ_mean_field_equation_with_corr_noise} directly follows from Corollary~\ref{cor_superposition_principle}.
\end{proof}

\subsubsection{Initial condition with finite second moment}
\label{ssub:arbitrary_initial_conditions}

We will now obtain the well-posedness and superposition principle for the stochastic mean-field equation for general initial conditions. This case includes both types of initial conditions considered before, however,  stronger assumptions on the regularity of the coefficients will be needed. Our main idea is to transform a solution to the stochastic mean-field equation by a smooth (in space) stochastic flow to get a solution to a continuity equation with random coefficients for which the superposition principle can be easily obtain, e.g., by a duality method. This will imply that the original equation has only superposition solutions which will yield the uniqueness result by Corollary~\ref{cor_superposition_principle}.  For the construction of the transformation flow, we will mainly use results from~\cite{Kunita:1990}. 

We first introduce the following assumption on the coefficients of the equation and formulate the main result of this section.

\begin{assumption} 
  \label{ass_smoothness_of_coefficients}
  There exists $\delta \in (0,1)$ such that $V_i(t,\cdot  ,\mu) \in \Cf_{lb}^{1,\delta}(\R^d )$, $\tilde{A}_{i,j}(t,\cdot ,\mu)\in \tilde{\Cf}^{3,\delta}_{lb}(\R^d )$ a.s. for all $t\geq 0$, $\mu \in \cP_2(\R^d )$, $i,j \in [d]$, and for every $T>0$ and a compact set $K \subset \cP_2(\R^d )$ a.s. 
  \[
    \sup\limits_{ t \in [0,T],\mu \in K } \left( \|V_i(t,\cdot ,\mu)\|_{1+\delta}+\|\tilde A_{i,j}(t,\cdot ,\mu)\|_{3+\delta}^{\sim} \right)< \infty, \quad i,j \in [d],
  \]
  where $\tilde{A}$ was defined in Remark~\ref{rem_about_quadratic_variation_of_solution_to_dk_equation}.
\end{assumption}

\begin{remark} 
  \label{rem_integral_version_of_assumption_4}
  Let $\mu_t$, $t\geq 0$, be an arbitrary continuous process in $\cP_2(\R^d )$. Then Assumption~\ref{ass_smoothness_of_coefficients} implies that for each $i,j \in [d]$ the processes $\int_{ 0 }^{ t }  V_i(s,\cdot ,\mu_s) ds$, $t\geq 0$, and $\int_{ 0 }^{ t } \tilde A_{i,j}(s,\cdot ,\mu_s)ds$, $t\geq 0$, are a.s. continuous in $\Cf^{1,\delta}(\R^d )$ and $\tilde\Cf^{3,\delta}(\R^d )$, respectively. Indeed, due to the continuity of the process $\mu_t$, $t\geq 0$,  the set $K^{\mu}_T:=\{ \mu_t:\ t \in [0,T] \}$ is compact in $\cP_2(\R^d )$ for every $T>0$. Then the direct computation shows that for every $T>0$, $t,t' \in [0,T]$ and compact $K \subset \R^d $
  \[
    \left\|\int_{ 0 }^{ t } V_i(s, \cdot ,\mu_s)ds- \int_{ 0 }^{ t' } V_i(s,\cdot ,\mu_s)ds  \right\|_{1+\delta,K}\leq \sup\limits_{ \substack{s \in [0,T],\\ \mu \in K_T^{\mu}} } \left\|V_i(s,\cdot ,\mu)\right\|_{1+\delta}|t-t'|
  \]
  and
  \[
    \left\|\int_{ 0 }^{ t } \tilde A_{i,j}(s, \cdot ,\mu_s)ds- \int_{ 0 }^{ t' } \tilde A_{i,j}(s,\cdot ,\mu_s)ds  \right\|_{3+\delta,K}^{\sim}\leq \sup\limits_{ \substack{s \in [0,T],\\ \mu \in K_T^{\mu}} } \left\|\tilde A_{i,j}(s,\cdot ,\mu)\right\|_{3+\delta}|t-t'|
  \]
  for all $i,j \in [d]$, that guarantees the continuity.
\end{remark}

\begin{theorem} 
  \label{the_well_posedness_of_dk_equation_in_general_case}
  Let the coefficients of the stochastic mean-field equation~\eqref{equ_mean_field_equation_with_corr_noise} satisfy Assumptions~\ref{ass_basic_assumption},~\ref{ass_lipshitz_continuity} and~\ref{ass_smoothness_of_coefficients}. Then for every $\mui \in \cP_2(\R^d )$ the equation~\eqref{equ_mean_field_equation_with_corr_noise} has a unique solution started from $\mui$. Moreover, it is a superposition solution.
\end{theorem}

In order to prove the theorem we will state a few auxiliary statements. Let $\mu_t$, $t\geq 0$, be a solution to the stochastic mean-field equation~\eqref{equ_mean_field_equation_with_corr_noise} whose coefficients satisfy Assumptions~\ref{ass_basic_assumption},~\ref{ass_lipshitz_continuity} and~\ref{ass_smoothness_of_coefficients}. As before, we will freeze $\mu_t$ in the coefficients, setting $v(t,x)=V(t,x,\mu_t)$, $a(t,x)=A(t,x,\mu_t)$, $\tilde{a}(t,x,y)=\tilde{A}(t,x,y,\mu_t)$ and $g(t,x,\theta)=G(t,x,\mu_t,\theta)$. We consider the following field of local martingales 
  \begin{align}
    M(x,t)&=  \int_{ 0 }^{ t } \int_{ \Theta }     g(s,x,\theta) W(d \theta,ds), \quad x \in \R^d,\ \ t\geq 0.
  \end{align}
  Note that its local quadratic variation 
  \[
    \left[ M_i(x,\cdot ) ,M_j(y, \cdot )  \right]_t =\int_{ 0 }^{ t } \tilde{a}_{i,j}(s,x,y)ds, \quad t\geq 0, \ \ x,y \in \R^d , \ \ i,j \in [d],
  \]
  is $\tilde{\Cf}^{3,\delta}(\R^d )$-valued continuous process, by Remark~\ref{rem_integral_version_of_assumption_4}. Consequently, there exists a version of $M$ which is a $\Cf^{3,\delta'}$-valued continuous process for every $\delta' \in (0,\delta)$, by \cite[Theorem~3.1.2]{Kunita:1990}. Moreover, for each $\alpha$ with $|\alpha|\leq 3$, $D^{\alpha}M(x,t)$, $t\geq 0$, $x \in \R^d $, is a family of continuous local martingales with quadratic variation
  \[
    \left[ D^{\alpha}M_i(x,\cdot ) ,D^{\beta}M_j(y, \cdot )  \right]_t =\int_{ 0 }^{ t } D^{\alpha}_xD^{\beta}_y\tilde{a}_{i,j}(s,x,y)ds, \quad t\geq 0,
  \]
  for any  $x,y \in \R^d$, $i,j \in [d]$ and $\alpha$, $\beta$ with $|\alpha|\leq 3$, $|\beta|\leq 3$, by \cite[Theorem~3.1.3]{Kunita:1990}. 

  We consider the following family of linear stochastic transport equations written in integral form
  \begin{equation} 
  \label{equ_stochastic_transport_equation}
  \psi_k(t,x)=x_k- \int_{ 0 }^{ t } \nabla \psi_k(s,x) \cdot  M(x,\circ ds), \quad t\geq 0,\ \ x \in \R^d , \ \ k \in [d], 
  \end{equation}
  where the Stratonovich integral was defined in \cite[Section~3.2]{Kunita:1990}. Using the connection between the \Ito and Stratonovich integrals (see \cite[Theorem~3.2.5]{Kunita:1990}), one gets 
  \begin{align}
    \int_{ 0 }^{ t } \nabla \psi_k(s,x) \cdot M (x,\circ ds)&=  \int_{ 0 }^{ t } \nabla \psi_k(s,x) \cdot M(x, ds)+ \frac{1}{ 2 } \sum_{ i=1 }^{ d } \left[ \int_{ 0 }^{ \cdot  } M_i(x,ds),\partial_i \psi_k(\cdot ,x)  \right]_t\\
    &= \int_{ 0 }^{ t } \nabla \psi_k(s,x)\cdot g(s,x, \theta) W(d \theta,ds)\\
    &- \frac{1}{ 2 }\sum_{ i=1 }^{ d } \left[ M_i(x, \cdot ),\partial_i \int_{ 0 }^{ \cdot  } \nabla \psi_k(s,x) \cdot M(x,ds)  \right]_t.
  \end{align}
  In order to compute the quadratic variation on the right hand side of the above expression, we will use \cite[Theorem~3.1.3]{Kunita:1990}. Thus, for $x,y \in \R^d $
  \begin{align}
    &\left[ M_i(y,\cdot ), \partial_i \int_{ 0 }^{ \cdot  } \nabla \psi_k(s,x)\cdot M(x,ds)  \right]_t= \frac{\partial }{\partial x_i}\left[ M_i(y,\cdot ),\int_{ 0 }^{ \cdot  } \nabla \psi_k(s,x) \cdot M(x,ds)  \right]\\
    &\qquad\qquad= \frac{\partial }{\partial x_i} \sum_{ j=1 }^{ d } \int_{ 0 }^{ t } \partial_j \psi_k(s,x)\int_{ \Theta }  g_j(s,x,\theta)g_j(s,y,\theta)\m(d \theta)ds\\
    &\qquad\qquad=\sum_{ j=1 }^{ d } \int_{ 0 }^{ t } \frac{\partial }{\partial x_i}\left( \partial_j \psi_k(s,x) \tilde{a}_{i,j}(s,x,y) \right)ds.
  \end{align}
  Therefore, equality~\eqref{equ_stochastic_transport_equation} can be rewritten in \Itos form as follows 
  \begin{equation} 
  \label{equ_stochastic_transport_equation_in_ito_form}
  \begin{split} 
    \psi_k(t,x)&= x_k- \int_{ 0 }^{ t } \int_{ \Theta }   \nabla \psi_k(s,x) \cdot g(s,x,\theta) W(d \theta,ds)\\
    &+ \frac{1}{ 2 } \int_{ 0 }^{ t }\int_{ \R^d  }   \nabla_x \cdot \left( \tilde{a}(s,x,y) \cdot \nabla \psi_k(s,x) \right) \delta_{x}(dy)ds
  \end{split}
  \end{equation}
  for any $t\geq 0$, $x \in \R^d $ and $k \in [d]$.

  \begin{proposition} 
  \label{pro_existence_of_solutions_to_the_stochastic_transport_equation}
  Under the assumptions of Theorem~\ref{the_well_posedness_of_dk_equation_in_general_case}, for each $k \in [d]$ there exist $\delta' \in (0,\delta)$ and an $(\F_t)$-adapted continuous $\Cf^{3, \delta'}$-valued process $\psi_k(t,\cdot )$, $t\geq 0$, that satisfies~\eqref{equ_stochastic_transport_equation} (and, therefore,~\eqref{equ_stochastic_transport_equation_in_ito_form}). Moreover, a.s. for every $t\geq 0$ the map $\psi(t,\cdot )=(\psi_1(t,\cdot ),\dots,\psi_d(t, \cdot )):\R^d \to \R^d  $ is invertible and $\varphi(t,\cdot ):=\psi^{-1}(t,\cdot )$ is an $(\F_t)$-adapted continuous $\Cf^{3,\delta'}(\R^d )$-valued stochastic process that satisfies the equation
  \begin{align}
    \varphi(t,x)&= x+ \int_{ 0 }^{ t } M(\varphi(s,x),\circ ds)
  \end{align}
  for every $t\geq 0$ and $x \in \R^d $.
  \end{proposition}

  We note that \cite[Theorem~3.2.5]{Kunita:1990} yields that the inverse flow $\varphi(t,\cdot )$, $t\geq 0$, to $\psi(t,\cdot )$, $t\geq 0$, solves the equation 
  \begin{equation} 
  \label{equ_ito_form_for_the_equation_for_phi}
    \begin{split} 
      \varphi(t,x)&= x+ \int_{ 0 }^{ t } \int_{ \Theta }   g(s,\varphi(s,x),\theta)W(d \theta,ds)\\
      &+ \frac{1}{ 2 } \int_{ 0 }^{ t }  (\nabla_x \cdot  \tilde{a})(s,\varphi(s,x),\varphi(s,x))ds , \quad t\geq 0,\ \ x \in \R^d ,
    \end{split}
  \end{equation}
  in \Ito form, where $\nabla_x \cdot \tilde{a}(s,x ,y )=\left( \sum_{ j=1 }^{ d }\frac{\partial }{\partial x_j} \tilde{a}_{i,j}(s,x,y) \right)_{i \in [d]}$. Indeed,
    \[
      \int_{ 0 }^{ t } M_k(\varphi(s,x), \circ ds) = \int_{ 0 }^{ t } M_k(\varphi(s,x),ds)+ \frac{1}{ 2 }\sum_{ i=1 }^{ d }  \left[ \partial_i \int_{ 0 }^{ \cdot  }  M_k(\varphi(s,x),ds),\varphi_i(\cdot ,x) \right]_t.
    \]
    By \cite[Theorem~3.1.3]{Kunita:1990}, we can compute for $x,y \in \R^d $
    \begin{align}
      \left[ \frac{\partial }{\partial x_i}\int_{ 0 }^{ \cdot }  M_k(\varphi(s,x),ds), \varphi_i(\cdot ,y) \right]_t&= \frac{\partial }{\partial x_i} \int_{ 0 }^{ t } \int_{ \Theta }   g_k(s,x,\theta)g_i(s,y, \theta)\m(d \theta)ds\\
      &= \int_{ 0 }^{ t } \frac{\partial }{\partial x_i}\tilde{a}_{k,i}(s,x,y)ds.
    \end{align}
    This implies the equivalence between the equations in \Ito and Stratonovich form.

    \begin{proof}[Proof of Proposition~\ref{pro_existence_of_solutions_to_the_stochastic_transport_equation}] 
    The existence of $\psi_k$, $k \in [d]$, follows from \cite[Theorem~6.1.8]{Kunita:1990}. We observe that
    \begin{align}
      \varphi(t,x)=x+\int_{ 0 }^{ t } M(\varphi(s,x),\circ ds) , \quad t\geq 0,\ \ x \in \R^d ,
    \end{align}
    is the stochastic characteristic equation for the SPDE~\eqref{equ_stochastic_transport_equation}. By~\eqref{equ_ito_form_for_the_equation_for_phi} and \cite[Theorems~3.4.6, 4.7.3]{Kunita:1990}, it has a unique continuous $\Cf^{3,\delta'}(\R^d )$-valued solution which is also a stochastic flow of $\Cf^3$-diffeomorphisms. Using \cite[Theorem~6.1.2]{Kunita:1990}, we can conclude that $\psi_k(t,\cdot )$ is the $k$-th coordinate of $\varphi^{-1}(t,\cdot )$ for any $t\geq 0$. This immediately implies the equality $\psi^{-1}(t,\cdot )=\varphi(t,\cdot )$ which ends the proof of the proposition.
  \end{proof}

  We will next consider for every $t\geq 0$ the following probability measure $\rho_t=\mu_t\circ \psi^{-1}(t,\cdot )$ on $\R^d $. It is easy to see that $\rho_t$, $t\geq 0$, is a continuous process in $\cP(\R^d )$. Let us show that this process (locally) satisfies a continuity equation with random coefficients. 

  \begin{lemma} 
  \label{lem_equation_for_transformed_measure_mu}
  Let the coefficients of~\eqref{equ_mean_field_equation_with_corr_noise} satisfy Assumptions~\ref{ass_basic_assumption},~\ref{ass_lipshitz_continuity} and~\ref{ass_smoothness_of_coefficients}. Then the measure-valued process $\rho_t$, $t\geq 0$, defined above, is a solution to the equation 
  \begin{equation} 
  \label{equ_equation_for_rho}
    d\rho_t=-\nabla(b(t,\cdot )\rho_t)dt, \quad \rho_0=\mui,
  \end{equation}
  that is, for every $\varphi \in \Cf_c^2(\R^d )$ a.s. the equality
  \begin{equation} 
  \label{equ_integral_equality_for_rho}
    \langle \rho_t , \varphi \rangle=\langle \mui , \varphi \rangle+ \int_{ 0 }^{ t } \left\langle \nabla \varphi \cdot b(s,\cdot ) , \rho_s  \right\rangle ds, \quad t\geq 0, 
  \end{equation}
  holds, where $b(t,x)=\tilde{b}(t,\psi^{-1}(t,x))$ and
  \[
    \tilde{b}_k(t,x)=\nabla \psi_k(t,x) \cdot v(t,x)- \frac{1}{ 2 } \int_{ \R^d  }   \nabla_x \cdot \left( \tilde{a}(t,x,y) \cdot \nabla \psi_k(t,x) \right) \delta_{x}(dy)
  \]
  for all $t\geq 0$, $x \in \R^d $, $k \in [d]$.
  \end{lemma}

  \begin{proof} 
    Let $\varphi \in \Cf_c^{2}(\R^d )$. We note that $\langle \varphi , \rho_t \rangle=\langle \varphi \circ \psi(t,\cdot ) , \mu_t \rangle$ for every $t\geq 0$, by the definition of the measure $\rho_t$. Hence, we first apply \Itos formula to $\varphi \circ \psi(t,x)$, $t\geq 0$, for each $x \in \R^d $. One gets for $t\geq 0$
  \begin{align}
    &\varphi\circ \psi(t,x)=  \varphi \left( \psi(t,x) \right)= \varphi(x)\\
    &\qquad- \sum_{ k=1 }^{ d } \int_{ 0 }^{ t } \int_{ \Theta }   \left( \partial_k \varphi \circ \psi\right)(s,x)  \left[ \nabla \psi_k(s,x) \cdot g(s,x,\theta) \right]W(d \theta,ds)  \\
    &\qquad+ \frac{1}{ 2 }\sum_{ k=1 }^{ d } \int_{ 0 }^{ t }  \left( \partial_k \varphi \circ \psi\right)(s,x)  \int_{ \R^d  }   \nabla_x \cdot \left( \tilde{a}(t,x,y) \cdot \nabla \psi_k(t,x) \right) \delta_{x}(dy)ds\\
    &\qquad+ \frac{1}{ 2 } \sum_{ k,l=1 }^{ d } \int_{ 0 }^{ t }\left( \partial^{2}_{k,l} \varphi \circ \psi\right)(s,x)\left[ \left( \nabla \psi_k(s,x)\otimes \nabla \psi_l(s,x)\right): a(s,x)\right]ds.
  \end{align}
  
  Next, we consider a non-negative function $\kappa \in \Cf_c^2(\R^d )$ such that $\int_{ \R^d  }   \kappa(x)dx=1 $, and set $\delta^{\eps}_x(y)= \frac{1}{ \eps^d }\kappa\left( \frac{1}{ \eps }(x-y)\right)$, $y \in \R^d $, for each $x \in \R^d $ and $\eps>0$. Integrating $\delta^{\eps}_x$ by $\mu_t$ and denoting $\mu_t^{\eps}(x)=\langle \delta^{\eps}_x ,\mu_t  \rangle$, we get for every $x \in \R^d $ a.s. 
  \begin{align}
    \mu_t^{\eps}(x)&= \mu_0^{\eps}(x)+ \frac{1}{ 2 }\int_{ 0 }^{ t } \left\langle D^2\delta^{\eps}_x:a(s,\cdot ) , \mu_s \right\rangle ds\\
    &+\int_{ 0 }^{ t } \left\langle \nabla\delta_x^{\eps} \cdot v(s, \cdot ) , \mu_s \right\rangle ds + \int_{ 0 }^{ t } \int_{ \Theta }   \left\langle \nabla \delta_x^{\eps} \cdot g(s, \cdot ,\theta) , \mu_s \right\rangle W(d \theta,ds)  
  \end{align}
  for all $t\geq 0$. Using now the expressions for $\varphi\circ \psi(t,x)$ and $\mu_t^{\eps}(x)$ and \Itos formula, we obtain for every $x \in \R^d $ a.s. 
  \begin{align}
    \varphi\circ \psi&(t,x)\mu^{\eps}_t(x)=  \varphi(x)\mu^{\eps}_0(x)\\
    &- \sum_{ k=1 }^{ d } \int_{ 0 }^{ t } \int_{ \Theta } \mu_s^{\eps}(x)  \left( \partial_k \varphi \circ \psi\right)(s,x)\left[ \nabla \psi_k(s,x) \cdot g(s,x,\theta) \right]W(d \theta,ds)  \\
    &+ \frac{1}{ 2 }\sum_{ k=1 }^{ d } \int_{ 0 }^{ t } \mu_s^{\eps}(x) \left( \partial_k \varphi \circ \psi\right)(s,x)  \int_{ \R^d  }   \nabla_x \cdot \left( \tilde{a}(s,x,y) \cdot \nabla \psi_k(s,x) \right) \delta_{x}(dy)ds\\
    &+ \frac{1}{ 2 } \sum_{ k,l=1 }^{ d } \int_{ 0 }^{ t }\mu_s^{\eps}(x)\left( \partial^{2}_{k,l} \varphi \circ \psi\right)(s,x)\left[ \left( \nabla \psi_k(s,x)\otimes \nabla \psi_l(s,x)\right): a(s,x)\right]ds\\
    &+\frac{1}{ 2 }\int_{ 0 }^{ t } \varphi\circ \psi(s,x)\left\langle D^2\delta^{\eps}_x:a(s,\cdot ) , \mu_s \right\rangle ds +\int_{ 0 }^{ t } \varphi\circ \psi(s,x)\left\langle \nabla\delta_x^{\eps} \cdot v(s, \cdot ) , \mu_s \right\rangle ds\\ 
    &+ \int_{ 0 }^{ t } \int_{ \Theta } \varphi \circ \psi(s,x) \left\langle \nabla \delta_x^{\eps} \cdot g(s, \cdot ,\theta) , \mu_s \right\rangle W(d \theta,ds)\\
    &- \sum_{ k=1 }^{ d } \int_{ 0 }^{ t } \int_{ \Theta }   \left\langle \nabla\delta_x^{\eps} \cdot g(s,\cdot, \theta) , \mu_s \right\rangle   \left( \partial_k \varphi \circ \psi\right)(s,x)\left[ \nabla \psi_k(s,x) \cdot g(s,x,\theta) \right]\m(d \theta) ds
  \end{align}
  for all $t\geq 0$. Note that for every $t\geq 0$ a.s. there exists a compact $K \subset \R^d $ such that $\supp \varphi \circ \psi(s, \cdot ) \in K$ for any $s \in [0,t]$. Indeed, one can take $K=\{ \psi^{-1}(s,x):\ s \in [0,t],\ x \in \supp \varphi \}$, which is a compact set as the image of $[0,t]\times \supp \varphi$ under the continuous map $(s,x)\mapsto\psi^{-1}(s,x)$ (for the continuity of $\psi^{-1}$ see Proposition~\ref{pro_existence_of_solutions_to_the_stochastic_transport_equation}). Therefore, we may integrate the above expression with respect to the Lebesgue measure $dx$ over $\R^d $. Then taking $\eps \to 0$, we get
  \[
    \int_{ \R^d  }   \varphi\circ \psi(t,x)\mu_t^{\eps}(x)dx \to \left\langle \varphi\circ \psi(t,\cdot ) ,\mu_t  \right\rangle \quad \mbox{a.s.}
  \]
  and 
  \[
    \int_{ \R^d  }   \varphi(x)\mu_0^{\eps}(x)dx \to \left\langle \varphi , \mui \right\rangle \quad \mbox{a.s.}
  \]
  Next using the equalities $\nabla_y\delta_x^{\eps}(y)=-\nabla_x\delta_x^{\eps}(y)$, $D^2_y\delta_x^{\eps}(y)=D^2_x\delta_x^{\eps}(y)$, the integration by parts formula and Fubini's theorem, we can conclude that the sum of all terms consisting of the integrals $\int_{ 0 }^{ t } (\dots)ds $ in the expression for $\int_{ \R^d  }   \varphi\circ \psi(t,x)\mu_t^{\eps}(x)dx$ converges a.s. to 
  \begin{align}
    I_1:&= \frac{1}{ 2 }\sum_{ k=1 }^{ d } \int_{ 0 }^{ t } \left\langle \left( \partial_k \varphi \circ \psi\right)(s,x )  \int_{ \R^d  }   \nabla_x \cdot \left( \tilde{a}(s,x ,y) \cdot \nabla \psi_k(s,x ) \right) \delta_{x }(dy),\mu_s(dx)\right \rangle ds\\
    &+ \frac{1}{ 2 } \sum_{ k,l=1 }^{ d } \int_{ 0 }^{ t }\left\langle \left( \partial^{2}_{k,l} \varphi \circ \psi\right)(s,x)\left[ \left( \nabla \psi_k(s,x)\otimes \nabla \psi_l(s,x)\right): a(s,x)\right],\mu_s(dx)\right\rangle ds\\
    &+\frac{1}{ 2 }\int_{ 0 }^{ t }\left\langle  D^2(\varphi\circ \psi)(s,x):a(s,\cdot ) , \mu_s \right\rangle ds +\int_{ 0 }^{ t } \left\langle \nabla(\varphi\circ \psi)(s,x)\cdot v(s, \cdot ) , \mu_s \right\rangle ds\\ 
    &- \sum_{ k,i,j=1 }^{ d } \int_{ 0 }^{ t } \int_{ \R^d  }      \left\langle \frac{\partial }{\partial x_i}\left[ \left( \partial_k \varphi \circ \psi\right)(s,x) \partial_j \psi_k(s,x)\tilde{a}_{i,j}(s,x,y) \right],\mu_s(dx)\right\rangle \delta_x(dy)ds.
  \end{align}
  Taking into account that the second and the third integrals are the same and that the last term can be rewritten as 
  \begin{align}
    &-\sum_{ k,l=1 }^{ d } \int_{ 0 }^{ t }\left\langle  \left( \partial^2_{k,l} \varphi \circ \psi\right)(s,x) \partial_j \left[\left(\nabla\psi_k(s,x)\otimes\nabla \psi_l(s,x)\right):\tilde{a}_{i,j}(s,x,x) \right],\mu_s(dx)\right\rangle ds\\
    &-\sum_{ k=1 }^{ d } \int_{ 0 }^{ t } \left\langle \left( \partial_k \varphi \circ \psi\right)(s,x) \int_{ \R^d  }   \nabla_x \cdot \left( \tilde{a}(s,x,y) \cdot \nabla \psi_k(s,x) \right)\delta_x(dy),\mu_s(dx)\right\rangle ds,
  \end{align}
  we obtain
  \begin{align}
    I_1&= -\frac{1}{ 2 }\sum_{ k=1 }^{ d } \int_{ 0 }^{ t } \left\langle \left( \partial_k \varphi \circ \psi\right)(s,x )  \int_{ \R^d  }   \nabla_x \cdot \left( \tilde{a}(s,x ,y) \cdot \nabla \psi_k(s,x ) \right) \delta_{x }(dy),\mu_s(dx)\right \rangle ds\\
    &+\sum_{ k=1 }^{ d }\int_{ 0 }^{ t }  \left\langle (\partial_k\varphi\circ \psi)(s,\cdot ) \nabla\psi_k(s,\cdot )\cdot v(s, \cdot ) , \mu_s \right\rangle ds\\
    &= \int_{ 0 }^{ t } \left\langle \left(\nabla \varphi\circ \psi\right)(s,\cdot ) \cdot b(s, \psi(s,\cdot )) , \mu_s \right\rangle ds.
  \end{align}
  We now show that the difference of stochastic integrals converges to zero. Set 
  \begin{align}
    F_{\eps}(s,x,\theta)&=  \varphi\circ \psi(s,x)\left\langle \nabla \delta_x^{\eps} \cdot g(s, \cdot ,\theta) , \mu_s \right\rangle\\
    &-\sum_{ k=1 }^{ d } \mu_s^{\eps}(x)  \left( \partial_k \varphi \circ \psi\right)(s,x)\left[ \nabla \psi_k(s,x) \cdot g(s,x,\theta) \right]\\
    &= \int_{ \R^d  }   \varphi\circ \psi(s,x) \nabla_y\delta_x^{\eps}(y)\cdot g(s,y, \theta)\mu_s(dy)\\
    &-\int_{ \R^d  }   \delta_x^{\eps}(y)\nabla (\varphi\circ \psi)(s,x) \cdot g(s,x,\theta)\mu_s(dy)  \\
    &= \int_{ \R^d  }    \left(\varphi\circ \psi(s,x) \nabla_y\delta_x^{\eps}(y)\cdot g(s,y, \theta)-\delta_x^{\eps}(y)\nabla (\varphi\circ \psi)(s,x) \cdot g(s,x,\theta)\right)\mu_s(dy)  
  \end{align}
  for all $s\geq 0$, $x \in \R^d $ and $\theta \in \Theta$. Using stochastic Fubini's theorem, one can see that 
  \begin{equation} 
  \label{equ_stochastic_integral_of_f}
    \int_{ \R^d  }   \left(\int_{ 0 }^{ t } \int_{ \Theta }   F_{\eps}(s,x,\theta)W(d \theta,ds)\right)dx= \int_{ 0 }^{ t } \int_{ \Theta }  \left( \int_{ \R^d  }   F_{\eps}(s,x, \theta) dx \right) W(d \theta,ds).      
  \end{equation}
  Moreover, 
  \begin{align}
    &\int_{ 0 }^{ t } \int_{ \Theta }   \left( \int_{ \R^d  }   F_{\eps}(s,x,\theta)dx  \right)^{2}\m(d \theta)ds\\
    &\qquad= \int_{ 0 }^{ t } \int_{ \R ^{4d} } \varphi\circ \psi(s,x) \varphi\circ \psi(s,\tilde x)\nabla_y\delta_x^{\eps}(y)\otimes\nabla_{\tilde{y}}\delta_{\tilde{x}}^{\eps}(\tilde y): \tilde{a}(s,y, \tilde{y})\mu_s(dy)\mu_s(d \tilde{y})dxd \tilde{x}ds\\    
    &\qquad+\int_{ 0 }^{ t } \int_{ \R^{4d}  }   \delta_x^{\eps}(y)\delta_{\tilde{x}}^{\eps}(\tilde y)\nabla (\varphi\circ \psi)(s,x)\otimes\nabla (\varphi\circ \psi)(s,\tilde x): \tilde{a}(s,x,\tilde{x})\mu_s(dy)\mu_s(d\tilde y)dx d \tilde{x} ds\\
    &\qquad-2 \int_{ 0 }^{ t } \int_{ \R ^{4d} }   \varphi\circ \psi(s,x)\delta_{\tilde{x}}^{\eps}(\tilde{y})\nabla_y \delta_x^{\eps}(y)\otimes \nabla\left( \varphi\circ \psi(s, \tilde{x}): \tilde{a}(s,y,\tilde{x}) \right)\mu_s(dy)\mu_s(d \tilde{y})dxd \tilde{x}ds  
  \end{align}
  for every $t\geq 0$. After the integration by parts, similarly as before, we obtain 
  \[
    \int_{ 0 }^{ t } \int_{ \Theta }   \left( \int_{ \R^d  }   F_{\eps}(s,x,\theta)dx  \right)^{2}\m(d \theta)ds \to 0 \quad \mbox{a.s.}
  \]
  as $\eps \to 0$. This simply implies the convergence of the right hand side of~\eqref{equ_stochastic_integral_of_f} to zero in probability as $\eps \to 0$, by, e.g., \cite[Theorem~II.7.2']{Ikeda:1989}. Summarizing our computations, we get that for every $t\geq 0$ a.s. 
  \[
    \left\langle \varphi\circ \psi(t, \cdot ) , \mu_t \right\rangle=\left\langle \varphi , \mui \right\rangle+ \int_{ 0 }^{ t } \left\langle \left( \nabla \varphi\circ \psi \right)(s, \cdot )\cdot b(s,\psi(s,\cdot)) , \mu_s \right\rangle ds.
  \]
  This directly implies that the process $\rho_t$, $t\geq 0$, solves the equation~\eqref{equ_equation_for_rho}. 
  \end{proof}

  Note that we cannot control the growth of the coefficient $b$ in the PDE~\eqref{equ_equation_for_rho} in the spatial variable. Therefore, the known superposition principle, e.g. from~\cite{Ambrosio:2008,Trevisan:2016,DiPerna:1989}, cannot be applied. However, we can construct a superposition solution to~\eqref{equ_equation_for_rho} and a solution to its dual equation precisely, and then use the duality principle to prove that only the superposition solution solves~\eqref{equ_equation_for_rho}. 

  Let $Y(u,t)$, $u \in \R^d $, $t\geq 0$, be a unique solution to~\eqref{equ_equation_with_interaction} with the frozen coefficients $v$ and $g$. In particular, for every $u \in \R^d $ a.s.
  \begin{equation} 
  \label{equ_equation_for_y_with_frozen_coefficients}
    Y(u,t)=u+\int_{ 0 }^{ t } v(s,Y(u,s))ds+ \int_{ 0 }^{ t } \int_{ \Theta }   g(s,Y(u,s),\theta)W(d \theta,ds), \quad t\geq 0.   
  \end{equation}
  By Assumptions~\ref{ass_lipshitz_continuity},~\ref{ass_smoothness_of_coefficients} and \cite[Theorem~4.6.5]{Kunita:1990}, $Y$ has a modification (also denoted by $Y$) which is a continuous process in the space of $\Cf^{1,\delta'}(\R^d )$-diffeomorphisms for every $\delta' \in (0,\delta)$, i.e., a.s. $Y(\cdot,t )$, $t\geq 0$, is a continuous $\Cf^{1,\delta'}(\R^d )$-valued process and a.s. for every $t\geq 0$ there exists the inverse map $Y^{-1}(\cdot,t ):\R^d \to \R^d $ to $Y(\cdot,t )$ which is a $\Cf^{1, \delta'}$-valued continuous process. We also set 
\begin{equation} 
  \label{equ_definition_of_the_process_z_by_psi}
  Z(\cdot,t )=\psi(t,Y(\cdot,t )), \quad t\geq 0,
\end{equation}
  which is a continuous process in the space of $\Cf^{1, \delta'}$-diffeomeorphisms for some $\delta \in (0,\delta)$, according to Proposition~\ref{pro_existence_of_solutions_to_the_stochastic_transport_equation}.

  \begin{lemma} 
  \label{lem_uniqueness_of_solutions_to_equation_for_rho}
  Under the assumptions of Lemma~\ref{lem_equation_for_transformed_measure_mu}, the process $\bar\rho_t=\mui\circ Z^{-1}(\cdot,t )$, $t\geq 0$, is the unique solution to~\eqref{equ_equation_for_rho}.
  \end{lemma}

  \begin{proof} 
    We first apply the generalised \Itos formula from \cite[Theorem~3.3.1]{Kunita:1990} to $Z(u,t)=\psi(t,Y(u,t))$, $t\geq 0$. We get for every $u \in \R^d $ a.s.
    \[
      Z(u,t)=u+\int_{ 0 }^{ t }\tilde{b}(s,Y(u,s))ds=u+\int_{ 0 }^{ t } b(s,Z(u,s))ds , \quad t\geq 0. 
    \]
    where $\tilde{b}$ and $b$ were defined in Lemma~\ref{lem_equation_for_transformed_measure_mu}.  Let $\varphi \in \Cf_c^{2}(\R^d )$. By the usual chain-rule, one obtains for each $u \in \R^d $ a.s.
    \begin{align}
      \varphi(Z(u,t))&= \varphi(u)+ \int_{ 0 }^{ t } (\nabla \varphi)\left( Z(u,s) \right) \cdot b(s,Z(u,s))ds  , \quad t\geq 0.
    \end{align}
    Then, integrating the obtained expression with respect to $\mui$ yields that the process $\bar\rho_t$, $t\geq 0$, satisfies the equality 
    \[
      \left\langle \varphi , \bar\rho_t \right\rangle=\langle \varphi , \mui \rangle+\int_{ 0 }^{ t } \left\langle \nabla \varphi \cdot b(s, \cdot ) , \bar\rho_s \right\rangle ds, \quad t\geq 0, 
    \]
    a.s. This implies that $\bar \rho_t$, $t\geq 0$, is a solution to~\eqref{equ_equation_for_rho}.

    We next prove the uniqueness of solutions to~\eqref{equ_equation_for_rho}. Let $\rho_t$, $t\geq 0$, be a solution to~\eqref{equ_equation_for_rho}.  For a fixed function $\varphi \in \Cf_c^{2}(\R^d )$ define 
    \[
      \gamma(t,\cdot )=\varphi(Z^{-1}(\cdot,t )), \quad t\geq 0, 
    \]
    which is a $\Cf^{3,\delta'}(\R^d )$-valued continuous process.  Then similarly to the proof of \cite[Lemma~6.1.1]{Kunita:1990}, one can show that for every $x \in \R^d $ a.s. 
    \[
      \gamma(t,x )= \varphi(x)-\int_{ 0 }^{ t } \nabla \gamma(s,x) \cdot b (s,x)ds, \quad t\geq 0.
    \]
    Repeating now the computation from the proof of Lemma~\ref{lem_equation_for_transformed_measure_mu}, we get a.s.
    \begin{align}
      \langle \gamma(t,\cdot ) , \rho_t \rangle&= \langle \gamma(0,\cdot ) , \varphi\rangle+\int_{ 0 }^{ t } \langle \nabla\gamma(s,\cdot ) \cdot b(s,\cdot ) , \rho_s \rangle ds- \int_{ 0 }^{ t } \left\langle \nabla \gamma(s,\cdot ) \cdot b(s,x) , \rho_s \right\rangle ds\\
      &= \langle \mui , \varphi \rangle  , \quad t\geq 0.
    \end{align}
    Thus, from the definition of $\gamma$ it follows that a.s.
    \[
      \left\langle \varphi , \mui \right\rangle=\int_{ \R^d  }   \gamma(t,x)\rho_t(dx)= \int_{ \R^d  }   \varphi(Z^{-1}(x,t))\rho_t(dx)= \int_{ \R^d  }   \varphi(u)\tilde{\rho}_t(du),  
    \]
    where $\tilde{\rho}_t(A)= \rho_t\{ Z(u,t):\ u \in A\}$, $A \in \B(\R^d )$. It is easy to see that $\tilde\rho_t$, $t\geq 0$, is a continuous process in $\cP(\R^d )$. Consequently, the above expression, which holds for every $\varphi \in \Cf_c^{2}(\R^d )$, yields that $\tilde\rho_t=\mui$, $t\geq 0$, a.s. Hence, $\rho_t=\mui\circ Z^{-1}(\cdot,t )=\bar\rho_t$, $t\geq 0$, a.s. This completes the proof of the lemma.
  \end{proof}

  We now ready to prove the well-posedness of the stochastic mean-field equation for arbitrary initial conditions.

  \begin{proof}[Proof of Theorem~\ref{the_well_posedness_of_dk_equation_in_general_case}] 
    Let $\mu_t$, $t\geq 0$, be an arbitrary solution to the stochastic mean-field equation~\eqref{equ_mean_field_equation_with_corr_noise}. Let also $\psi$, $\rho$, $Y$ and $Z$ be defined by~\eqref{equ_stochastic_transport_equation},~\eqref{equ_equation_for_rho},~\eqref{equ_equation_for_y_with_frozen_coefficients}, and~\eqref{equ_definition_of_the_process_z_by_psi}, respectively, where the process $\mu_t$, $t\geq 0$, is frozen in the coefficients $V,A,\tilde{A}$ and $G$, i.e, $v(t,x)=V(t,x,\mu_t)$, $a(t,x)=A(t,x,\mu_t)$, $\tilde{a}(t,x,y)=\tilde{A}(t,x,y,\mu_t)$ and $g(t,x,\theta)=G(t,x,\mu_t,\theta)$. We recall that $\rho_t=\mu_t\circ \psi^{-1}(t,\cdot )$, $t\geq 0$. Then by Lemmas~\ref{lem_equation_for_transformed_measure_mu} and~\ref{lem_uniqueness_of_solutions_to_equation_for_rho}, we obtain that a.s.
    \[
      \rho_t=\mui\circ Z^{-1}(\cdot,t )=\left( \mui\circ Y^{-1}(\cdot,t ) \right)\circ \psi^{-1}(t,\cdot ), \quad t\geq 0.
    \]
    Using that $\psi(t,\cdot )$ is a bijection, we conclude that $\mu_t=\mui\circ Y^{-1}(\cdot,t )$, $t\geq 0$, a.s. Hence, $Y$ is a solution to the SDE with interaction~\eqref{equ_equation_with_interaction} due to the definition of $Y$ and the fact that $\bar\mu_t=\mui\circ Y^{-1}(\cdot,t )=\mu_t$ for all $t\geq 0$. Moreover, $Y$ is also the unique solution, according to Theorem~\ref{the_well_posedness_of_sde_with_interaction}. This yields that $\mu_t$, $t\geq 0$, is a superposition solution to~\eqref{equ_mean_field_equation_with_corr_noise}. This completes the proof of the theorem since the superposition principle implies the uniqueness of the stochastic mean-field equation, by Corollary~\ref{cor_superposition_principle}.
  \end{proof}

\section{Limit theorems for the stochastic mean-field equation}
\label{sec:lln_and_clt_for_the_mean_field_equation}

The aim of this section is to prove an analog of the law of large numbers (LLN) and the central limit theorem (CLT) for solutions to the stochastic mean-field equation~\eqref{equ_mean_field_equation_with_corr_noise}. More precisely, we provide a rate of convergence of the superposition solution $\mu_t^{\eps}$, $t\geq 0$, to 
\begin{equation} 
  \label{equ_mean_field_equation_with_vanishing_noise}
  \begin{split} 
    d\mu_t^{\eps}&=  \frac{ \eps }{ 2 }D^2:(A(t, \cdot ,\mu_t^{\eps})\mu_t^{\eps})dt-\nabla \cdot \left( V(t,\cdot,\mu_t^{\eps})\mu_t^{\eps} \right)dt\\
    &-\sqrt{ \eps }\int_{ \Theta }   \nabla\cdot \left( G(t, \cdot ,\mu_t^{\eps}, \theta)\mu_t^{\eps} \right)W(d \theta,dt)
  \end{split}
\end{equation}
started from $\mu_0^{\eps} \in \cP_2(\R^d )$ to the superposition solution $\mu^0_t$, $t\geq 0$, to the PDE
\begin{equation} 
  \label{equ_mean_field_equation_without_noise}
  d\mu^0_t=-\nabla \cdot \left( V(t,\cdot,\mu^0_t)\mu^0_t \right)dt
\end{equation}
started from $\mu_0^0\in\cP_2(\R^d)$. We also show that the fluctuation field 
\begin{equation} 
  \label{equ_definition_of_eta_eps}
  \eta^{\eps}_t= \frac{1}{ \sqrt{ \eps } }\left( \mu_t^{\eps}-\mu^0_t \right), \quad t\geq 0,
\end{equation}
converges to a Gaussian process $\eta^0_t$, $t\geq 0$, which is a solution to the linear SPDE
\begin{equation} 
  \label{equ_equation_for_eta0}
  \begin{split}
    d\eta^0_t&= -\nabla \cdot \left( V(t,\cdot ,\mu_t^0)\eta_t^0+ \langle \tilde{V}(t,x ,\cdot ) , \eta^0_t \rangle_0 \mu_t^0(dx) \right)dt\\
    &- \int_{ \Theta }   \nabla \cdot \left( G(t,\cdot ,\mu_t^0,\theta)\mu_t^0 \right)W(d \theta,dt),
  \end{split}
\end{equation}
and estimate the speed of the convergence. 

\subsection{Law of large numbers}
\label{sub:law_of_large_numbers}

In this section, we will prove the following theorem.
\begin{theorem}[LLN for the stochastic mean-field equation] 
  \label{the_lln}
  Let the coefficients $A,V,G$ of the stochastic mean-field equation~\eqref{equ_mean_field_equation_with_vanishing_noise} satisfy Assumptions~\ref{ass_basic_assumption} and~\ref{ass_lipshitz_continuity}. Furthermore, let $\mu^{\eps}_t$, $t\geq 0$, be a superposition solution to the SPDE~\eqref{equ_mean_field_equation_with_vanishing_noise} started from $\mui^{\eps} \in \cP_2(\R^d )$ for each $\eps>0$ and $\mu_t^0$, $t\geq 0$, be a superposition solution to the PDE~\eqref{equ_mean_field_equation_without_noise} started from $\mui^0 \in \cP_2(\R^d )$. Then, for every $T>0$ there exists a constant $C>0$ such that 
  \begin{equation} 
  \label{equ_estimate_from_lln}
  \e \sup\limits_{ t \in [0,T] }\W_2^2(\mu_t^{\eps},\mu^0_t)\leq C \left(\eps\left( 1+\langle \phi_2 , \mu^{\eps}_0 \rangle \right)+ \W_2^2(\mu^{\eps}_0,\mu^0_0)\right),
  \end{equation}
  where $\phi_2(x)=|x|^2$, $x \in \R^d $.
\end{theorem}

\begin{proof} 
  The proof is similar to the proof of Theorem~\ref{the_continuous_dependents_of_solutions_to_sde_with_interaction}. Since the processes $\mu_t^{\eps}$, $t\geq 0$, and $\mu_t^0$, $t\geq 0$, are superposition solutions to~\eqref{equ_mean_field_equation_with_vanishing_noise} and~\eqref{equ_mean_field_equation_without_noise}, respectively, we have $\mu_t^{\eps}=\mu_0^{\eps}\circ X_{\eps}(t,\cdot )$, $t\geq 0$, $\eps\geq 0$, where $X_{\eps}$ are solutions to the corresponding SDEs with interaction~\eqref{equ_equation_with_interaction}. Using the Burkholder--Davis--Gundy inequality, Assumption~\ref{ass_lipshitz_continuity} and Remark~\ref{rem_linear_growth_of_coefficients_under_lipschitz_conditions}, we get for each $t \in [0,T]$ and $u,v \in \R^d $
  \begin{align}
    \e \sup\limits_{ s \in [0,t] }&|X_{\eps}(u,s)-X_0(v,s)|^2\leq 3|u-v|^2\\
    &+3T \e \int_{ 0 }^{ t } \left| V(s,X_{\eps}(u,r),\mu_s^{\eps})-V(s,X_0(v,s),\mu^0_s) \right|^2ds\\
    &+3 \eps \e \sup\limits_{ s \in [0,t] }\left| \int_{ 0 }^{ s } \int_{ \Theta }   G(r,X_{\eps}(u,r),\mu_r^{\eps}, \theta)W(d \theta,dr)   \right|^2\\
    &\leq 3|u-v|^2+C \int_{ 0 }^{ t } \left( \e|X_{\eps}(u,s)-X_0(v,s)|^2+\e \W_2^2(\mu_s^{\eps},\mu^0_s) \right)ds\\
    &+  \eps C\e \int_{ 0 }^{ t } \left\||G(s,X_{\eps}(u,s),\mu_s^{\eps})|\right\|_{\m}^2 ds\\     
    &\leq 3|u-v|^2+C \int_{ 0 }^{ t }  \e \sup\limits_{ r \in [0,s] }|X_{\eps}(u,r)-X_0(v,r)|^2ds+ C\int_{ 0 }^{ t }  \e \W_2^2(\mu_s^{\eps},\mu^0_s)ds\\
    &+ \eps 3CL^2 \int_{ 0 }^{ t } \e\left( 1+|X_{\eps}(u,s)|^2+\W_2^2(\mu_s^{\eps},\delta_0) \right) ds, 
  \end{align}
  where $C$ is independent of $u,v,t$ and $\eps$. We note that 
  \[
    \W_2^2(\mu_s^{\eps}, \delta_0)=\int_{ \R^d  }   |x|^2\mu_s^{\eps} (dx)= \int_{ \R^d  }   |X_{\eps}(\tilde{u},s)|^2\mu^{\eps}_0(d \tilde{u}). 
  \]
  Hence, by Gronwall's lemma, we conclude 
  \begin{align}
    \e \sup\limits_{ s \in [0,t] }&|X_{\eps}(u,s)-X_0(v,s)|^2\leq C |u-v|^2+C \int_{ 0 }^{ t } \e\W_2^2(\mu_s^{\eps},\mu^0_s)ds\\
    &+ \eps C \int_{ 0 }^{ t } \e\left( 1+|X_{\eps}(u,s)|^2+ \int_{ \R^d  }   |X_{\eps}(\tilde{u},s)|^2\mu_0^{\eps}(d \tilde{u})  \right) ds\\
    &\leq C |u-v|^2+C \int_{ 0 }^{ t } \e\W_2^2(\mu_s^{\eps},\mu^0_s)ds+ \eps C \left( 1+|u|^2+ \langle \phi_2 , \mu_0^{\eps} \rangle  \right) ds
  \end{align}
  for all $t \in [0,T]$, $u,v \in \R^d $ and $\eps>0$, where $C$ depends only on $L$ and $T$. Note that in the last step of the inequality, we have used Theorem~\ref{the_well_posedness_of_sde_with_interaction}.

  For fixed $\eps>0$, let $\chi$ be an arbitrary probability measure on $\R^d \times \R^d $ with marginals $\mu^{\eps}_0$ and $\mu^0_0$. Let also $\chi_s(B)=\chi\{ (u,v):\ (X_{\eps}(u,s),X_0(v,s)) \in B \}$, $B \in \B(\R^d \times \R^d )$. Then, for $t \in [0,T]$, we have the following
  \begin{align}
    \e \sup\limits_{ s \in [0,t] }\W_2^2(\mu_s^{\eps},\mu^0_s)&\leq \e \sup\limits_{ s \in [0,t] } \int_{ \R^d  }   \int_{ \R^d  }   |x-y|^2\chi_s(dx,dy)\\
    &\leq \int_{ \R^d  }   \int_{ \R^d  }   \e \sup\limits_{ s \in [0,t] }|X_{\eps}(u,s)-X_0(v,s)|^2\chi(du,dv)\\
    &\leq  C   \int_{ \R^d  }   \int_{ \R^d  }   |u-v|^2\chi(du,dv)+ C \int_{ \R^d  }   \int_{ \R^d  }   \int_{ 0 }^{ t } \e\W_2^2(\mu_s^{\eps},\mu^0_s)ds\chi(du,dv)\\
    &+\eps C\int_{ \R^d  }  \int_{ \R^d  }  \e \left( 1+|u|^2+ \langle \phi_2 , \mu_0^{\eps} \rangle  \right)\chi(du,ds)\\
    &=C \int_{ \R^d  }   \int_{ \R^d  }   |u-v|^2\chi(du,dv)+C \int_{ 0 }^{ t } \e \W_2^2(\mu_s^{\eps},\mu^0_s)ds\\
    &+ \eps C \left( 1+2 \langle \phi_2 , \mu_0^{\eps} \rangle  \right) \\
    &\leq C \int_{ \R^d  }   \int_{ \R^d  }   |u-v|^2\chi(du,dv)+C \int_{ 0 }^{ t } \e \sup\limits_{ r \in [0,s] } \W_2^{2}(\mu_r^{\eps},\mu^0_r)ds\\
    &+ \eps C\left( 1+ 2 \langle \phi_2 , \mu_0^{\eps} \rangle\right).
  \end{align}
   Taking the infimum over all probability measures $\chi$ with marginals $\mu_0^{\eps}$ and $\mu^0_0$, we obtain 
  \begin{align}
    \e \sup\limits_{ s \in [0,t] }\W_2^2(\mu_s^{\eps},\mu^0_s)&\leq C_1 \W_2^{2}(\mu^{\eps}_0,\mu^0_0)+C_1 \int_{ 0 }^{ t } \e \sup\limits_{ r \in [0,s] }\W_2^2(\mu_r^{\eps},\mu^0_r)ds \\
    &+ \eps C_2 \left( 1+2 \langle \phi_2 , \mu_0^{\eps} \rangle \right)
  \end{align}
  for all $t \in [0,T]$.  Using Gronwall's lemma again, we get the required inequality. This completes the proof of the theorem.
\end{proof}

\begin{remark} 
  \label{rem_estimate_for_difference_of_x_eps_and_x_0}
  Combining the estimate for $\e \sup\limits_{ s \in [0,t] }|X_{\eps}(u,s)-X_0(v,s)|^2$ from the proof of Theorem~\ref{the_lln} and the inequality~\eqref{equ_estimate_from_lln} from the statement of Theorem~\ref{the_lln}, one gets for every $t \in [0,T]$ and $\eps>0$
  \begin{align}
    \e \sup\limits_{ s \in [0,t] }|X_{\eps}(u,s)-X_0(v,s)|^2&\leq C\left(|u-v|^2+\W_2^2(\mu_0^{\eps},\mu_0^0)+ \eps  \left( 1+|u|^2+ \langle \phi_2 , \mu_0^{\eps} \rangle \right)\right),
  \end{align}
  where the constant $C$ depends only on $T$, $L$ and $d$.
\end{remark}

\subsection{Central limit theorem}
\label{sub:clt_for_mean_field_equation}
We note that Theorem~\ref{the_lln} implies that $\mu^{\eps}\to \mu^0$ as $\mu^{\eps}_0\to\mu^0_0$ and $\eps \to 0$. It this section, we will consider the fluctuations $\eta^{\eps}_t= \frac{1}{ \sqrt{ \eps } }\left( \mu_t^{\eps}-\mu^0_t \right)$ of $\mu_t^{\eps}$, $t\geq 0$, around $\mu^0_t$, $t\geq 0$, as $\eps \to 0$. Note that the process $\eta_t^{\eps}$, $t\geq 0$, takes values in the space $\M(\R^d )$ of all signed measures on $\R^d $ with finite total variation.  Since the Sobolev embedding theorem (see \cite[Theorem~4.12]{Adams:1975}) tells us that the space $H^J(\R^d )$ is continuously embedded into $\Cf_b^m(\R^d )$ for any $J> \frac{d}{ 2 }+m$, for every $J> \frac{d}{ 2 }$ and $\eps> 0$ the process $\eta^{\eps}_t$, $t \geq 0$, can be considered as a continuous process in $H^{-J}(\R^d )$, by the identification $\langle \varphi , \eta_t^{\eps} \rangle_0:=\langle \varphi, \eta_t^{\eps}\rangle$, $\varphi \in H^J(\R^d )$,  where $\langle\cdot,\cdot\rangle_0$ denotes the dualization between $H^J(\R^d)$ and $H^{-J}(\R^d)$. Moreover, a simple computation shows that
\begin{align}
  d\eta^{\eps}_t&=  \frac{ \sqrt{\eps}  }{ 2 }D^2:\left( A(t,\cdot,\mu_t^{\eps})\mu_t^{\eps} \right)-\nabla \cdot \left( V(t,\cdot ,\mu_t^{\eps})\eta_t^{\eps}+ \langle \tilde{V}(t,x ,\cdot ) , \eta^{\eps}_t \rangle_0 \mu_t^0(dx) \right)dt\\
  &- \int_{ \Theta }   \nabla \cdot \left( G(t,\cdot ,\mu_t^{\eps},\theta)\mu_t^{\eps} \right)W(d \theta,dt),
\end{align}
where we assume that $V(t,x,\mu)=\bar{V}(t,x)+ \langle \tilde{V}(t,x,\cdot ),\mu\rangle$. Passing formally to the limit as $\eps \to 0$, we expect that the limit of $\eta^{\eps}_t$, $t \geq 0$, is a solution to the same equation with $\eps=0$.

Therefore, the main goal of this section is to show that $\eta^{\eps}_t$, $t\geq 0$, converges to the solution $\eta^0_t$, $t\geq 0$, to the linear SPDE~\eqref{equ_equation_for_eta0} and to estimate the speed of convergence. We first prove some auxiliary statements and study the well posedness of the linear SPDE~\eqref{equ_equation_for_eta0}. We set $\Gamma_n=(-n,n)^d$ for $n \in \bar\N:=\N \cup \{ \infty \}$. We start with an auxiliary technical lemma that will prove useful in the proofs of the later results.

\begin{lemma} 
  \label{lem_estimate_for_coersivety}
  Let $J \in \N$, $n \in \bar\N$ and $v=(v_1,\dots,v_d)$ with $v_i \in \Cf_b^{J}(\Gamma_n )$, $i \in [d]$. Then the map $\cG:H^{-J+1}(\Gamma_n) \to H^{-J}(\Gamma_n)$ defined by 
  \[
    \langle \varphi,\cG(f) \rangle_{0,\Gamma_n}= \langle \nabla \varphi \cdot v , f \rangle_{0,\Gamma_n}, \quad f \in H^{-J+1}(\Gamma_n),\ \ \varphi \in \Cf^{\infty}_c(\Gamma_n),
  \]
  satisfies 
  \begin{equation} 
  \label{equ_estimate_for_coersivity}
|\langle \cG(f) , f \rangle_{-J,\Gamma_n}|\leq C \max\limits_{ i \in [d] }\|v_i\|_{\Cf^J_b}\|f\|_{-J,\Gamma_n}^2
  \end{equation}
  for all $f \in H^{-J+1}(\Gamma_n)$, where the constant $C$ depends only on $J$ and $d$. 
\end{lemma}

\begin{proof}
  We first remark that $\cG$ maps $H^{-J+1}(\Gamma_n)$ to $H^{-J}(\Gamma_n)$. Indeed, we have
  \begin{align}
  \|\cG(f)\|_{H^{-J}}&=\sup\limits_{\varphi \in \Cf_c^\infty(\Gamma_n)}\frac{\langle \nabla \varphi \cdot v , f \rangle_{0,\Gamma_n}}{\|\varphi\|_J}\\
  &\leq \sup\limits_{\varphi \in \Cf_c^\infty(\Gamma_n)}\frac{\| \nabla \varphi \cdot v \|_{J-1}}{\|\varphi\|_J}\|f\|_{-J+1}
  \leq C\max\limits_{i\in[d]}\|v_i\|_{\Cf_b^{J-1}}\|f\|_{-J+1}
  \end{align}
  for all $f\in H^{-J+1}$.
  
  Since $H^{-J}(\Gamma_n)$ is the dual space of $H^J(\Gamma_n)$, for every $f \in H^{-J+1}(\Gamma_n) \subset  H^{-J}(\Gamma_n)$ there exists a unique function $\tilde{f}=L_{J,\Gamma_n}^{-1}f \in H^J(\Gamma_n)$ such that $\langle \varphi , f \rangle_{0,\Gamma_n}=\langle \varphi , \tilde{f} \rangle_{J,\Gamma_n}$ for any $\varphi \in \Cf^{\infty}_c(\Gamma_n)$. We set 
  \[
    R:=\left\{ f \in H^{-J}(\Gamma_n):\ \tilde{f} \in \Cf_c^{\infty}(\Gamma_n ) \right\} \subset H^{-J+1}(\Gamma_n)
  \]
  and compute 
  \begin{align}
    \langle \cG(f) , f \rangle_{-J,\Gamma_n}&= \langle \tilde{f}, \cG(f)  \rangle_{0,\Gamma_n}=\langle \nabla \tilde{f} \cdot v , f \rangle_{0,\Gamma_n}=\langle \nabla \tilde{f}\cdot v , \tilde{f} \rangle_{J,\Gamma_n}\\
    &= \sum_{ |\alpha|\leq J }   \int_{ \Gamma_n  }   D^{\alpha}\left( \nabla \tilde{f}(x) \cdot v(x) \right)D^{\alpha} \tilde{f}(x)dx \\
    &=\sum_{ i=1 }^{ d } \sum_{ |\alpha|\leq J }   \int_{ \Gamma_n  }   D^{\alpha}\left(\partial_i \tilde{f}(x) v_i(x)  \right)D^{\alpha} \tilde{f}(x)dx.
  \end{align}
  For $\beta, \alpha \in \N_0^d$ we say $\beta\leq \alpha$ if $\beta_i\leq \alpha_i$ for all $i \in [d]$. We also set $\binom{\alpha}{\beta}=\binom{\alpha_1}{\beta_1}\dots\binom{\alpha_d}{\beta_d}$ for $\beta\leq \alpha$. Then, for every multi-index $\alpha$ with $|\alpha|=J$, we compute  
  \begin{align}
    D^{\alpha}\left( \partial_i \tilde{f}(x) v_i(x) \right)&= \sum_{ \beta\leq \alpha }   \binom{\alpha}{\beta}D^{\beta}\left( \partial_i \tilde{f}(x) \right)D^{\alpha-\beta}v_i(x)\\
    &=  \sum_{ \beta< \alpha }   \binom{\alpha}{\beta}D^{\beta}\left( \partial_i \tilde{f}(x) \right)D^{\alpha-\beta}v_i(x)+v_i(x)D^{\alpha}\partial_i \tilde{f}(x),
  \end{align}
  where  $\beta<\alpha$ means $\beta\leq \alpha$ and $|\beta|<|\alpha|$. Hence,
  \begin{equation} 
  \label{equ_rewriting_of_cg}
    \begin{split}
      \langle \cG(f) , f \rangle_{-J,\Gamma_n}&=\sum_{ i=1 }^{ d } \sum_{ |\alpha|<J } \sum_{ \beta\leq \alpha }  \binom{\alpha}{\beta}\int_{ \Gamma_n }   D^{\beta}\left( \partial_i \tilde{f}(x) \right)D^{\alpha-\beta}v_i(x)D^{\alpha}\tilde{f}(x)dx\\
      &+ \sum_{ i=1 }^{ d } \sum_{ |\alpha|=J }   \sum_{ \beta<\alpha }  \binom{\alpha}{\beta}\int_{ \Gamma_n }   D^{\beta}\left( \partial_i \tilde{f}(x) \right)D^{\alpha-\beta}v_i(x)D^{\alpha}\tilde{f}(x)dx\\
      &+\sum_{ i=1 }^{ d } \sum_{ |\alpha|=J }\int_{ \Gamma_n }   v_i(x)D^{\alpha}\partial_i \tilde{f}(x)D^{\alpha}\tilde{f}(x)dx .
    \end{split}
  \end{equation}
  Using the fact that $\tilde{f} \in \Cf_c^{\infty}(\Gamma_n)$ and the integration by parts, the last integral in the equality above can be rewritten as follows 
  \begin{equation} 
  \label{equ_rewriting_of_cg_integration_by_parts}
    \begin{split}
      \int_{ \Gamma_n }   v_i(x)D^{\alpha}\partial_i \tilde{f}(x)D^{\alpha}\tilde{f}(x)dx&= \frac{1}{ 2 }\int_{ \Gamma_n }   v_i(x)\partial_i\left( D^\alpha \tilde{f}(x) \right)^2dx\\
      &= - \frac{1}{ 2 }\int_{ \Gamma_n }   \partial_i v_i(x)\left( D^{\alpha}\tilde{f}(x) \right)^2dx.
    \end{split}
  \end{equation}
  Thus, Young's inequality and the equalities~\eqref{equ_rewriting_of_cg},~\eqref{equ_rewriting_of_cg_integration_by_parts} yield 
  \[
    |\langle \cG(f) , f \rangle_{-J,\Gamma_n}|\leq C \max\limits_{ i \in [d] }\|v_i\|_{\Cf_b^J(\Gamma_n)}\|\tilde{f}\|_{J,\Gamma_n}^2=C \max\limits_{ i \in [d] }\|v_i\|_{\Cf_b^J(\Gamma_n)}\|f\|_{-J,\Gamma_n}^2,
  \]
  for every $f \in R$ and some constant depending only on $d$ and $J$. We now prove the estimate~\eqref{equ_estimate_for_coersivity} for all $f \in H^{-J+1}$. Since the set $R$ is dense in $H^{-J+1}(\Gamma_n)$, there exists a sequence $\{ f_m,\ m\geq 1 \} \subset R$ which converges to $f$ in $H^{-J+1}(\Gamma_n)$. Since $\cG$ is a bounded linear operator, $\cG(f_m) \to \cG(f)$ in $H^{-J}(\Gamma_n)$. Moreover, $f_m \to f$ in $H^{-J}(\Gamma_n)$ due to the continuous embedding of $H^{-J+1}(\Gamma_n)$ into $H^{-J}(\Gamma_n)$. Hence,
  \begin{align}
    |\langle \cG(f) , f \rangle_{-J,\Gamma_n}|&= \lim_{ m\to\infty }|\langle \cG(f_m) , f_m \rangle_{-J,\Gamma_n}|\leq \lim_{ m\to\infty }C \max\limits_{ i \in [d] }\|v_i\|_{\Cf_b^J(\Gamma_n)} \|f_m\|_{-J,\Gamma_n}^2\\
    &=C \max\limits_{ i \in [d] }\|v_i\|_{\Cf_b^J(\Gamma_n)} \|f\|_{-J,\Gamma_n}^2.
  \end{align}
  This completes the proof of the lemma.
\end{proof} 

We now prove the well-posedness of the linear SPDE~\eqref{equ_equation_for_eta0}.

\begin{proposition} 
  \label{pro_well_posedness_of_the_equation_for_eta0}
  Let $J\geq \frac{d}{ 2 }+3$ and $\bar v:[0,T]\times \R^d \to \R^d $, $\tilde{v}:[0,T] \times \R^d \times \R^d \to \R^d $, $g_i:[0,T] \times \R^d \to L_2(\Theta,\m)$, $i\in[d]$, be measurable functions such that for every $t \in [0,T]$, $\bar v(t,\cdot ) \in \Cf^J(\R^d )$, $\tilde v(t,\cdot ,\cdot ) \in \Cf^J(\R^d \times \R^d )$, and  
  \begin{equation} 
  \label{equ_assumption_on_coefficients_for_equation_for_eta}
    \begin{split}
      &\sup\limits_{ t \in [0,T] }\|\bar v(t,\cdot )\|_{\Cf_b^J}+ \sup\limits_{ t \in [0,T], y \in \R^d  } \|\tilde v(t,\cdot ,y)\|_{\Cf_b^J}\\
      &\qquad\qquad+ \sup\limits_{ t \in [0,T],x \in \R^d  }\|\tilde{v}(t,x,\cdot )\|_{J}+ \sup\limits_{ t \in [0,T],x \in \R^d  } \frac{ \left\| |g(t,x,\cdot )| \right\|_{\m} }{ 1+|x| }<\infty.
    \end{split}
  \end{equation}
  Let also $\mu_t$, $t \in [0,T]$, be a continuous curve in $\cP_2(\R^d )$, $\varrho \in H^{-J+1}(\R^d )$, and $v(t,x)=\bar v(t,x)+\langle \tilde{v}(t,x,\cdot ) ,\mu_t  \rangle $. Then, there exists a unique $(\F_t)$-adapted continuous process in $H^{-J}(\R^d )$ such that for every $\varphi \in \Cf_c^{\infty}(\R^d )$ we have a.s. 
  \begin{equation} 
  \label{equ_linear_spde_for_eta}
    \begin{split}
      \langle \varphi , \eta_t \rangle_{0}&= \langle \varphi, \varrho \rangle_{0} +\int_{ 0 }^{ t } \left(\left\langle \nabla \varphi \cdot v(s,\cdot ) , \eta_s \right\rangle_{0,\R^d } + \left\langle \nabla \varphi \cdot \langle \tilde{v}(s,x,\cdot ) ,\eta_s  \rangle_{0,\R^d } , \mu_s(dx) \right\rangle \right) ds\\
    &+ \int_{ 0 }^{ t } \int_{ \Theta }   \left\langle \nabla \varphi \cdot g(s,\cdot, \theta) , \mu_s \right\rangle  W(d \theta,ds), \quad t \in [0,T].
    \end{split}
  \end{equation}
  Moreover, $\eta_t$, $t\geq 0$, is a Gaussian process in $H^{-J}(\R^d )$ and 
  \begin{equation} 
  \label{equ_bound_of_supremum_for_eta}
    \e \sup\limits_{ t \in [0,T] } \|\eta_t\|_{-J,\R^d }^2<\infty.
  \end{equation}
\end{proposition}

\begin{proof} 
  To prove the well-posedness of~\eqref{equ_linear_spde_for_eta}, we will use the general theory developed in~\cite{ESR12,Liu:2013}, where one needs to check that the coefficients of the equation satisfy some conditions (see (H0)-(H4) below for more details). Unfortunately, working in the Sobolev space $H^{-J}(\R^d )$, we will not be able to verify (H0). In order to overcome this problem, we will first construct solutions to cut-off versions of the SPDE in the Sobolev spaces $H^{-J}(\Gamma_n)$ for all $G_n=(-n,n)^d$, $n \in \N$, and then pass to the limit as $n\to\infty$. 
  
  Since for every $n \in \bar\N$ the space $\cV_n:=H^{-J+1}(\Gamma_n)$ is continuously and densely embedded into the Hilbert space $\cH_n:=H^{-J}(\Gamma_n)$, 
  \[
    \cV_n \subset \cH_n\cong\cH_n^* \subset \cV_n^*
  \]
  is a Gelfand triple, where $\cV_n^*$ is the dual space of $\cV_n$ with respect to the inner product $\langle \cdot  , \cdot  \rangle_{\cH_n}=\langle \cdot  , \cdot  \rangle_{-J,\Gamma_n}$. In particular,
  \[
    {}_{\cV^*_n}\langle f , \varphi \rangle_{\cV_n}= \langle f , \varphi \rangle_{\cH_n}, \quad f \in \cH_n,\ \ \varphi \in \cV_n.
  \] 
  Let $\cL_2\left( L_2(\Theta,\m);\cV_n \right)$ denote the space of all Hilbert--Schmidt operators from $L_2(\Theta,\m)$ to $\cV_n$ and $\|\cdot \|_{\mathrm{HS},\cV_n}$ be the Hilbert--Schmidt norm on that space.  For the cylindrical Brownian motion $W_t$, $t\geq 0$, in $L_2(\Theta,\m)$ and fixed $n \in \bar\N$, we consider the following equation 
  \begin{equation} 
  \label{equ_general_linear_spde}
  d\xi_t^n=U_n(t,\xi_t^n)dt+B_n(t)dW_t,
  \end{equation}
  where $U_n=\bar U_n+\tilde U_n:[0,T]\times \cV_n \to \cH_n \subset  \cV_n^*$ and $B_n:[0,T] \to \cL_2\left( L_2(\Theta,\m); \cV_n \right)$  are measurable functions defined by 
  \begin{align}
    \langle \bar U_n(t,f),\varphi \rangle_{0,\Gamma_n}&= \langle \nabla \varphi \cdot v(t, \cdot) , f \rangle_{0,\Gamma_n},\\ 
    \langle \tilde U_n(t,f),\varphi \rangle_{0,\Gamma_n}&= \left\langle \nabla \varphi(x) \cdot \langle \kappa_n\tilde{v}(t,x,\cdot ) , f \rangle_{0,\Gamma_n} , \mu_t(dx) \right\rangle
  \end{align}
  and
  \[
    \langle B_n(t)h, \varphi \rangle_{0,\Gamma_n}=\left\langle h , \langle \nabla \varphi(x) \cdot g(t,x ,\cdot ) , \mu_t(dx) \rangle \right\rangle_{\m}
  \]
  for all $\varphi \in \Cf_c^{\infty}(\Gamma_n)$, $f \in \cV_n$ and $h \in L_2(\Theta,\m)$, where we extend the function $\varphi$ to the whole space $\R^d$ by zero in the integrals with respect to $\mu_t$, the function $\kappa_n \in \Cf_c^{\infty}(\R^d)$ has a compact support in $\Gamma_n$ and $\kappa_n=1$ on $\Gamma_{n-1}$ for every $n \in \N$, $\kappa_{\infty}=1$ on $\R^d $, and $\sup_{n\geq 1}\|\kappa_n\|_{\Cf_b^{J-1}}<\infty$. The fact that $U_n$ and $B_n(t)h$ take values in $\cH_n$ and $\cV_n$, respectively, follows from the continuous embedding of the Sobolev space $H^{-J}(\Gamma_n)$ into $\Cf_b^1(\Gamma_n )$.

  To prove the well-posedness of the equation~\eqref{equ_general_linear_spde}, we need to check the following conditions for all $n \in \N$.
  \begin{enumerate}
    \item [(H0)] There exists an orthogonal set $\{ e_1,e_2,\dots \}$ in $\cV_n$ such that it constitutes an orthonormal basis of $\cH_n$.

    \item [(H1)] (Hemicontinuity) The map $s \mapsto {}_{\cV^*_n}\langle U_n(t,f_1+s f_2) , f \rangle_{\cV_n}$ is continuous on $\R $ for any $f_1,f_2, f \in \cV_n$ and $t \in [0,T]$.

    \item [(H2)] (Monotonicity) There exists a constant $C>0$ such that 
      \[
	{}_{\cV_n^*}\langle U_n(t,f_1)-U_n(t,f_2) , f_1-f_2 \rangle_{\cV_n}\leq C \|f_1-f_2\|_{\cH_n}^2
      \]
      for all $t \in [0,T]$ and $f_1,f_2 \in \cV_n$.

    \item [(H3)] (One-side linear growth) For any $k \in \N$ the operator $U_n(t,\cdot )$ maps the set $E^k:=\spann\{ e_1,\dots,e_k \}$ into $\cV_n$, and there exists a constant $C>0$ such that 
      \[
	\langle U_n(t,f) , f \rangle_{\cV_n}\leq C(1+\|f\|_{\cV_n}^2) 
      \]
      for all $f \in E^k$ and some constant independent of $k$ and $t$.

    \item [(H4)] (Growth) There exists a constant $C>0$ such that 
      \[
	\|U_n(t,f)\|_{\cV^*_n}\leq C(1+ \|f\|_{\cV_n}), \quad \|B_n(t)\|^2_{\mathrm{HS},\cV_n}\leq C
      \]
      for all $t \in [0,T]$ and $f \in \cV_n$.
  \end{enumerate}

  The condition (H0) follows from the compact embedding of the space $\cV_n$ into $\cH_n$. Indeed, due to the fact that the inner product $\langle \cdot  , \cdot  \rangle_{\cV_n}$ induces a closed quadratic form on $\cH_n$, there exists a densely defined self-adjoint operator $L:\cH_n \supset D(L) \to \cH_n$ on $\cH_n$ such that $\langle f_1 , f_2 \rangle_{\cV_n}=\langle f_1 , Lf_2 \rangle_{\cH_n}$ for all $f_1 \in \cV_n$ and $f_2 \in D(L)$. By the compact embedding of $\cV_n$ into $\cH_n$, the operator $L$ has a discrete spectrum (see also~\cite[Section~2.1]{ESR12}) $\{ \lambda_k,\ k\geq 1 \}$ with corresponding eigenbasis $\{ e_k,\ k\geq 1 \}$ in $\cH_n$. It is easy to see that $\{ e_k,\ k\geq 1 \}$ is an orthogonal system in $\cV_n$ and $e_k \in H^{-J+2} (\Gamma_n)$ for every $k\geq 1$. 
  
  The condition (H1) holds due to the linearity of $U_n(t, \cdot )$ for every $t \in [0,T]$. 
  
  In order to check the monotonicity of $U_n$, we estimate for $f \in \cV_n$ and $t \in [0,T]$
  \begin{align}
    {}_{\cV_n^*}\langle U_n(t,f) , f \rangle_{\cV_n}&= \langle U_n(t,f) , f \rangle_{\cH_n}=\langle \bar U_n(t,f) , f \rangle_{\cH_n}+\langle \tilde{U}_n(t,f) ,f  \rangle_{\cH_n}.
  \end{align}
  By Lemma~\ref{lem_estimate_for_coersivety}, 
  \[
    \langle \bar U_n(t,f) , f \rangle_{\cH_n}=\langle \bar U_n(t,f) , f \rangle_{-J,\Gamma_n} \leq C \max\limits_{ i \in [d] }\sup\limits_{ t \in [0,T] }\|v_i(t,\cdot)\|_{\Cf_b^J}\|f\|_{-J,\Gamma_n}^2, 
  \]
  where the constant $C$ is independent of $n$. We remark that $\sup\limits_{ t \in [0,T] }\|v_i(t,\cdot)\|_{\Cf_b^J}$ is finite according to~\eqref{equ_assumption_on_coefficients_for_equation_for_eta}. For every $f \in \cV_n \subset H^{-J}(\Gamma_n)$ we define $\tilde{f}=L^{-1}_{J,\Gamma_n}f\in H^{J}(\Gamma_n)$ and note that $\langle \tilde{f} , \varphi \rangle_{J,\Gamma_n}=\langle f , \varphi \rangle_{0,\Gamma_n}$ for all $\varphi \in \Cf_c^{\infty}(\Gamma_n)$. By ~\cite[Theorem~5.29]{Adams:1975}, the extension by zero of $\tilde{f}$ on $\R^d$, which we also denote $\tilde{f}$, belongs to $H^J(\R^d)$. We next estimate 
  \begin{equation} 
  \label{equ_coersivity_of_tilda_u}
    \begin{split}
      \langle \tilde{U}_n(t,f) , f \rangle_{\cH_n}&= \langle \tilde{U}_n(t,f) , f \rangle_{-J,\Gamma_n}=\langle \tilde{U}_n(t,f) , \tilde{f} \rangle_{0,\Gamma_n}\\
      &= \left\langle \nabla \tilde{f}(x) \cdot \langle \kappa_n\tilde{v}(t,x,\cdot ) , f \rangle_{0,\Gamma_n} , \mu_t(dx) \right\rangle\\
      &\leq \sum_{ i=1 }^{ d } \int_{ \R^d  }   \left|\partial_i \tilde{f}(x)\right|\|\kappa_n \tilde{v}_i(t,x,\cdot )\|_{J,\Gamma_n}\|f\|_{-J,\Gamma_n}\mu_t(dx)\\
      &\leq \|\tilde{f}\|_{\Cf_b^1}|\|f\|_{-J,\Gamma_n}\sum_{ i=1 }^{ d } \sup\limits_{t \in [0,T],   x \in \R^d}\|\tilde{v}_i(t,x,\cdot )\|_{J,\R^d }. 
    \end{split}
  \end{equation}
  By the continuous embedding of $H^J(\R^d )$ into $\Cf_b^1(\R^d )$, we obtain
  \[
    \|\tilde{f}\|_{\Cf_b^1}\leq C \|\tilde{f}\|_{J,\R^d }=C \|\tilde{f}\|_{J,\Gamma_n}=C \|f\|_{-J,\Gamma_n}.
  \]
  Thus, the linearity of $U_n(t,\cdot )$ and the estimates above imply the monotonicity of $U_n$ with the constant $C$ independent of $n$. This establishes condition (H2).

  The function $U_n(t,\cdot )$ maps $E^k=\spann\{ e_1,\dots,e_k \}$ to $\cV_n$ due to the fact that $e_k \in H^{-J+2}(\Gamma_n)$ for all $k\geq 1$. Moreover, Lemma~\ref{lem_estimate_for_coersivety} and a similar computation to~\eqref{equ_coersivity_of_tilda_u} imply 
  \[
    \langle U_n(t,f) , f \rangle_{\cV_n}\leq C \max\limits_{ i \in [d] }\sup\limits_{ t \in [0,T], x \in \R^d  }\left(\|v_i(t,\cdot)\|_{\Cf_b^{J-1}}+\|\tilde{v}_i(t,x,\cdot )\|_{J-1,\R^d }\right) \|f\|_{\cV_n}^2
  \]
  for all $f \in E^k$, where the constant $C$ is independent of $n$ and $k$. Thus, condition (H3) holds.

  We separately estimate the norms of $\bar{U}_n$ and $\tilde{U}_n$ in $\cH_n$ as follows
  \begin{align}
    \|\bar{U}_n(t,f)\|_{\cH_n}&= \sup\limits_{ \varphi \in \Cf_c^{\infty}(\Gamma_n) } \frac{1}{ \|\varphi\|_{J,\Gamma_n} } \langle \nabla \varphi \cdot v(t,\cdot) , f \rangle_{0,\Gamma_n}\\
    &\leq \sup\limits_{ \varphi \in \Cf_c^{\infty}(\Gamma_n) } \frac{ \|\nabla \varphi \cdot v(t,\cdot)\|_{J-1,\Gamma_n} }{ \|\varphi\|_{J,\Gamma_n} } \|f\|_{\cV_n}\\
    &\leq C \max\limits_{ i \in [d] }\sup\limits_{ t \in [0,T] }\|v_i(t,\cdot)\|_{\Cf_b^{J-1}} \|f\|_{\cV_n},
  \end{align}
  for all $f \in \cV_n$ and $t \in [0,T]$, by the continuous embedding of $H^{J}(\R^d )$ into $\Cf_b^1(\R^d )$. Similarly, 
  \begin{align}
    \|\tilde{U}_n(t,f)\|_{\cH_n}&= \sup\limits_{ \varphi \in \Cf_c^{\infty}(\Gamma_n) } \frac{1}{ \|\varphi\|_{J,\Gamma_n} }\left\langle \nabla \varphi(x) \cdot \langle \kappa_n\tilde{v}(t,x,\cdot ) , f \rangle_{0,\Gamma_n} , \mu_t(dx) \right\rangle\\
    &\leq C \max\limits_{ i \in [d] }\sup\limits_{ t \in [0,T], x \in \R^d  }\|\kappa_n\tilde{v}_i(t,x,\cdot )\|_{J-1,\R^d } \|f\|_{\cV_n}\\
    &\leq C\|\kappa_n\|_{\Cf_b^{J-1}} \max\limits_{ i \in [d] }\sup\limits_{ t \in [0,T], x \in \R^d  }\|\tilde{v}_i(t,x,\cdot )\|_{J-1,\R^d } \|f\|_{\cV_n}\\
    &\leq C\max\limits_{ i \in [d] }\sup\limits_{ t \in [0,T], x \in \R^d  }\|\tilde{v}_i(t,x,\cdot )\|_{J-1,\R^d } \|f\|_{\cV_n}
  \end{align}
  for all $f \in \cV_n$ and $t \in [0,T]$, where the constant $C$ is independent of $n$ due to the specific choice of $\kappa_n$. 

  To estimate the Hilbert--Schmidt norm of $B_n(t)$, we consider an orthonormal basis $\{ h_k,\ k\geq 1 \}$ in $L_2(\Theta,\m)$ and compute for every $k\geq 1$ and $\varphi \in \Cf_c^{\infty}(\Gamma_n)$
  \begin{align}
    \langle B_n(t)h_k , \varphi \rangle_{0,\Gamma_n} = \left\langle h_k , \left\langle \nabla \varphi(x)\cdot g(t,x,\cdot ) , \mu_t(dx) \right\rangle \right\rangle_{\m}=\left\langle \nabla \varphi \cdot g^k(t,\cdot) , \mu_t \right\rangle, 
  \end{align}
  where $g^k(t,x)=\langle g(t,x,\cdot ) , h_k \rangle_{\m}$. Hence, by the continuous embedding of $H^{J-1}(\R^d )$ into $\Cf_b^1(\R^d )$ due to $J\geq \frac{d}{2}+3$, we obtain
  \begin{align}
    \|B_n(t)h_k\|_{\cV_n}&= \sup\limits_{ \varphi \in \Cf_c^{\infty}(\Gamma_n) } \frac{ \left\langle \nabla \varphi \cdot g^k(t,\cdot) , \mu_t \right\rangle }{ \|\varphi\|_{J-1,\Gamma_n} }\leq  C \langle |g^k(t,\cdot)| , \mu_t \rangle,
  \end{align}
  where the constant $C$ is independent of $n$. Consequently, by Jensen's inequality and Fubini's theorem,
  \begin{align}
    \|B_n(t)\|_{\mathrm{HS},\cV_n}^2&= \sum_{ k=1 }^{ \infty } \|B_n(t)h_k\|^{2}_{\cV_n}\leq C^2 \sum_{ k=1 }^{ \infty } \langle |g^k(t,\cdot)| , \mu_t \rangle^2\\
    &\leq C^2\sum_{ i=1 }^{ d } \left\langle\sum_{ k=1 }^{ \infty } (g_i^k(t,\cdot))^2,\mu_t\right\rangle=C^2 \sum_{ i=1 }^{ d } \left\langle \|g_i(t,x,\cdot )\|_{\m}^2 , \mu_t(dx) \right\rangle.
  \end{align}
  We note that $\sup\limits_{ t \in [0,T] }\langle \|g_i(t,x,\cdot )\|^2_{\m} , \mu_t(dx) \rangle<\infty$ due to~\eqref{equ_assumption_on_coefficients_for_equation_for_eta}. Thus, condition (H4) holds true with a constant $C$ that is independent of $n$.

  Now we can apply ~\cite[Theorem~1.1]{Liu:2013}, to obtain that, for every $n\geq 1$ and $\varrho \in H^{-J+1} \subset \cV_n$, there exists a continuous $\cH_n$-valued $(\F_t)$-adapted process $\eta_t^n$, $t \in [0,T] $, which has a $dt \times \p$-version $\bar\eta^n$ from $L^2([0,T]\times \Omega,dt \times \p;V)$ and $\p$-a.s.
  \[
    \eta_t^{n}=\varrho+\int_{ 0 }^{ t } U_n(s,\bar{\eta}_s^n)ds+\int_{ 0 }^{ t } B_n(s)dW_s, \quad t \in [0,T].  
  \]
  Using~\cite[Lemma~2.1]{Liu:2013}, one can see that 
  \begin{align}
    \e \sup\limits_{ t \in [0,T] }\|\eta_t^n\|_{\cH_n}^2+\e \int_{ 0 }^{ T } \|\bar{\eta}_t^n\|_{\cV_n}^2 dt
    &\leq C (1+ \|\varrho\|_{\cV_n}^2),
  \end{align}
  where the constant $C$ is independent of $n$ because all constants in the conditions (H1)-(H4) are independent of $n$. Since for every $m\geq n$ the relationships $\cH_m \subset \cH_n$, $\cV_m \subset \cH_n $, and $\|\cdot \|_{\cH_n}\leq \|\cdot \|_{\cH_m}$, $\|\cdot \|_{\cV_n}\leq \|\cdot\|_{\cV_m}$ hold, we get 
  \begin{equation} 
  \label{equ_estimate_for_z_m}
    \e \sup\limits_{ t \in [0,T] }\|\eta_t^m\|_{\cH_n}^2+\e \int_{ 0 }^{ T } \|\bar{\eta}_t^m\|_{\cV_n}^2 dt \leq C (1+\e \|\varrho\|_{\cV_m}^2)\leq C(1+\|\varrho\|_{\cV_{\infty}}^2)
  \end{equation}
  for all $m\geq n$. Hence the sequence $\{ \eta^m,\ m\geq n \}$ is relatively compact in the weak topologies of $L_2(\cH_n):=L_2([0,T]\times \Omega,dt \times \p;\cH_n)$, and $L_2(\cV_n):=L_2([0,T] \times \Omega, dt \times \p;\cV_n)$ for every $n\geq 1$. Using a diagonal argument, there exists a subsequence $\{ m_k,\ k\geq 1 \}$ such that $\eta^{m_k} \to \bar\eta$ in the weak topologies of $L_2(\cH_n)$ and $L_2(\cV_n)$ for every $n\geq 1$. In particular,
  \[
    \e \int_{ 0 }^{ T } \|\bar\eta_t\|_{\cV_n}^2dt \leq C(1+ \|\varrho\|_{\cV_{\infty}}^2) 
  \]
  for all $n\geq 1$. Using Fatou's lemma and the fact that the sequence $\|\bar\eta_t\|_{\cV_n}$, $n\geq 1$, increases, we get 
  \[
    \e \int_{ 0 }^{ T } \lim_{ n\to\infty }\|\bar\eta_t\|_{\cV_n}^2dt\leq  C(1+\|\varrho\|_{-J+1,\R^d }^2). 
  \]
  This yields that $\bar\eta \in L_2([0,T] \times \Omega,dt \times \p;H^{-J+1}(\R^d ))$ and 
  \[
    \e \int_{ 0 }^{ T } \|\bar\eta_t\|_{-J+1,\R^d }^2dt\leq C(1+\|\varrho\|_{-J+1,\R^d }^2). 
  \]

  Next, let $\gamma \in L_{\infty}([0,T]\times \Omega,dt \times \p;\R )$ and $\varphi \in \Cf_c(\R^d )$. We take $n \in \N$ satisfying $\supp \varphi \subset \Gamma_n$ and compute
  \begin{align}
    &\e \int_{ 0 }^{ T } \gamma(t)\langle \bar\eta_t , \varphi \rangle_{0,\R^d }dt = \e \int_{ 0 }^{ T } \gamma(t)\langle \bar\eta_t , \varphi \rangle_{0,\Gamma_n}dt =\lim_{ k\to\infty }\e \int_{ 0 }^{ T } \gamma(t)\langle \eta^{m_k}_t , \varphi \rangle_{0,\Gamma_n} dt\\
    &\qquad\qquad= \e\int_{ 0 }^{ T } \gamma(t) \langle \varrho , \varphi \rangle_{0,\Gamma_n}dt+ \lim_{ k\to\infty }\e \int_{ 0 }^{ T } \gamma(t)\int_{ 0 }^{ t } \langle U_{m_k}(s,\bar \eta_t^{m_k}),\varphi\rangle_{0,\Gamma_n} ds dt  \\
    &\qquad\qquad+\lim_{ k\to\infty }\e \int_{ 0 }^{ T } \gamma(t)\int_{ 0 }^{ t } \langle B_{m_k}(s) , \varphi \rangle_{0,\Gamma_n} dW_s dt  \\
    &\qquad\qquad= \e \int_{ 0 }^{ T }\gamma(t) \langle \varrho , \varphi \rangle_{0,\R^d } dt+ \lim_{ k\to\infty }\e \int_{ 0 }^{ T } \gamma(t) \int_{ 0 }^{ t } \langle \nabla \varphi \cdot v(s,\cdot) , \eta^{m_k}_s \rangle_{0,\Gamma_n} ds dt\\
    &\qquad\qquad+\lim_{ k\to\infty }\e \int_{ 0 }^{ T } \gamma(t)\int_{ 0 }^{ t } \left\langle \nabla \varphi (x) \cdot \langle \kappa_{m_k} \tilde{v}(s,x,\cdot ) , \eta_s^{m_k} \rangle_{0,\Gamma_{m_k} } , \mu_s(dx) \right\rangle ds dt\\
    &\qquad\qquad+\int_{ 0 }^{ T } \gamma(t)\int_{ 0 }^{ t } \int_{ \Theta }   \langle \nabla \varphi \cdot g(s,\cdot, \theta) , \mu_s \rangle W(d \theta,ds) dt.   
  \end{align}

  Using the convergence of $\eta^{m_k}$ in the weak topology of $L_2(\cV_n)$, we get
  \begin{align}
    \lim_{ k\to\infty }\e \int_{ 0 }^{ T } \gamma(t) \int_{ 0 }^{ t } \langle \nabla \varphi \cdot v(s,\cdot) , \eta^{m_k}_s \rangle_{0,\Gamma_n} ds dt&=
    \e \int_{ 0 }^{ T } \gamma(t) \int_{ 0 }^{ t } \langle \nabla \varphi \cdot v(s,\cdot) , \bar\eta_s \rangle_{0,\Gamma_n} ds dt\\
    &= \e \int_{ 0 }^{ T } \gamma(t) \int_{ 0 }^{ t } \langle \nabla \varphi \cdot v(s,\cdot) , \bar\eta_s \rangle_{0,\R^d } ds dt.
  \end{align}
  Next,  we rewrite for every $k,l \geq 1$
  \begin{align}
    \langle \kappa_{m_k} \tilde{v}(s,x,\cdot ) , \eta_s^{m_k} \rangle_{0,\Gamma_{m_k}}&-\langle \tilde{v}(s,x,\cdot ) , \bar\eta_s \rangle_{0,\R^d}= \langle (\kappa_{m_k}-1)\tilde{v}(s,x,\cdot ) , \bar\eta_s \rangle_{0,\R^d }\\
    &+ \langle (\kappa_{m_k}-\kappa_{l})\tilde{v}(s,x,\cdot ) ,\eta^{m_k}_s-\bar\eta_s  \rangle_{0,\Gamma_{m_k}}+ \langle \kappa_{l}\tilde{v}(s,x,\cdot ) ,\eta^{m_k}_s-\bar\eta_s  \rangle_{0,\Gamma_{l}}.
  \end{align}
  Since 
  \begin{align}
    |\langle (1-\kappa_{m_k})\tilde{v}(s,x,\cdot ) , \bar\eta_s \rangle_{0,\R^d }|&\leq \|(\kappa_{m_k}-1)\tilde{v}(s,x,\cdot )\|_{J-1,\R^d } \|\bar\eta_s\|_{-J+1,\R^d },\\
    |\langle (\kappa_{m_k}-\kappa_{l})\tilde{v}(s,x,\cdot ) , \eta_s^{m_k}-\bar\eta_s \rangle_{0,\Gamma_{m_k} }|&\leq \|(\kappa_{m_k}-\kappa_{l})\tilde{v}(s,x,\cdot )\|_{J-1,\Gamma_{m_k} } \|\eta_s^{m_k}-\bar\eta_s\|_{\cV_{m_k} },\\
    &\leq C\|(1-\kappa_{l})\tilde{v}(s,x,\cdot )\|_{J-1,\R^d  } \|\eta_s^{m_k}-\bar\eta_s\|_{\cV_{m_k} }
  \end{align}
  and $\eta^{m_k} \to \bar\eta$ in the weak topology of $L_2(\cV_l)$ for every $l$, it is easily seen that
  \begin{align}
    &\lim_{ k\to\infty }\bigg(\e \int_{ 0 }^{ T } \gamma(t)\int_{ 0 }^{ t } \left\langle \nabla \varphi (x) \cdot \langle \kappa_{m_k} \tilde{v}(s,x,\cdot ) , \eta_s^{m_k} \rangle_{0,\Gamma_{m_k}} , \mu_s(dx) \right\rangle ds dt\\
    &\qquad\qquad-\e \int_{ 0 }^{ T } \gamma(t)\int_{ 0 }^{ t } \left\langle \nabla \varphi (x) \cdot \langle \tilde{v}(s,x,\cdot ) , \bar\eta_s \rangle_{0,\R^d } , \mu_s(dx) \right\rangle ds dt\bigg)\\
    &\qquad\qquad =\lim_{ k\to\infty }\e \int_{ 0 }^{ T } \gamma(t)\int_{ 0 }^{ t } \left\langle \nabla \varphi (x) \cdot \langle (\kappa_{m_k}-1) \tilde{v}(s,x,\cdot ) , \bar\eta_s \rangle_{0,\R^d } , \mu_s(dx) \right\rangle ds dt\\
    &\qquad\qquad +\lim_{ l\to\infty }\lim_{ k\to\infty }\e \int_{ 0 }^{ T } \gamma(t)\int_{ 0 }^{ t } \left\langle \nabla \varphi (x) \cdot \langle (\kappa_{m_k}-\kappa_l) \tilde{v}(s,x,\cdot ) , \eta_s^{m_k}-\bar\eta_s \rangle_{0,\Gamma_{m_k}} , \mu_s(dx) \right\rangle ds dt\\
    &\qquad\qquad +\lim_{ l\to\infty }\lim_{ k\to\infty }\e \int_{ 0 }^{ T } \gamma(t)\int_{ 0 }^{ t } \left\langle \nabla \varphi (x) \cdot \langle \kappa_l \tilde{v}(s,x,\cdot ) , \eta_s^{m_k}-\bar\eta_s \rangle_{0,\Gamma_{l}} , \mu_s(dx) \right\rangle ds dt=0.
  \end{align}
  This implies that for that for a.e. $t \in [0,T]$ and all $\varphi \in \Cf_c^{\infty}(\R^d )$
  \begin{align}
    \langle \varphi , \bar\eta_t \rangle_{0,\R^d }&= \langle \varphi, \varrho \rangle_{0,\R^d } +\int_{ 0 }^{ t } \left(\left\langle \nabla \varphi \cdot v(s,\cdot) , \bar\eta_s \right\rangle_{0, \R^d } + \left\langle \nabla \varphi \cdot \langle \tilde{v}(s,x,\cdot ) ,\bar\eta_s  \rangle_{0, \R^d } , \mu_s(dx) \right\rangle \right) ds\\
  &+ \int_{ 0 }^{ t } \int_{ \Theta }   \left\langle \nabla \varphi \cdot g(s,\cdot, \theta) , \mu_s \right\rangle  W(d \theta,ds), \quad t \in [0,T].
  \end{align}
  Taking 
  \[
    \eta_t= \varrho +\int_{ 0 }^{ t } U_{\infty}(s,\bar\eta_s)ds+\int_{ 0 }^{ t } B_{\infty}(s)dW_s, \quad t \in [0,T],  
  \]
  and using the continuity of the right hand side of the expression above, we get the existence of solutions to~\eqref{equ_equation_for_eta0}.

  \Itos formula from \cite[Theorem~4.2.4]{PR07} and the condition (H2), which also holds true for $n=\infty$, imply the uniqueness of solutions to~\eqref{equ_linear_spde_for_eta} and the bound~\eqref{equ_bound_of_supremum_for_eta}. The fact that $\eta_t$, $t \in [0,T]$, is a Gaussian process in $H^{-J}(\R^d )$ follows from the linearity of the equation. This completes the proof of the proposition.
\end{proof}

\begin{lemma} 
  \label{lem_boundedness_of_the_norm_of_eta_eps}
  Let the coefficients $A,V,G$ satisfy Assumptions~\ref{ass_basic_assumption},~\ref{ass_lipshitz_continuity}. Let also $\mu_t^{\eps}$, $t\geq 0$, be a superposition solution to the stochastic mean-field equation~\eqref{equ_mean_field_equation_with_vanishing_noise} started from $\mu_0^{\eps}\in\cP_2(\R^d)$, $\mu_t$, $t\geq 0$, be a superposition solution to the PDE~\eqref{equ_mean_field_equation_without_noise} started from $\mu_0^0 \in \cP_2(\R^d )$ and  the process $\eta^{\eps}_t$, $t\geq 0$, be defined by~\eqref{equ_definition_of_eta_eps}.  Then for every $J> \frac{d}{ 2 }+1$ and $T>0$ there exists a constant $C>0$ such that 
  \[
    \e \sup\limits_{ t \in [0,T] }\|\eta_t^{\eps}\|_{-J}^2\leq C\left( 1+\langle \phi_2 , \mu_0^{\eps} \rangle \right)+ \frac{C}{ \eps } \W_2^2\left( \mu_0^{\eps},\mu_0^0 \right).
  \]
  for all $\eps>0$.
\end{lemma}

\begin{proof} 
  Since for each $\eps\geq 0$ the process $\mu_t^{\eps}$, $t\geq 0$, is a superposition solution to the corresponding equation, $\mu_t^{\eps}=\mu_0^{\eps}\circ X_{\eps}(t,\cdot )$, $t\geq 0$, where $X_{\eps}$ is a solution to the SDE with interaction~\eqref{equ_equation_with_interaction} with $G$ replaced by $\sqrt{ \eps }G$. Similarly to the proof of Theorem~\ref{the_lln}, we fix $\eps>0$ and consider an arbitrary probability measure $\chi$ on $\R^d \times \R^d $ with marginals $\mu_0^{\eps}$ and $\mu_0^0$. Using Remark~\ref{rem_estimate_for_difference_of_x_eps_and_x_0} and the continuous embedding of $H^{J}_0(\R^d)$ into $\Cf_b^1(\R^d)$, we get 
  \begin{align}
    \e \sup\limits_{ t \in [0,T] }\|\eta_t^{\eps}&\|^2_{-J}= \frac{1}{ \eps }\e \sup\limits_{ t \in [0,T] } \sup\limits_{ \varphi \in \Cf_c^{\infty}(\Gamma) } \frac{1}{ \|\varphi\|_{J}^2 } \langle \varphi , \mu^{\eps}_t-\mu^0_t \rangle^2\\
    &= \frac{1}{ \eps }\e \sup\limits_{ t \in [0,T] } \sup\limits_{ \varphi \in \Cf_c^{\infty}(\Gamma) } \frac{1}{ \|\varphi\|_{J}^2 } \left(\int_{ \R^d  }   \int_{ \R^d  }   \left( \varphi(X_{\eps}(u,t))-\varphi(X_0(v,t)) \right) \chi(du,dv)  \right)^2\\
    &\leq  \frac{1}{ \eps }\e \sup\limits_{ t \in [0,T] } \sup\limits_{ \varphi \in \Cf_c^{\infty}(\Gamma) } \frac{1}{ \|\varphi\|_{J}^2 } \int_{ \R^d  }   \int_{ \R^d  }   \left( \varphi(X_{\eps}(u,t))-\varphi(X_0(v,t)) \right)^2 \chi(du,dv)\\
    &\leq  \frac{1}{ \eps }\e \sup\limits_{ t \in [0,T] } \sup\limits_{ \varphi \in \Cf_c^{\infty}(\Gamma) } \frac{\sup\limits_{ x \in \Gamma }|\nabla \varphi(x)|^2}{ \|\varphi\|_{J}^2 } \int_{ \R^d  }   \int_{ \R^d  }   \left| X_{\eps}(u,t)-X_0(v,t) \right|^2 \chi(du,dv)\\
    &\leq  \frac{C}{ \eps }\e \sup\limits_{ t \in [0,T] }  \int_{ \R^d  }   \int_{ \R^d  }   \left| X_{\eps}(u,t)-X_0(v,t) \right|^2 \chi(du,dv)\\
    &\leq  \frac{C}{ \eps }\int_{ \R^d  }   \int_{ \R^d  }  \e \sup\limits_{ t \in [0,T] }   \left| X_{\eps}(u,t)-X_0(v,t) \right|^2 \chi(du,dv)\\
    &\leq  \frac{C_1}{ \eps } \int_{ \R^d  }   \int_{ \R^d  }   \left( |u-v|^2+\W_2^2( \mu_0^{\eps},\mu_0^0 )+ \eps\left( 1+|u|^2+ \langle \phi_2 , \mu_0^{\eps} \rangle \right) \right) \chi(du,dv) \\
    &=  \frac{C_1}{ \eps }\left( \int_{ \R^d  }   \int_{ \R^d  }   |u-v|^2\chi(du,dv)+\W_2^2( \mu_0^{\eps},\mu_0^0 )+ \eps\left( 1+2\langle \phi_2 , \mu_0^{\eps} \rangle \right)\right).
  \end{align}
  Taking infimum over all $\chi$ with marginals $\mu^{\eps}_0$ and $\mu^0_0$, we obtain the needed estimate, which completes the proof of the lemma.
\end{proof}

To prove the convergence of the fluctuation field $\eta^{\eps}$, we will need the following assumptions on the coefficients of the stochastic mean-field equation~\eqref{equ_mean_field_equation_with_vanishing_noise}.

\begin{assumption} 
  \label{ass_compact_support}
  \,
  \begin{enumerate}
    \item [(i)] The coefficients $A,V,G$ do not depend on $\omega\in \Omega$ and there exist Borel measurable functions $\bar{V}:[0,\infty) \times \R^d \to \R^d $ and $\tilde{V}:[0,\infty) \times \R^d \times \R^d \to \R^d $ such that \[
    V(t,x,\mu)=\bar{V}(t,x)+ \langle \tilde{V}(t,x,\cdot ) ,  \mu \rangle
    \]
    for all $t \in [0,\infty)$,  $x \in \R^d$, $\mu \in \cP_2(\R^d )$.

    \item [(ii)] For all $t\geq 0$, $\mu \in \cP_2(\R^d )$, $x,y \in \R^d $, and some $J\geq \frac{d}{ 2 }+4$ one has that $\bar{V}(t,\cdot ) \in \Cf^{J}_b(\R^d )$, $\tilde{V}(t,\cdot ,\cdot ) \in \Cf^J(\R^d \times \R^d )$, and for every compact set $K \in \cP_2(\R^d )$ and $i \in [d]$
      \begin{align}
	&\sup\limits_{ t \in [0,T] }\|\bar V_i(t,\cdot )\|_{\Cf_b^J}+ \sup\limits_{ t \in [0,T] }\|\tilde{V}_i(t,\cdot ,\cdot )\|_{\Cf_b^J \times H^J}+\sup\limits_{ t \in [0,T],\mu \in K }\|A_{i,i}(t,\cdot ,\mu)\|_{\Cf_b}<\infty,
      \end{align}
      where 
      \[
	\|f\|_{\Cf_b^m \times H^J}^2=\sum_{ |\alpha|\leq m } \sum_{ |\beta|\leq J } \sup\limits_{ x \in \R^d } \int_{ \R^d  }   \left( D^{\alpha}_xD^{\beta}_y f(x,y) \right)^2dy. 
      \]
  \end{enumerate}
\end{assumption}

Now we can formulate the main result of this section. 
\begin{theorem}[Quantified CLT for the stochastic mean field equation] 
  \label{the_clt}
  Let the coefficients $A,V,G$ of the equations~\eqref{equ_mean_field_equation_with_vanishing_noise},~\eqref{equ_mean_field_equation_without_noise} satisfy Assumptions~\ref{ass_basic_assumption},~\ref{ass_lipshitz_continuity},~\ref{ass_compact_support} for some $J\geq \frac{d}{ 2 }+4$. Let $\mu_t^{\eps}$, $t\geq 0$, be a superposition solution to the stochastic mean-field equation~\eqref{equ_mean_field_equation_with_vanishing_noise} started from $\mu_0^{\eps}\in\cP_2(\R^d)$ for each $\eps>0$, $\mu_t$, $t\geq 0$, be a superposition solution to the PDE~\eqref{equ_mean_field_equation_without_noise} started from $\mu_0^0 \in \cP_2(\R^d )$, and the continuous Gaussian process $\eta_t^{0}$, $t\geq 0$, be a solution to the linear SPDE~\eqref{equ_equation_for_eta0} started from $\varrho \in H^{-J+1}(\R^d )$. Furthermore, we assume that $\langle \phi_4 , \mu_0^{\eps} \rangle<\infty$ for all $\eps \in (0,1]$, and define the continuous process $\eta^{\eps}= \frac{1}{ \sqrt{ \eps } }\left( \mu_t^{\eps}-\mu_t^0 \right)$, $t\geq 0$. Then for every $T>0$ there exists a constant $C>0$ such that the inequality
   \begin{equation} 
  \label{equ__estimate_of_eta_eps_minus_eta_0}
    \e \sup\limits_{ t \in [0,T] }\|\eta_t^{\eps}-\eta_t^0\|^2_{-J}\leq C\eps \left(1+\langle \phi_4 , \mu_0^{\eps} \rangle\right)+C \W_2^2(\mu_0^{\eps},\mu_0^0)+C \|\eta^{\eps}_0 -\varrho\|_{-J}^2
  \end{equation}
  holds.
\end{theorem}
\begin{proof}
  The existence of a unique solution $\eta^0_t$, $t \in [0,T]$, to the linear SPDE~\eqref{equ_equation_for_eta0} follows from Proposition~\ref{pro_well_posedness_of_the_equation_for_eta0}. Moreover, it is a.s. a continuous $H^{-J+1}(\R^d )$-valued process such that 
  \[
    \e\sup\limits_{ t \in [0,T] }\|\eta_t^0\|_{-J+1}^2<\infty.
  \]
  We also remark that $\eta_t^{\eps}$, $t \in [0,T]$, is a.s. a continuous $H^{-J+1}(\R^d)$-valued process for every $\eps>0$.

  Define the functions $Q=U+\tilde U:[0,T] \times H^{-J+1}(\R^d) \times \cP_2(\R^d ) \to H^{-J}(\R^d)$, $R:[0,T] \times \cP_2(\R^d ) \to H^{-J}(\R^d )$, and $B: [0,T] \times \cP_2(\R^d) \to \cL_2(L_2(\Theta,\m),H^{-J}(\R^d)) $ as follows 
  \begin{align}
    \langle U(t,f,\mu),\varphi \rangle_{0}&:= \langle \nabla \varphi \cdot V(t, \cdot ,\mu) , f \rangle_{0},\\ 
    \langle \tilde U(t,f,\mu),\varphi \rangle_{0}&:= \left\langle \nabla \varphi(x) \cdot \langle \tilde{V}(t,x,\cdot ) , f \rangle_{0} , \mu(dx) \right\rangle\\
    \langle R(t,\mu) , \varphi \rangle_0&:=\langle D^2 \varphi:A(t,\cdot,\mu) , \mu \rangle
  \end{align}
  and 
  \[
    \langle B(t)h, \varphi \rangle_{0}=\left\langle h , \langle \nabla \varphi(x) \cdot G(t,x ,\mu,\cdot ) , \mu(dx) \rangle \right\rangle_{\m}
  \]
  for all $\varphi \in \Cf_c^{\infty}(\R^d )$, $f \in H^{-J+1}(\R^d)$, and $h \in L_2(\Theta,\m)$. 
  Then 
  \begin{align}
    \eta^{\eps}_t&= \eta_0^{\eps}+\frac{\sqrt{ \eps }}{2}\int_{ 0 }^{ t } R(s,\mu_s^{\eps})ds+ \int_{ 0 }^{ t } Q(s,\eta^{\eps}_s,\mu_s^{\eps}) ds+ \int_{ 0 }^{ t } B(s,\mu_s^{\eps})dW_s, \quad t \in [0,T],   
  \end{align}
  for all $\eps\geq 0$. 

  Let $\zeta^{\eps}=\eta_t^{\eps}-\eta_t^0$, $t \in [0,T]$. Then
  \begin{align}
    \zeta^{\eps}_t&= \zeta_0^{\eps}+\frac{\sqrt{ \eps }}{2}\int_{ 0 }^{ t } R(s,\mu_s^{\eps})ds+ \int_{ 0 }^{ t } \left(Q(s,\eta^{\eps}_s,\mu_s^{\eps})-Q(s, \eta^0_s,\mu_s^0)\right) ds\\
    &+ \int_{ 0 }^{ t } \left(B(s,\mu_s^{\eps})-B(s,\mu_s^0)\right)dW_s, \quad t \in [0,T].
  \end{align}
  
Using \Itos formula from \cite[Theorem~4.2.4]{PR07}, we get 
  \begin{align}
    \|\zeta_t^{\eps}\|_{-J}^2&= \|\zeta_0^{\eps}\|_{-J}^2+ \sqrt{ \eps }\int_{ 0 }^{ t } \langle R(s,\mu_s^{\eps}) , \zeta^{\eps}_s \rangle_{-J}ds +2\int_{ 0 }^{ t } \langle Q(s,\eta_s^{\eps},\mu_s^{\eps})-Q(s,\eta_s^0,\mu_s^0) , \zeta^{\eps}_s \rangle_{-J}ds  \\
    &+ \int_{ 0 }^{ t } \|B(s,\mu_s^{\eps})-B(s,\mu_s^0)\|_{\mathrm{HS},H^{-J}}^2ds + \int_{ 0 }^{ t } \langle \zeta_s^{\eps} , (B(s,\mu_s^{\eps})-B(s,\mu_s^0))dW_s \rangle_{-J} 
  \end{align}
  for all $t \in [0,T]$. We fix an $(\F_t)$-stopping time $\tau$ and apply the Burkholder--Davis--Gundy inequality to estimate $\e \sup\limits_{ s \in [0,\tau] }\|\zeta_s^{\eps}\|_{-J}^2=\e \sup\limits_{ s \in [0,t] }\|\zeta_{s\wedge \tau}^{\eps}\|_{-J}^2$. We get
  \begin{align}
    \e \sup\limits_{ s \in [0,t] }\|\zeta_{s\wedge\tau}^{\eps}\|_{-J}^2&\leq \|\zeta_0^{\eps}\|_{-J}^2+ \sqrt{ \eps }\e \sup\limits_{ s \in [0,t] }\left|\int_{ 0 }^{ s\wedge\tau } \langle R(r,\mu_r^{\eps}) , \zeta_r^{\eps} \rangle_{-J}dr\right| \\
    &+2\e \sup\limits_{ s \in [0,t] }\left|\int_{ 0 }^{ s\wedge\tau }  \langle Q(r,\eta_r^{\eps},\mu_r^{\eps})-Q(r,\eta_r^0,\mu_r^0) , \zeta_r^{\eps} \rangle_{-J}dr\right|\\
    &+ \e \sup\limits_{ s \in [0,t] }\int_{ 0 }^{ s\wedge\tau } \|B(r,\mu_r^{\eps})-B(r,\mu_r^0)\|_{\mathrm{HS},H^{-J}}^2 dr\\
    &+ \e \sup\limits_{ s \in [0,t] }\left|\int_{ 0 }^{ s\wedge\tau } \langle \zeta_r^{\eps} , (B(r,\mu_r^{\eps})-B(r,\mu_r^0))dW_r \rangle_{-J}\right| \\
    &\leq \|\zeta_0^{\eps}\|_{-J}^2+ \sqrt{ \eps }\e \int_{ 0 }^{ t } \left|\langle R(s\wedge\tau,\mu_{s\wedge\tau}^{\eps}) , \zeta_{s\wedge\tau}^{\eps} \rangle_{-J}\right|ds \\
    &+2\e \int_{ 0 }^{ t }  \left|\langle Q(s\wedge\tau,\eta_{s\wedge\tau}^{\eps},\mu_{s\wedge\tau}^{\eps})-Q(s\wedge\tau,\eta_{s\wedge\tau}^0,\mu_{s\wedge\tau}^0) , \zeta_{s\wedge\tau}^{\eps} \rangle_{-J}\right|ds\\
    &+ \e \int_{ 0 }^{ t } \|B(s\wedge\tau,\mu_{s\wedge\tau}^{\eps})-B(s\wedge\tau,\mu_{s\wedge\tau}^0)\|_{\mathrm{HS},H^{-J}}^2 ds\\
    &+ C\e \left(\int_{ 0 }^{ t } \|\hat{B}^{\eps}(s\wedge\tau)\|_{\mathrm{HS},\R}^2ds \right)^{ \frac{1}{ 2 }}=: \sum_{ i=1 }^{ 5 } I_i(t),
  \end{align}
  where $\hat{B}^{\eps}:[0,T] \to \cL_2(L_2(\Theta,\m),\R )$ is defined as follows 
  \[
    \hat{B}^{\eps}(s)h:= \langle \zeta_s^{\eps} , B(s,\mu_s^{\eps})h-B(s,\mu_s^0)h \rangle_{-J}
  \]
  for all $s \in [0,T]$ and $h \in L_2(\Theta,\m)$, and $C$ is the universal constant from the Burkholder--Davis--Gundy inequality. We next estimate every term in the inequality above.
  
  {\it Estimate of $I_2$.} Using the Cauchy--Schwarz inequality and Young's inequality, we get
  \begin{align}
    \sqrt{ \eps }|\langle R(s,\mu_s^{\eps}) , \zeta_s^{\eps} \rangle_{-J}|&\leq \sqrt{ \eps }\|R(s,\mu_s^{\eps})\|_{-J}\|\zeta_s^{\eps}\|_{-J}\leq \frac{ \eps}{ 2 }\|R(s,\mu_s^{\eps})\|_{-J}^2+ \frac{1}{ 2 }\|\zeta_s^{\eps}\|_{-J}^2.
  \end{align}
  Then, by the continuity of the embedding of $H^{J}(\R^d)$ into $\Cf_b^2(\R^d )$, Assumptions~\ref{ass_basic_assumption},~\ref{ass_lipshitz_continuity} and Remark~\ref{rem_linear_growth_of_coefficients_under_lipschitz_conditions}, we have 
  \begin{align}
    \|R(s,\mu_s^{\eps})\|_{-J}^2 &=  \sup\limits_{ \varphi \in \Cf_c^{\infty}(\R^d ) } \frac{1}{ \|\varphi\|_{J}^2 }\langle D^2 \varphi:A(s,\cdot ,\mu_s^{\eps}) , \mu_s^{\eps} \rangle^2\\
    &\leq \sup\limits_{ \varphi \in \Cf_c^{\infty}(\R^d) } \frac{ \|\varphi\|_{\Cf_b^2}^2 }{ \|\varphi\|_{J}^2 }\sup\limits_{t \in [0,T]}\langle |A(t,\cdot ,\mu_t^{\eps})|,\mu_t^{\eps}\rangle^2\\
    &\leq C \sup\limits_{ t \in [0,T]}\langle \||G(t,x,\mu_t^{\eps}, \cdot )|\|_\m^2 , \mu_t^{\eps}(dx) \rangle^2\\
    &\leq C \sup\limits_{ t \in [0,T]}\langle 1+|x|^2+\W_2^2(\mu_t^{\eps},\delta_0) , \mu_t^{\eps}(dx) \rangle^2\\
    &\leq C \left(1+4\sup\limits_{ t \in[0,T]  }\W_2^4(\mu_t^{\eps},\delta_0)\right).
  \end{align}
 Using Corollary~\ref{cor_moment_preserving_property} and the fact that $\langle \phi_4 , \mu_0^{\eps} \rangle<\infty$ for each $\eps \in (0,1]$, we get 
  \begin{align} 
    \e \sup\limits_{ t \in [0,T] }\W_2^4(\mu^{\eps}_t,\delta_0)=&\e \sup\limits_{ t \in [0,T] }\langle \phi_2 , \mu_t^{\eps} \rangle^2 \\\leq & \e \sup\limits_{ t \in [0,T] }\langle \phi_4 , \mu_t^{\eps} \rangle\\\leq& C(1+\langle \phi_4 , \mu_0^{\eps} \rangle)
  \end{align}
  for all $\eps \in (0,1]$. Hence, for every $\eps \in (0,1]$, we have
  \begin{align}
    I_2(t)&\leq C \eps(1+\langle \phi_4 , \mu_0^{\eps} \rangle)+C \int_{ 0 }^{ t } \|\zeta_{s\wedge\tau}^{\eps}\|_{-J}^2 \, ds\\
    &\leq C \eps(1+\langle \phi_4 , \mu_0^{\eps} \rangle)
    +C \int_{ 0 }^{ t } \sup\limits_{ r \in [0,s] }\|\zeta_{r\wedge\tau}^{\eps}\|_{-J}^2 \, ds \,
  \end{align}
  for all $t \in [0,T]$. 
  
 {\it Estimate of $I_3$.} We estimate $I_3$ as follows
  \begin{equation} 
  \label{equ_estimate_of_bar_u}
    \begin{split}
      \left|\langle  U(s,\eta_s^{\eps},\mu_s^{\eps})- U(s,\eta_s^0,\mu_s^0) , \zeta_s^{\eps} \rangle_{-J}\right|&\leq \left|\langle  U(s,\eta_s^{\eps},\mu_s^{\eps})- U(s,\eta_s^{0},\mu_s^{\eps}) , \zeta_s^{\eps} \rangle_{-J}\right|\\
      &+ \left|\langle  U(s,\eta_s^{0},\mu_s^{\eps})- U(s,\eta_s^0,\mu_s^0) , \zeta_s^{\eps} \rangle_{-J}\right|.
    \end{split}
  \end{equation}
  By Lemma~\ref{lem_estimate_for_coersivety}, 
  \begin{align}
    \left|\langle  U(s,\eta_s^{\eps},\mu_s^{\eps})- U(s,\eta_s^{0},\mu_s^{\eps}) , \zeta_s^{\eps} \rangle_{-J}\right|&=   \left|\langle  U(s,\zeta_s^{\eps},\mu_s^\eps) , \zeta_s^{\eps} \rangle_{-J}\right|\\
    &\leq C \max\limits_{ i \in [d] }\sup\limits_{ t \in [0,T] }\|V_i(t,\cdot ,\mu_t^{\eps})\|_{\Cf_b^J} \|\zeta_s^{\eps}\|_{-J}^2.
  \end{align}
  We note that 
  \begin{align}
    \sup\limits_{ t \in [0,T] }\|V_i(t,\cdot ,\mu_t^{\eps})\|_{\Cf_b^J}&\leq \sup\limits_{ t \in [0,T] }\|\bar V_i(t,\cdot )\|_{\Cf_b^J}+\sup\limits_{ t \in [0,T] }\|\langle \tilde V_i(t,\cdot ,y) , \mu_t^{\eps}(dy) \rangle\|_{\Cf_b^J}\\
    &\leq \sup\limits_{ t \in [0,T] }\|\bar V_i(t,\cdot )\|_{\Cf_b^J}+\sup\limits_{ t \in [0,T],y \in \R^d  }\| \tilde V_i(t,\cdot ,y)\|_{\Cf_b^J}\leq C,
  \end{align}
  where $C$ is a non-random constant, according to Assumption~\ref{ass_compact_support} and the continuous embedding of $H^J(\R^d )$ into $\Cf_b(\R^d )$. In order to estimate the second term of the right hand side of~\eqref{equ_estimate_of_bar_u}, we consider the following bilinear map $\hat U:[0,T] \times H^{-J+1}(\R^d)\times H^{-J}(\R^d) \to H^{-J}(\R^d)$ defined by
  \begin{equation}\label{equ_hat_U}
    \langle \hat U(t,f_1,f_2) , \varphi \rangle_0:= \left\langle \nabla \varphi(x) \cdot \langle \tilde{V}(t,x,\cdot ) , f_1 \rangle_0 , f_2(dx) \right\rangle_0
  \end{equation}
  for all $\varphi \in \Cf_c^{\infty}(\R^d)$, which coincides with $\tilde U$ on $[0,T] \times H^{-J+1}(\R^d)\times \cP_2(\R^d)$. Then, using Lemma~\ref{lem_estimate_for_coersivety}, we obtain
  \begin{equation}\label{equ_estimate_of_bar_U_1}
  \begin{split}
    \big|\langle  U(s,\eta_s^{0},\mu_s^{\eps})&- U(s,\eta_s^0,\mu_s^0) , \zeta_s^{\eps} \rangle_{-J}\big|=  \sqrt{ \eps }|\langle \hat U(s,\eta_s^0,\eta_s^{\eps}) , \zeta_s^{\eps} \rangle_{-J}|\\
    &\leq \sqrt{ \eps }|\langle \hat U(s,\eta_s^0,\zeta_s^{\eps}) , \zeta_s^{\eps} \rangle_{-J}|+\sqrt{ \eps }|\langle \hat U(s,\eta_s^0,\eta_s^0) , \zeta_s^{\eps} \rangle_{-J}|\\
    &\leq \sqrt{ \eps }C\max\limits_{ i \in [d] }\|\langle \tilde{V}_i(s,\cdot ,x) , \eta_s^0(dx) \rangle_0\|_{\Cf_b^J}\|\zeta_s^{\eps}\|_{-J}^2\\
    &+\sqrt{ \eps }\|\hat U(s,\eta_s^0,\eta_s^0)\|_{-J}\|\zeta_s^{\eps}\|_{-J}.
  \end{split}
  \end{equation}
  We estimate separately 
  \begin{align}
    \|\langle \tilde{V}_i(s,\cdot ,x) , \eta_s^0(dx) \rangle\|_{\Cf_b^J}\leq \|\tilde{V}_i(s,\cdot ,\cdot )\|_{\Cf_b^J \times H^J}\|\eta_s^0\|_{-J}
  \end{align}
  and 
  \begin{align}
    \|\hat U(s,\eta_s^0,\eta_s^0)\|_{-J}&= \sup\limits_{ \varphi \in \Cf_c^{\infty}(\R^d) } \frac{1}{ \|\varphi\|_J }\langle \nabla \varphi(x) \cdot \langle \tilde{V}(t,x,\cdot ) , \eta_s^0 \rangle_0 , \eta_s^0(dx) \rangle_0\\
    &\leq \sup\limits_{ \varphi \in \Cf_c^{\infty}(\R^d)} \frac{1}{ \|\varphi\|_J } \|\nabla \varphi\|_{J-1} \|\tilde{V}(s,\cdot ,\cdot )\|_{\Cf_b^{J-1}\times H^{J-1}} \|\eta_s^0\|_{-J+1}^2\\
    &\leq C \sup\limits_{ t \in [0,T] }\|\tilde{V}(t,\cdot ,\cdot )\|_{\Cf_b^{J-1}\times H^{J-1}} \|\eta_s^0\|_{-J+1}^2.
  \end{align}
  Therefore, by Assumption~\eqref{ass_compact_support} and Young's inequality,
  \begin{equation}\label{equ_estimate_of_bar_U_2}
  \begin{split}
    \big|\langle  U(s,\eta_s^{0},\mu_s^{\eps})&- U(s,\eta_s^0,\mu_s^0) , \zeta_s^{\eps} \rangle_{-J}\big|\leq \sqrt{ \eps }C \|\eta_s^0\|_{-J}\|\zeta_s^{\eps}\|_{-J}^2+ \sqrt{ \eps }C \|\eta_s^0\|^2_{-J+1}\|\zeta_s^{\eps}\|_{-J}\\
    &\leq \sqrt{ \eps }C \|\eta_s^0\|_{-J}\|\zeta_s^{\eps}\|_{-J}^2+ C \eps \|\eta_s^0\|^4_{-J+1}+ C \|\zeta_s^{\eps}\|_{-J}^2.
  \end{split}
  \end{equation}

We next estimate
  \begin{align}
    |\langle \tilde U(s,\eta_s^{\eps},\mu_s^{\eps})&-\tilde{U}(s,\eta_s^0,\mu_s^0) , \zeta_s^{\eps} \rangle_{-J}|\\
    &\leq|\langle \tilde U(s,\eta_s^{\eps},\mu_s^{\eps})-\tilde{U}(s,\eta_s^0,\mu_s^{\eps}) , \zeta_s^{\eps} \rangle_{-J}|
    +|\langle \tilde U(s,\eta_s^0,\mu_s^{\eps})-\tilde{U}(s,\eta_s^0,\mu_s^0) , \zeta_s^{\eps} \rangle_{-J}|\\
     &=|\langle \tilde U(s,\zeta_s^{\eps},\mu_s^{\eps}),\zeta_s^{\eps}\rangle|+\sqrt{\eps}|\langle \hat U(s,\eta_s^0,\eta_s^{\eps}) , \zeta_s^{\eps} \rangle_{-J}|
    \end{align}
    where $\hat U$ is defined by~\eqref{equ_hat_U}.
    Then,
    \begin{align}
    |\langle \tilde U(s,\zeta_s^{\eps},\mu_s^{\eps}),\zeta_s^{\eps}\rangle|&\leq\|\tilde U(s,\zeta_s^{\eps},\mu_s^{\eps})\|_{-J}\|\zeta_s^{\eps}\|_{-J}\\
    & \leq \|\zeta_s^{\eps}\|_{-J} \sup\limits_{ \varphi \in \Cf_c^{\infty}(\R^d) } \frac{1}{ \|\varphi\|_J }\left\langle \nabla \varphi \cdot \langle \tilde{V}(s,x, \cdot ) , \zeta_s^{\eps} \rangle_0 , \mu_s^{\eps}(dx) \right\rangle\\
    &\leq \|\zeta_s^{\eps}\|_{-J}^2 \sup\limits_{ \varphi \in \Cf_c^{\infty}(\R^d) } \frac{ \|\varphi\|_{\Cf_b^1} }{ \|\varphi\|_J }\|\tilde{V}(s,\cdot ,\cdot )\|_{\Cf_b \times H^{J}}\\& \leq C \sup\limits_{ t \in [0,T] } \|\tilde{V}(s,\cdot ,\cdot )\|_{\Cf_b \times H^J} \|\zeta_s^{\eps}\|^2_{-J}.
  \end{align}
  Using the first equality in the estimate~\eqref{equ_estimate_of_bar_U_1} and then~\eqref{equ_estimate_of_bar_U_2}, we can conclude that
  \begin{align}
      \sqrt{\eps}|\langle \hat U(s,\eta_s^0,\eta_s^{\eps}) , \zeta_s^{\eps} \rangle_{-J}|\leq \sqrt{ \eps }C \|\eta_s^0\|_{-J}\|\zeta_s^{\eps}\|_{-J}^2+ C \eps \|\eta_s^0\|^4_{-J+1}+ C \|\zeta_s^{\eps}\|_{-J}^2.
  \end{align}
  Thus, we have shown that there exists a constant $C>0$ such that
  \begin{align}
    I_3(t)&\leq C \int_{ 0 }^{ t } \e \|\zeta_{s\wedge\tau}^{\eps}\|_{-J}^2ds+ C \sqrt{ \eps }\int_{ 0 }^{ t }\e \|\eta_{s\wedge\tau}^0\|_{-J}\|\zeta_{s\wedge\tau}^{\eps}\|_{-J}^2ds + C\eps\int_{ 0 }^{ t } \e \|\eta_{s\wedge\tau}^0\|_{-J+1}^{4}ds\\
    &\leq C \int_{ 0 }^{ t } \e \sup\limits_{ r \in [0,s] }\|\zeta_{r\wedge\tau}^{\eps}\|_{-J}^2ds+ C \sqrt{ \eps }\int_{ 0 }^{ t }\E{ \|\eta_{s\wedge\tau}^0\|_{-J} \sup\limits_{ r \in [0,s] }\|\zeta_{r\wedge\tau}^{\eps}\|_{-J}^2}ds\\
    &+ C\eps \e \sup\limits_{ s \in [0,T] }\|\eta_{s}^0\|_{-J+1}^{4}
  \end{align}
  for all $t \in [0,T]$. 
  
  {\it Estimate of $I_4$.} Let $\{ h_k,\ k\geq 1 \}$ be an orthonormal basis in $L_2(\Theta,\m)$ and $\bar G^k(t,x,\mu)=\langle G(t,x,\mu,\cdot ) , h_k \rangle$, $k\geq 1$. Then 
  \begin{align} 
  \allowdisplaybreaks
      &\|B(s,\mu_s^{\eps})-B(s,\mu_s^0)\|_{\mathrm{HS},H^{-J}}^2= \sum_{ k=1 }^{ \infty } \|B(s,\mu_s^{\eps})h_k-B(s,\mu_s^0)h_k\|_{-J}^2\\
      &\qquad\qquad= \sum_{ k=1 }^{ \infty } \sup\limits_{ \varphi \in \Cf_c^{\infty}(\R^d) } \frac{1}{ \|\varphi\|_{J}^2 }\left( \langle \nabla \varphi \cdot \bar G^k(s,\cdot ,\mu_s^{\eps}) , \mu_s^{\eps} \rangle-\langle \nabla \varphi \cdot \bar G^k(s,\cdot ,\mu_s^0) , \mu_s^0 \rangle \right)^2\\
      &\qquad\qquad \leq  2\sum_{ k=1 }^{ \infty } \sup\limits_{ \varphi \in \Cf_c^{\infty}(\R^d) } \frac{1}{ \|\varphi\|_{J}^2 }\left( \langle \nabla \varphi \cdot \bar G^k(s,\cdot ,\mu_s^{\eps}) , \mu_s^{\eps} \rangle-\langle \nabla \varphi \cdot \bar G^k(s,\cdot ,\mu_s^0) , \mu_s^{\eps} \rangle \right)^2 \label{equ_estimate_of_i4} \\   
      &\qquad\qquad+ 2\sum_{ k=1 }^{ \infty } \sup\limits_{ \varphi \in \Cf_c^{\infty}(\R^d) } \frac{1}{ \|\varphi\|_{J}^2 }\left( \langle \nabla \varphi \cdot \bar G^k(s,\cdot ,\mu_s^0) , \mu_s^{\eps} \rangle-\langle \nabla \varphi \cdot \bar G^k(s,\cdot ,\mu_s^0) , \mu_s^0 \rangle \right)^2\\
      &\qquad\qquad =  2\sum_{ k=1 }^{ \infty } \sup\limits_{ \varphi \in \Cf_c^{\infty}(\R^d) } \frac{1}{ \|\varphi\|_{J}^2 }\left\langle \nabla \varphi \cdot (\bar G^k(s,\cdot ,\mu_s^{\eps}) - \bar G^k(s,\cdot ,\mu_s^0)) , \mu_s^{\eps} \right\rangle^2\\
      &\qquad\qquad+ 2\sum_{ k=1 }^{ \infty } \sup\limits_{ \varphi \in \Cf_c^{\infty}(\R^d) } \frac{1}{ \|\varphi\|_{J}^2 }\left\langle \nabla \varphi \cdot \bar G^k(s,\cdot ,\mu_s^0) , \mu_s^{\eps}-\mu_s^0 \right\rangle^2.
  \end{align}
  Using Jensen's inequality, Fubini's theorem, Parseval's identity, and Assumption~\ref{ass_lipshitz_continuity}, we get 
  \begin{align}
    &\sum_{ k=1 }^{ \infty } \sup\limits_{ \varphi \in \Cf_c^{\infty}(\R^d) } \frac{1}{ \|\varphi\|_{J}^2 }\left\langle \nabla \varphi \cdot (\bar G^k(s,\cdot ,\mu_s^{\eps}) - \bar G^k(s,\cdot ,\mu_s^0)) , \mu_s^{\eps} \right\rangle^2\\
    &\qquad\qquad \leq \sum_{ k=1 }^{ \infty } \sup\limits_{ \varphi \in \Cf_c^{\infty}(\R^d) } \frac{1}{ \|\varphi\|_{J}^2 }\|\varphi\|_{\Cf_b^1}^2\left\langle |\bar G^k(s,\cdot ,\mu_s^{\eps})-\bar G^k(s,\cdot,\mu_s^0)|^2 , \mu_s^{\eps} \right\rangle\\
    &\qquad\qquad\leq  C\left\langle \sum_{ k=1 }^{ \infty } |\bar G^k(s,\cdot ,\mu_s^{\eps})-\bar G^k(s,\cdot,\mu_s^0)|^2 , \mu_s^{\eps} \right\rangle\\
    &\qquad\qquad=C \sum_{ i=1 }^{ d } \left\langle \|G_i(s,\cdot ,\mu_s^{\eps})-G_i(s,\cdot ,\mu_s^0)\|_{\m}^2 , \mu_s^{\eps} \right\rangle\leq CL \W_2^2(\mu_s^{\eps},\mu_s^0).
    \allowdisplaybreaks
  \end{align}

  In order to estimate the second term in the right hand side of~\eqref{equ_estimate_of_i4}, we use the fact that $\mu_t^{\eps}$, $t\geq 0$, is a superposition solution for every $\eps\geq 0$. Therefore, for every $\eps\geq 0$ there exists a solution $X_{\eps}$ to the corresponding SDE with interaction~\eqref{equ_equation_with_interaction} such that $\mu_t^{\eps}=\mu_0^{\eps}\circ X_{\eps}(t,\cdot )$, $t\geq 0$. Let $\chi$ be an arbitrary probability measure on $\R^d \times \R^d $ with marginals $\mu_0^{\eps}$ and $\mu_0^0$. Then, by Jensen's inequality and the mean value theorem, we have
  \begin{align}
  \allowdisplaybreaks
    &\langle \nabla \varphi \cdot \bar G^k(s,\cdot ,\mu_s^0) , \mu_s^{\eps}-\mu_s^0 \rangle^2=  \bigg(\int_{ \R^d  }   \int_{ \R^d  }   \big( \nabla \varphi(X_{\eps}(u,s)) \cdot \bar G^k(s,X_{\eps}(u,s),\mu_s^0)\\
    &\hspace{65mm}-\nabla \varphi(X_0(v,s)) \cdot \bar G^k(s,X_0(v,s),\mu_s^0) \big)\chi(du,dv)  \bigg)^2\\
    &\qquad\qquad \leq 2 \int_{ \R^d  }   \int_{ \R^d  }   \left[ \nabla \varphi(X_{\eps}(u,s)) \cdot \left(\bar G^k(s,X_{\eps}(u,s),\mu_s^0)-\bar G^k(s,X_{0}(v,s),\mu_s^0)\right) \right]^2\chi (du,dv)  \\
    &\qquad\qquad + 2 \int_{ \R^d  }   \int_{ \R^d  }   \left[ \left(\nabla \varphi(X_{\eps}(u,s))-\nabla \varphi(X_0(v,s))\right) \cdot \bar G^k(s,X_{0}(v,s),\mu_s^0)\right]^2\chi (du,dv)  \\
    &\qquad\qquad \leq 2 \|\varphi\|_{\Cf_b^1}^2\int_{ \R^d  }   \int_{ \R^d  }    \left|\bar G^k(s,X_{\eps}(u,s),\mu_s^0)-\bar G^k(s,X_{0}(v,s),\mu_s^0)\right|^2\chi (du,dv)  \\
    &\qquad\qquad + 2 \|\varphi\|_{\Cf_b^2} \int_{ \R^d  }   \int_{ \R^d  }   |X_{\eps}(u,s)-X_0(v,s)|^2 |\bar G^k(s,X_{0}(v,s),\mu_s^0)|^2\chi (du,dv).
  \end{align}
  Thus, using the continuous embedding of $H^{J}(\R^d)$ into $\Cf_b^2(\R^d )$, Fubini's theorem, and Parseval's identity, the second term of the right hand side of~\eqref{equ_estimate_of_i4} can be estimated as follows
  \begin{align}
      &\sum_{ k=1 }^{ \infty } \sup\limits_{ \varphi \in \Cf_c^{\infty}(\R^d) } \frac{1}{ \|\varphi\|_{J}^2 }\langle \nabla \varphi \cdot \bar G^k(s,\cdot ,\mu_s^0) , \mu_s^{\eps}-\mu_s^0 \rangle^2\\
      &\qquad\qquad\leq C\sum_{ k=1 }^{ \infty } \int_{ \R^d  }   \int_{ \R^d  }    \left|\bar G^k(s,X_{\eps}(u,s),\mu_s^0)-\bar G^k(s,X_{0}(v,s),\mu_s^0)\right|^2\chi (du,dv)  \\
      &\qquad\qquad +C \sum_{ k=1 }^{ \infty } \int_{ \R^d  }   \int_{ \R^d  }   |X_{\eps}(u,s)-X_0(v,s)|^2 |\bar G^k(s,X_{0}(v,s),\mu_s^0)|^2\chi (du,dv) \\
      &\qquad\qquad =C \int_{ \R^d  }   \int_{ \R^d  }   \left|\|G(s,X_{\eps}(u,s),\mu_s^0)-G(s,X_0(u,s),\mu_s^0)\|_{\m}\right|^2\chi (du,dv)\\
      &\qquad\qquad +C \int_{ \R^d  }   \int_{ \R^d  }   |X_{\eps}(u,s)-X_0(v,s)|^2 \left|\|G(s,X_0(v,s),\mu_s^0)\|_{\m}\right|^2\chi(du,dv).
  \end{align}
  Using Assumption~\ref{ass_lipshitz_continuity} and the expression $\|G_i(t,x,\mu)\|_{\m}^2=A_{i,i}(t,x,\mu)$ for all $i \in [d]$, we can bound the expression above by 
  \begin{align}
    C\left(L+\sum_{ i=1 }^{ d } \|A_{i,i}(s,\cdot ,\mu_s^0)\|_{\Cf_b}^2\right) \int_{ \R^d  }   \int_{ \R^d  }   |X_{\eps}(u,s)-X_0(v,s)|^2 \chi(du,dv).
  \end{align}
  Using Assumption~\ref{ass_compact_support}, Theorem~\ref{the_lln}, Remark~\ref{rem_estimate_for_difference_of_x_eps_and_x_0}, and the estimate above, we get
  \begin{align}
    I_4(t)&\leq C \int_{ 0 }^{ t } \e\W_2^2(\mu_{s\wedge\tau}^{\eps},\mu_{s\wedge\tau}^0)ds+C \int_{ 0 }^{ t } \int_{ \R^d  }   \int_{ \R^d  }  \e |X_{\eps}(u,s\wedge\tau)-X_0(v,s\wedge\tau)|^2\chi(du,dv)ds\\
    &\leq CT\e \sup\limits_{ s \in [0,T] }\W_2^2(\mu_s^{\eps},\mu_s^0)+CT \int_{ \R^d  }   \int_{ \R^d  }  \e \sup\limits_{ s \in [0,T] } |X_{\eps}(u,s)-X_0(v,s)|^2\chi(du,dv)\\
    &\leq C\W_2^2(\mu_0^{\eps},\mu_0^0)+C\eps (1+\langle \phi_2 , \mu_0^{\eps} \rangle)+ C\int_{ \R^d  }   \int_{ \R^d  }   |u-v|^2\chi (du,dv).    
  \end{align}
  Taking infimum over all measures $\chi$ on $\R^d \times \R^d $ with marginals $\mu_0^{\eps}$ and $\mu_0^0$, the last term on the right hand side of the above inequality  will give $\W_2^2(\mu_0^{\eps},\mu_0^0)$.  Moreover,
  \[
  \langle \phi_2 , \mu_0^{\eps} \rangle\leq \langle \phi_4 , \mu_0^{\eps} \rangle^{\frac{1}{2}}\leq \frac{1}{2}(1+\langle \phi_4 , \mu_0^{\eps} \rangle) \, .
  \]
  Hence,
  \[
    I_4(t)\leq C\W_2^2(\mu_0^{\eps},\mu_0^0)+C\eps (1+\langle \phi_4 , \mu_0^{\eps} \rangle)
  \]
  for all $t \in [0,T]$. For the final term $I_5$, we have
  \begin{align}
    \|\hat{B}^{\eps}(s)\|_{\mathrm{HS},\R }^2=& \sum_{ k=1 }^{ \infty } \langle \zeta_s^{\eps} , B(s,\mu_s^{\eps})h_k-B(s,\mu_s^0)h_k \rangle_{-J}^2\\\leq& \|\zeta_s^{\eps}\|^2_{-J} \sum_{ k=1 }^{ \infty } \|B(s,\mu_s^{\eps})h_k-B(s,\mu_s^0)h_k\|_{-J}^2\\
    \leq& \|\zeta_s^{\eps}\|_{-J}^2 \|B(s,\mu_s^{\eps})-B(s,\mu_s^0)\|_{\mathrm{HS},H^{-J}}^2.
  \end{align}
  Therefore, for every $t \in [0,T]$
 \begin{align}
    I_5(t)&\leq C\e\left( \int_{ 0 }^{ t } \|\zeta_{s\wedge\tau}^{\eps}\|_{-J}^2 \|B(s\wedge\tau,\mu_{s\wedge\tau}^{\eps})-B(s\wedge\tau,\mu_{s\wedge\tau}^0)\|^2_{\mathrm{HS},H^{-J}} ds  \right)^{ \frac{1}{ 2 }}\\
    &\leq C\e\left[ \sup\limits_{ s \in [0,t] }\|\zeta_{s\wedge\tau}^{\eps}\|_{-J} \left( \int_{ 0 }^{ t } \|B(s\wedge\tau,\mu_{s\wedge\tau}^{\eps})-B(s\wedge\tau,\mu_{s\wedge\tau}^0)\|_{\mathrm{HS},H^{-J}}^2 ds  \right)^{ \frac{1}{ 2 }} \right]\\
    &\leq \frac{1}{ 2 }\e \sup\limits_{ s \in [0,t] }\|\zeta_{s\wedge\tau}^{\eps}\|_{-J}^2+ \frac{C^2}{ 2 } \int_{ 0 }^{ t } \e \|B(s\wedge\tau,\mu_{s\wedge\tau}^{\eps})-B(s\wedge\tau,\mu_{s\wedge\tau}^0)\|_{\mathrm{HS},H^{-J}}^2ds\\
    &= \frac{1}{ 2 }\e \sup\limits_{ s \in [0,t] }\|\zeta_{s\wedge\tau}^{\eps}\|_{-J}^2+\frac{ C^2 }{ 2 }I_4(t)\\
    &\leq \frac{1}{ 2 }\e \sup\limits_{ s \in [0,t] }\|\zeta_{s\wedge\tau}^{\eps}\|_{-J}^2+ C\W_2^2(\mu_0^{\eps},\mu_0^0)+C\eps (1+\langle \phi_4 , \mu_0^{\eps} \rangle).
  \end{align}

  We remark that $\eta^0_t$, $t \in [0,T]$, is a Gaussian random element in $\Cf([0,T],H^{-J+1}(\R^d))$ and $\e \sup\limits_{ t \in [0,T] }\|\eta_t\|_{-J+1}^2<\infty$, by Proposition~\ref{pro_well_posedness_of_the_equation_for_eta0}. Using \cite[Proposition~3.14]{Hairer:2009}, one can see that there exists a constant $C>0$ such that 
  \begin{equation} 
  \label{equ_estimate_of_the_fourth_moment_of_eta0}
    \e \sup\limits_{ t \in [0,T] }\|\eta_t^0\|^{4}_{-J+1}\leq C\left( \e \sup\limits_{ t \in [0,T] }\|\eta_t^0\|^2_{-J+1} \right)^2<\infty.
  \end{equation}
  Defining for each $\eps>0$ the stopping time 
  \[
    \tau^{\eps}=\inf\left\{ t\geq 0:\ \|\eta_t^0\|_{-J}\geq \frac{1}{ \sqrt{ \eps } } \right\}
  \]
  and combining the estimates of $I_i(t)$, $i \in [5]$, with $\tau=\tau^{\eps}$, we get
  \begin{align}
    \e \sup\limits_{ s \in [0,t] }\|\zeta_{s \wedge\tau^{\eps}}^{\eps}\|_{-J}^2&\leq C\eps \left(1+\langle \phi_4 , \mu_0^{\eps} \rangle\right)+C \W_2^2(\mu_0^{\eps},\mu_0^0)+C \|\eta^{\eps}_0 -\varrho\|_{-J}^2\\
    &+ C\int_{ 0 }^{ t } \e \sup\limits_{ r \in [0,s] }\|\zeta_{r\wedge\tau^{\eps}}^{\eps}\|_{-J}^2ds 
  \end{align}
  for all $t \in [0,T]$, where $C$ is independent of $\eps$. Using Gronwall's lemma, we obtain 
  \begin{equation} 
  \label{equ_estimate_of_stopping_zeta}
    \e \sup\limits_{ s \in [0,t] }\|\zeta_{s\wedge\tau^{\eps}}^{\eps}\|_{-J}^2\leq C^*(\eps)e^{TC}, \quad t \in [0,T],
  \end{equation}
  where 
  \[
  C^*(\eps):=C\eps \left(1+\langle \phi_4 , \mu_0^{\eps} \rangle\right)+C \W_2^2(\mu_0^{\eps},\mu_0^0)+C \|\eta^{\eps}_0 -\varrho\|_{-J}^2.
  \]
  
   Set
  \begin{align}
    \eta^*:&=\sup\limits_{ t \in [0,T] }\|\eta^0_t\|_{-J}
  \end{align}
  and note that 
  \[
    \sup\limits_{ t \in [0,T] }\|\eta^{\eps}_t\|_{-J}= \frac{1}{ \sqrt\eps  }\sup\limits_{ t \in [0,T] }\|\mu_t^{\eps}-\mu_t^0\|_{-J}\leq \frac{2}{ \sqrt\eps  }
  \]
  for all $\eps>0$. Thus,  by~\eqref{equ_estimate_of_the_fourth_moment_of_eta0},~\eqref{equ_estimate_of_stopping_zeta}, and Chebyshev's inequality, we obtain 
  \begin{align}
    \e \sup\limits_{ t \in [0,T] }\|\zeta_{t}^{\eps}\|_{-J}^2&= \E{\sup\limits_{ t \in [0,T] }\|\zeta_{t}^{\eps}\|_{-J}^2\I_{\left\{ \eta^*< \frac{1}{ \sqrt{ \eps } } \right\}}}+\E{ \sup\limits_{ t \in [0,T] }\|\zeta_{t}^{\eps}\|_{-J}^2\I_{\left\{ \eta^*\geq \frac{1}{ \sqrt{ \eps } } \right\}} }\\
    &\leq \e \sup\limits_{ t \in [0,T] }\|\zeta_{t\wedge\tau^{\eps}}^{\eps}\|_{-J}^2+\E{ \sup\limits_{ t \in [0,T] }\|\eta_t^{\eps}-\eta_t^0\|_{-J}^2\I_{\left\{ \eta^*\geq \frac{1}{ \sqrt{ \eps } } \right\}} }\\
    &\leq e^{CT}C^*(\eps)+ 4\E{ \sup\limits_{ t \in [0,T] }\|\eta_t^{\eps}\|_{-J}^2\I_{\left\{ \eta^*\geq \frac{1}{ \sqrt{ \eps } } \right\}} }+4\E{ \sup\limits_{ t \in [0,T] }\|\eta_t^0\|_{-J}^2\I_{\left\{ \eta^*\geq \frac{1}{ \sqrt{ \eps } } \right\}} }\\
    &\leq e^{CT}C^*(\eps)+ \frac{4}{ \eps } \e \I_{\left\{ \eta^*\geq \frac{1}{ \sqrt{ \eps } } \right\}}+4 \eps \e \sup\limits_{ t \in [0,T] }\|\eta_t^0\|_{-J}^4\\
    &\leq e^{CT}C^*(\eps)+ \left(\frac{ 4 \eps^2 }{ \eps }+4 \eps\right)\e \sup\limits_{ t \in [0,T] } \|\eta_t^0\|_{-J}^4\leq C (C^*(\eps)+\eps)
  \end{align}
  for all $\eps>0$. This completes the proof of the theorem.
\end{proof}

\section{Mean field limit and stochastic gradient descent}
\label{sec:mean_field_limit_and_stochastic_gradient_descent}

In this section, we consider the one-hidden layer neural network $f^M$ defined by~\eqref{eq:ML-shallow} for the approximation of a function $f: \R^{n_0} \to \R $, where $c_i \in \R $, $U_i$ is an $1\times n_0$-matrix, $\theta \in \R^{n_0}$, $x_i=(c_i,y_i) \in \R \times \R^{n_0}=:\R^d $, $b_i=0$ and $\phi \in \Cf_b^{\infty}(\R )$ is a fixed activation function. We also assume that $\theta$ is a random element in $\Theta:=\R^{n_0}$ with distribution $\m$. To minimize the risk function $L$ defined by~\eqref{eq:intro_Lemp}, the parameter $x=(x_i)_{i\in M}$ can be estimated using the stochastic gradient descent~\eqref{equ_SGD} with $P=1$ and the learning rate $\alpha=\frac{\beta}{M}$, where $\theta_k$, $k \in \N_0$, are i.i.d. random variables with distribution $\m$ and $x_i(0)$, $i \in [M]$, are i.i.d. random variables generated from a distribution $\mu_0$. We define the empirical distribution $\nu^{M,\frac \beta M}_t$, $t\geq 0$, of the network parameters by~\eqref{equ_empirical_process_for_sgd_intro} which is a random element in the Skorohod space $\Df([0,\infty),\cP(\R^d ))$ of all \cdl functions from $[0,\infty)$ to $\cP(\R^d )$ equipped with the Skorohod topology. 

We next consider the stochastic mean-field equation~\eqref{equ_mean_field_equation_with_vanishing_noise} with the coefficients $A,V,G$ defined by~\eqref{equ_definition_of_coefficients_intro} with $\Phi(x_i,\theta)$ replaced by $\beta\Phi(x_i,\theta)$. Using the assumptions on the activation function $\phi$, it is easy to see that $A,V,G$ satisfy Assumptions~\ref{ass_basic_assumption},~\ref{ass_lipshitz_continuity}. Thus, by Theorem~\ref{the_uniqueness_for_atomic_initial_conditions}, for every $M \in \N$ there exists a unique superposition solution $\mu^{\frac 1M}_t$, $t\geq 0$, to the stochastic mean-field equation~\eqref{equ_mean_field_equation_with_vanishing_noise} with $\eps= \frac{1}{ M }$ started from $\nu_0^{M,\frac \beta M}$.  

We will further consider $\mu_t^{\frac 1M}$, $t\geq 0$,  and $\nu^{M,\frac{\beta}{M}}_t$, $t\geq 0$, as random processes in the Sobolev space $H^{-J}(\R^d )$ with $J> \frac{d}{ 2 }$. Let $\law_{T,-J}(\rho)$ denote the distribution of a random element $\rho_t$, $t \in [0,T]$, in the Skorohod space $\Df([0,T],H^{-J}(\R^d ))$. The distribution of the marginals $\rho_t$ will be denoted by $\law_{-J}(\rho_t)$. The following theorem is the main result of this section.
 
\begin{theorem}\label{the_approximation_of_sgd} 
   Let the measure $\m$ and $\mu_0$ be compactly supported on $\Theta$ and $\R^d$, respectively, $f$ be bounded on the support of $\m$, and the processes $\mu_t^{\frac 1M}$, $t\geq 0$, $\nu_t^{M,\frac \beta M}$, $t\geq 0$, be defined as above. Then for every $J\geq \frac{ 3d }{ 2 }+7$, $T>0$ and $p \in [1,2)$
  \[
    \tilde\W_{p}\left(\law_{T,-J}(\mu^{\frac 1M}),\law_{T,-J}(\nu^{M,\frac \beta M})\right)=o(M^{- \frac{1}{ 2 }}),
  \]
  where $\tilde\W_p$ is the $p$-Wasserstein distance on the space $\Df([0,T],H^{-J}(\R^d ))$ that is equipped with the uniform norm.
  In particular, 
  \begin{equation} 
  \label{equ_assimptotic_for_sup_of_w_in_application}
  \sup\limits_{ t \in [0,T] }\W_p(\law_{-J}(\mu_t^{\frac 1M}),\law_{-J}(\nu_t^{M,\frac \beta M}))=o(M^{- \frac{1}{ 2 }})
  \end{equation}
  and for every $\varphi \in \Cf_c^{\infty}(\R^d )$
  \begin{equation} 
  \label{equ_assumptotic_for_e_of_phi_in_application}
    \sup\limits_{ t \in [0,T] }|\e\langle \varphi , \mu_t^{\frac 1M} \rangle-\e\langle \varphi , \nu_t^{M,\frac \beta M} \rangle|=o(M^{- \frac{1}{ 2 }}).
  \end{equation}
\end{theorem}

\begin{proof} 
  To simplify the notation, we will write $\mu_t^{M}$ and $\nu_t^M$ instead of $\mu_t^{\frac 1M}$ and $\nu_t^{M,\frac \beta M}$, respectively. Let $T>0$ and $J\geq \frac{ 3d }{ 2 }+7$ be fixed. Using~\cite[Theorem~1.2]{Si.Sp2020}, we get that the process $\nu^M_t$, $t \in [0,T]$, converges in $\Df([0,T],\cP(\R^d ))$ in distribution to the unique superposition solution $\mu^0_t$, $t\geq 0$, to the PDE~\eqref{equ_mean_field_equation_without_noise} started from $\mu_0$. We next consider the fluctuation field
  \[
    \xi_t^{M}=\sqrt{ M }(\nu_t^{M}-\mu^0_t), \quad t \in [0,T].
  \]
  According to~\cite[Lemma~4.3]{Si.Sp2020b}, there exists a rectangle $[-R,R]^d$ such that $\mu^0_t$, $\nu_t^M$ and $\xi_t^M$ are supported on it for each $t \in [0,T]$ and $M\geq 1$. We take $\Gamma=(-3 \sqrt{ d }R,3 \sqrt{ d }R)^d$. Then, by~\cite[Theorem~1.5]{Si.Sp2020b}, the process $\xi_t^{M}$, $t \in [0,T]$, converges in $\Df([0,T],H^{-J}(\Gamma))$ in distribution  to the continuous Gaussian process $\xi_t$, $t \in [0,T]$, in $H^{-J}(\Gamma)$ satisfying for every $\varphi \in \Cf_c^{\infty}(\Gamma)$
  \begin{equation} 
  \label{equ_equation_for_xi_in_application}
    \begin{split}
      \langle \varphi , \xi_t \rangle_{0,\Gamma}&= \langle \varphi , \xi_0 \rangle_{0,\Gamma}+\int_{ 0 }^{ t }\left(  \left\langle \nabla \varphi \cdot V(\cdot ,\mu^0_s), \xi_s \right\rangle_{0,\Gamma}+\left\langle \nabla \varphi \cdot \langle \tilde{V}(x,\cdot) , \xi_s \rangle_{0,\Gamma} , \mu^0_s(dx) \right\rangle\right)ds\\
      &+M_t^{\varphi}, \quad  t \in [0,T],
    \end{split}
  \end{equation}
  where $M_t^{\varphi}$, $t \in [0,T]$, is a mean-zero Gaussian process with variance 
  \begin{equation} 
  \label{equ_variation_of_xi_in_application}
    \Var(M_t^{\varphi})=\int_{ 0 }^{ t } \int_{ \Theta }   \langle \nabla \varphi \cdot G(\cdot ,\mu^0_s,\theta) , \mu^0_s \rangle^2 \m(d\theta)ds, \quad t \in [0,T].
  \end{equation}
  Moreover, by~\cite[Theorem~6.2]{Si.Sp2020b}, such a process $\xi_t$, $t \in [0,T]$, is unique. Using the Skorohod theorem~\cite[Theorem~3.1.8]{Ethier:1986}, we may assume that the process $\xi^{M}$ converges to $\xi$ in $\Df([0,T],H^{-J}(\Gamma))$ a.s. Since $\xi_t$, $t \in [0,T]$, is continuous in $H^{-J}(\Gamma)$, it is easy to see that 
  \[
    \sup\limits_{ t \in [0,T] }\|\xi_t^M-\xi_t\|_{-J,\Gamma} \to 0 \quad \mbox{a.s.}
  \]
  as $M \to \infty$. Finally, we remark that 
  \[
    \sup\limits_{ M \geq 1 }\e \sup\limits_{ t \in [0,T] }\|\xi_t^M\|_{-J,\Gamma}^2<\infty,
  \]
  according to~\cite[Lemma~4.8]{Si.Sp2020b}. Hence, by the de la Vall\'ee-Poussin theorem~\cite[Theorem~1.8]{Liptser:2001}, the sequence $\sup\limits_{ t \in [0,T] }\|\xi_t^M\|_{-J,\Gamma}^p$, $M\geq 1$, is uniformly integrable for every $p \in [1,2)$. Therefore,  
  \begin{equation} 
  \label{equ_convergence_of_wasserstein_distance_for_xi_in_application}
    \tilde\W_p^p\left(\law_{T,-J,\Gamma}(\xi^M),\law_{T,-J,\Gamma}(\xi)\right)\leq \e\sup\limits_{ t \in [0,T] }\|\xi_t^M-\xi_t\|_{-J,\Gamma}^p \to 0
  \end{equation}
  as $M\to\infty$ for every $p \in [1,2)$, where $\law_{T,-J,\Gamma}(\rho)$ is the law of a random element $\rho_t$, $t \in [0,T]$, in $\Df([0,T],H^{-J}(\Gamma))$.  

Since the initial condition $\mu^M_0$ is independent of the noise driving the equation~\eqref{equ_mean_field_equation_with_vanishing_noise}, we may assume that $\mu^M_0=\nu^M_0$, without loss of generality. Then
  \[
  \e\|\eta_0^M-\eta_0\|_{-J,\R^d}^p \to 0\quad \mbox{as}\ \ M\to\infty
  \]
  for every $p\in[1,2)$, where $\eta_0=\xi_0$, by~\eqref{equ_convergence_of_wasserstein_distance_for_xi_in_application}.  Consider the fluctuation field $\eta^M_t= \sqrt{ M }(\mu_t^M-\mu^0_t)$, $t \in [0,T]$, and the $\sigma$-field $\cS=\sigma(\eta^M_0,\eta_0)$. Note that, by Jensen's inequality and Theorem~\ref{the_clt}, for each $p\in[1,2)$ 
  \begin{equation} 
  \label{equ_convergence_of_etam_in_application}
  \begin{split}
    \e \sup\limits_{ t \in [0,T] }\|\eta_t^M-\eta_t\|_{-J,\R^d }^p&=\e\left[\e\left( \sup\limits_{ t \in [0,T] }\|\eta_t^M-\eta_t\|_{-J,\R^d }^p\Big|\cS\right)\right]\\
    &\leq \e\left[\e\left( \sup\limits_{ t \in [0,T] }\|\eta_t^M-\eta_t\|_{-J,\R^d }^2\Big|\cS\right)^{\frac{p}{2}}\right]\\
    &\leq\e\left[\left(\frac{C}{M} \left(1+\langle \phi_4 , \mu_0^{M} \rangle\right)+C \W_2^2(\mu_0^{M},\mu_0^0)+C \|\eta^{M}_0 -\eta_0\|_{-J,\R^d}^2\right)^{\frac{p}{2}}\right]\\
     &\leq\left(\frac{C}{M} \left(1+\e\langle \phi_4 , \mu_0^{M} \rangle\right)+C \e\W_2^2(\mu_0^{M},\mu_0^0)\right)^{\frac{p}{2}}+C^{\frac{p}{2}} \e\|\eta^{M}_0 -\eta_0\|_{-J,\R^d}^p.
    \end{split}
  \end{equation}
  Since $\mu_0$ has a compact support, we have that $\e\langle \phi_4 , \mu_0^{M} \rangle\leq \langle \phi_4 , \mu_0 \rangle<\infty$ for all $M$. Hence, by~\cite[Theorem~1]{Fournier:2015}, $\e\W_2^2(\mu_0^{M},\mu_0^0)\to 0$ as $M\to\infty$, and, consequently,
  \[
    \e \sup\limits_{ t \in [0,T] }\|\eta_t^M-\eta_t\|_{-J,\R^d }^p\to 0 \quad \mbox{as}\ \ M\to\infty.
  \]
Using the definition of the Wasserstein distance $\tilde\W_p$, we get 
\begin{equation} 
  \label{equ_convergence_of_wasserstein_distance_for_eta_application}
  \tilde{\W}_p\left( \law_{T,-J,\Gamma}(\eta^M),\law_{T,-J,\Gamma}(\eta) \right) \to 0 \quad \mbox{as}\ \ M\to\infty
\end{equation}
for every $p \in [1,2)$.

Let us make the following observations. Since $\mu^M_t$, $t \in [0,T]$, is a superposition solution to the stochastic mean-field equation~\eqref{equ_mean_field_equation_with_vanishing_noise} and $\mu_0$ is compactly supported, without loss of generality we may assume that the support of $\mu^M_t$ is contained in $\Gamma$ for each $M\geq 1$ and $t \in [0,T]$, otherwise, we may choose $\Gamma$ larger. Taking an arbitrary $\varphi \in \Cf_c^{\infty}(\Gamma)$, it is easily seen that $\langle \varphi , \eta_t \rangle$, $t \in [0,T]$, satisfies~\eqref{equ_equation_for_xi_in_application},~\eqref{equ_variation_of_xi_in_application}. Hence $\eta_t$, $t \in [0,T]$, is a continuous process in $H^{-J}(\Gamma)$ and has the same distribution as $\xi_t$, $t \in [0,T]$. We can see also that for every $p \in [1,2)$
\begin{align}
  \sqrt{ M^p }\tilde\W_p&\left( \law_{T,-J}(\mu_t^M),\law_{T,-J}(\nu_t^M) \right)^p\\
  &= \sqrt{ M^p }\inf\left\{ \e \sup\limits_{ t \in [0,T] }\|\tilde{\mu}_t^M-\tilde{\nu}_t^M\|^p_{-J,\R^d }:\ \tilde\mu^M\sim \law(\mu^M),\ \tilde\nu^M\sim \law (\nu^M) \right\}\\
  &= \sqrt{ M^p }\inf\left\{ \e \sup\limits_{ t \in [0,T] }\|\tilde{\mu}_t^M-\tilde{\nu}_t^M\|^p_{-J,\Gamma }:\ \tilde\mu^M\sim \law(\mu^M),\ \tilde\nu^M\sim \law (\nu^M) \right\}\\
  &= \inf\Bigg\{ \e \sup\limits_{ t \in [0,T] }\left\|\sqrt{ M }(\tilde{\mu}_t^M-\mu^0_t)-\sqrt{ M }(\tilde{\nu}_t^M-\mu^0_t)\right\|_{-J,\Gamma}^p:\\  
  &\qquad\qquad\qquad\qquad\qquad\qquad\tilde\mu^M\sim \law(\mu^M),\ \tilde\nu^M\sim \law (\nu^M)\Bigg\}\\
  &= \inf\left\{ \e \sup\limits_{ t \in [0,T] }\|\tilde\eta_t^M-\tilde{\xi}_t^M\|_{-J,\Gamma}^p:\ \ \tilde\eta^M\sim\law (\eta^M),\ \tilde{\xi}^M\sim\law(\xi^M) \right\}\\
  &= \tilde\W_p\left(\law_{T,-J,\Gamma}(\eta^M),\law_{T,-J,\Gamma}(\xi^M)\right).
\end{align}
Therefore, using the previous observations, we can estimate for every $p \in [1,2)$ 
\begin{align}
  \sqrt{ M }\tilde\W_p&\left(\law_{T,-J}(\mu^M),\law_{T,-J}(\nu^M)\right)= \tilde{\W}_p\left( \law_{T,-J,\Gamma}(\eta^M),\law_{T,-J,\Gamma}(\xi^M) \right)\\
  &\leq \tilde\W_p(\law_{T,-J,\Gamma}(\eta^M),\law_{T,-J,\Gamma}(\eta))+\tilde\W_p\left( \law_{T,-J,\Gamma}(\xi),\law_{T,-J,\Gamma}(\xi^M) \right)\\
  &\leq \tilde\W_p(\law_{T,-J}(\eta^M),\law_{T,-J}(\eta))+\tilde\W_p\left( \law_{T,-J,\Gamma}(\xi),\law_{T,-J,\Gamma}(\xi^M) \right).
\end{align}
Thus, 
\[
  \sqrt{ M }\tilde{\W}_p\left(\law_{T,-J}(\mu^M),\law_{T,-J}(\nu^M)\right) \to 0 \quad \mbox{as}\ \ M\to\infty,
\]
by~\eqref{equ_convergence_of_wasserstein_distance_for_xi_in_application} and~\eqref{equ_convergence_of_wasserstein_distance_for_eta_application}. This completes the proof of the first part of the theorem.

The asymptotics~\eqref{equ_assimptotic_for_sup_of_w_in_application},~\eqref{equ_assumptotic_for_e_of_phi_in_application} directly follow from the inequalities 
\[
  \sup\limits_{ t \in [0,T] }\W_p\left( \law_{-J}(\mu_t^M),\law_{-J}(\nu_t^M) \right)\leq \tilde{\W}_p\left( \law_{-J,T}(\mu^M),\law_{-J,T}(\nu^M) \right)
\]
and 
\[
  \left|\e \langle \varphi , \mu_t^M \rangle -\e \langle \varphi , \nu_t^M \rangle\right|\leq \|\varphi\|_{J}\W_p(\law_{-J}(\mu_t^M),\law_{-J}(\nu_t^M)),
\]
for every $\varphi \in \Cf_c^{\infty}(\R^d )$ and $t \in [0,T]$, respectively.
\end{proof}

\begin{remark} Note that the optimal rate in the quantified central limit Theorem~\ref{the_clt} implies that 
\[
\tilde\W_p(\law_{T,-J}(\eta^M),\law_{T,-J}(\eta)) = O(M^{-\frac 12}).
\]
Therefore, a quantified CLT for the SGD dynamics $\nu^M$ with the same rate of convergence would imply the stronger approximation error
\[
\tilde{\W}_p\left(\law_{T,-J}(\mu^M),\law_{T,-J}(\nu^M)\right)=O(M^{-1}).
\]
However, proving a quantified CLT for SGD is an open problem.
\end{remark}

\subsection*{Acknowledgements}
The first and third authors were supported by the Deutsche Forschungsgemeinschaft (DFG, German Research Foundation) – SFB 1283/2 2021 – 317210226.   BG acknowledges support by the Max Planck Society through the Research Group "Stochastic Analysis in the Sciences (SAiS)". The third author thanks the Max Planck Institute for Mathematics in the Sciences for its warm hospitality, where a part of this research was carried out.

\providecommand{\bysame}{\leavevmode\hbox to3em{\hrulefill}\thinspace}
\providecommand{\MR}{\relax\ifhmode\unskip\space\fi MR }
\providecommand{\MRhref}[2]{%
  \href{http://www.ams.org/mathscinet-getitem?mr=#1}{#2}
}
\providecommand{\href}[2]{#2}

\end{document}